\documentclass[10pt]{amsart}
\usepackage{egastylemodified}

\usepackage{amssymb,amsmath,amscd,amsthm,amsfonts,wasysym,mathrsfs,amsxtra}
\usepackage{stmaryrd,graphicx,array,url}
\usepackage[vcentermath]{genyoungtabtikz}
\usepackage{tikz}
\usepackage[section]{placeins}
\usepackage{afterpage}
\input xy
\xyoption{all}

\newcommand{\OO}{\mathcal{O}}

\newcommand{\image}{\mathrm{im}\,}
\newcommand{\kernel}{\mathrm{ker}\,}
\newcommand{\cokernel}{\mathrm{coker}\,}
\newcommand{\rank}{\mathrm{rk}\,}

\newcommand{\Hom}{\textnormal{Hom}}
\newcommand{\dimension}{\textnormal{dim}\,}

\newcommand{\rk}{\mathrm{rk}}

\newcommand{\CC}{\mathcal{C}}
\newcommand{\Ac}{\mathcal{A}}
\newcommand{\Fc}{\mathcal{F}}
\newcommand{\Tc}{\mathcal{T}}

\newcommand{\Coh}{\mathrm{Coh}}

\newcommand{\arinj}{\ar@{^{(}->}}
\newcommand{\arsurj}{\ar@{->>}}
\newcommand{\areq}{\ar@{=}}

\newcommand{\Bw}{\mathcal B_\omega}
\newcommand{\oZw}{\overline{Z}_\omega}

\newcommand{\bZob}{\overline{Z}_{\omega, B}}

\newcommand{\Fw}{\mathcal F_\omega}
\newcommand{\Tw}{\mathcal T_\omega}

\newcommand{\wh}{\widehat}

\newcommand{\ch}{\mathrm{ch}}
\newcommand{\bo}{{\bar{\omega}}}

\newcommand{\ysize}{\Yboxdim{8pt}}

\newcommand{\scalea}{\scalebox{0.5}}
\newcommand{\Bl}{\mathcal{B}^l}
\newcommand{\Fl}{\mathcal{F}^l}
\newcommand{\Tl}{\mathcal{T}^l}

\makeatletter
\newtheorem*{rep@theorem}{\rep@title}
\newcommand{\newreptheorem}[2]{%
\newenvironment{rep#1}[1]{%
 \def\rep@title{#2 \ref{##1}}%
 \begin{rep@theorem}}%
 {\end{rep@theorem}}}
\makeatother

\newreptheorem{theorem}{Theorem}

\newtheorem{theorem}{Theorem}[section]

\newtheorem{lemma}[theorem]{Lemma}

\newtheorem{proposition}[theorem]{Proposition}

\theoremstyle{definition}

\theoremstyle{remark}
\newtheorem{remark}[theorem]{Remark}

\numberwithin{equation}{section}

\begin{document}

\title{Fourier-Mukai transforms of slope stable torsion-free sheaves on a product  elliptic threefold}

\author[Jason Lo]{Jason Lo}
\address{Department of Mathematics \\
California State University Northridge\\
18111 Nordhoff Street\\
Northridge CA 91330 \\
USA}
\email{jason.lo@csun.edu}


\begin{abstract}
On the product elliptic threefold $X = C \times S$ where $C$ is an elliptic curve and $S$ is a K3 surface of Picard rank 1, we define a notion of limit tilt stability, which satisfies the Harder-Narasimhan property.  We   show that under the Fourier-Mukai transform $\Phi$ on $D^b(X)$ induced by the classical Fourier-Mukai transform on $D^b(C)$, a slope stable torsion-free sheaf satisfying a vanishing condition in codimension 2 (e.g.\ a reflexive sheaf) is taken to a limit tilt stable object.  We also show that a limit tilt semistable object on $X$ is taken by $\Phi$ to a slope semistable sheaf, up to modification by the transform of a codimension 2 sheaf.
\end{abstract}

\maketitle
\tableofcontents

\section{Introduction}

The question of whether stability for sheaves is preserved under a Fourier-Mukai transform is a long-standing question in algebraic geometry.  Answering questions of this type not only helps us understand the relations between moduli spaces (e.g.\ whether they are birational), but also helps us understand the relations among counting invariants associated to these moduli spaces.

Fourier-Mukai transforms are equivalences between derived categories of coherent sheaves on varieties that naturally arise when we study moduli problems for sheaves, particularly on abelian varieties, K3 surfaces, or fibrations of Calabi-Yau type, which include elliptic fibrations, abelian-surface fibrations and K3-surface fibrations.  Given  a Fourier-Mukai transform $\Phi : D^b(X) \overset{\thicksim}{\to} D^b(Y)$  between the derived categories of coherent sheaves on $X$ and $Y$, we can ask the following  concrete question:
\begin{itemize}
\item[] If $E$ is a slope stable sheaf on $X$, is $\Phi E$ a slope stable sheaf on $Y$? If not, what is a natural stability associated to $\Phi E$?
\end{itemize}
Of course, under a Fourier-Mukai transform $\Phi$, a sheaf $E$ on $X$ is not always taken to a sheaf.  Even when a slope stable sheaf $E$ is taken by $\Phi$ to a sheaf, there is no guarantee that $\Phi E$ is slope stable with respect to an arbitrary polarisation on $Y$ \cite{YosNote}.  The second part of the question above, therefore, is where the real question lies.  The above  question has been  discussed in  different contexts, including: slope stability, Gieseker stability and Bridgeland stabiliy on Abelian surfaces \cite{Mac1,YosAS,BriTh,YosPI,YosPII,MacMea}, slope stability and Gieseker stability on K3 surfaces \cite{BBR1, BBR2, Huy1,BruMac}, twisted stability on Abelian and K3 surfaces \cite{YosTw1,YosTw2,MYY}, rank-one torsion-free sheaves on elliptic surfaces \cite{YosPI,YosPII,FMTes}, rank-one torsion-free sheaves on elliptic threefolds \cite{BMef},  torsion sheaves on elliptic threefolds \cite{Dia15},  torsion sheaves on K3-surface fibrations \cite{ARG}, and from the point of view of Postnikov stability  \cite{HP}, just to name a few.  A comprehensive introduction to results of this type can be found in \cite{FMNT}.

In this article, we settle the above question on a product elliptic threefold `up to codimension two' in Theorem \ref{theorem2}.

\subsection{Main results}

Given an elliptic threefold $\pi : X \to S$ that is  Weierstra{\ss}  in the sense of \cite{FMNT} or  relatively minimal in the sense of \cite{BMef}, there always exists a dual elliptic fibration $\hat{\pi} : Y \to S$ where $X$ (resp.\ $Y$) is a moduli space of stable sheaves supported on the fibers of $\hat{\pi}$ (resp.\ $\pi$), and subsequently a Fourier-Mukai transform $\Phi : D^b(X) \to  D^b(Y)$.

On the other hand, on any smooth projective threefold $X$ with a fixed ample divisor $\omega$, we can always consider the notion of tilt stability (or $\nu_\omega$-stability) \cite{BMT1}, which is defined on objects in the Abelian subcategory of $D^b(X)$
\[
  \Bw = \langle \Fw [1], \Tw \rangle
\]
where
\begin{align*}
  \Tw &= \{ E \in \Coh (X): \text{ all $\mu_\omega$-HN factors of $E$ have $\mu_\omega > 0$}\}, \\
  \Fw &= \{ E \in \Coh (X): \text{ all $\mu_\omega$-HN factors of $E$ have $\mu_\omega \leq 0$}\},
\end{align*}
where $\mu_\omega$ is the usual slope function defined by $\mu_\omega (E) = \omega^2 \ch_1/\rk (E)$.

The category $\Bw$ is the heart of a t-structure on $D^b(X)$, and tilt stability on $\Bw$ is a key step in the proposed construction of Bridgeland stability conditions on arbitrary smooth projective threefolds by Bayer-Macr\`{i}-Toda \cite{BMT1}.

In this paper, we focus on the product elliptic threefold $X = C \times S$ where $C$ is an elliptic curve and $S$ is a K3 surface of Picard rank 1, considered as an elliptic fibration over $S$ via the second projection $\pi : X \to S$.  In this case, the Picard rank of $X$ is two, and ample $\mathbb{R}$-divisors on  $X$ are precisely of the form $\omega = tH + sD$ where $t, s>0$ \cite{LZ2}.  We can therefore identify the ample cone of $\mathrm{Amp}(X)_{\mathbb{R}}$ with the first quadrant of the plane $\mathbb{R}^2$.  For any fixed real number $\alpha >0$, we can then consider the branch of the hyperbola $ts = \alpha$ in the first quadrant.  As $s \to \infty$ along this hyperbola, \textit{a priori}, the heart $\Bw$ and hence the notion of $\nu_\omega$-stability will both vary.  However, we can define an appropriate `limit' of $\Bw$, denoted $\Bl$, which itself is  the heart of a t-structure on $D^b(X)$ (see Lemma \ref{lemma3}).  We then define a polynomial stability in the sense of Bayer \cite{BayerPBSC} on $\Bl$ (see Theorem \ref{theorem3}), which we call $\nu^l$-stability (or limit tilt stability).  There is also a Fourier-Mukai transform $\Phi : D^b(X) \to D^b(X)$ induced by the classical Fourier-Mukai transform $D^b(C) \to D^b(C)$ on $C$ \cite[Section 4]{LZ2}.  Putting these constructions together, we prove the main result of this article:

\begin{reptheorem}{theorem2}
Fix any real numbers $\lambda, \alpha >0$.  Let $\bo = \tfrac{\lambda}{\alpha} H + \lambda D$ and $\omega = tH + sD$ where $ts=\alpha$.
\begin{itemize}
\item[(A)] Suppose $E$ is a $\mu_{\bar{\omega}}$-stable torsion-free coherent sheaf on $X$ satisfying
    \begin{equation}
      \Hom (W_{0,X}\cap \Coh^{\leq 1}(X),E[1])=0.
    \end{equation}
    Then $\Phi E [1]$ is a $\nu^l$-stable object in $\Bl$.
\item[(B)] Suppose $F \in \Bl$ is a $\nu^l$-semistable object with $\ch_{10}(F) \neq 0$.  Consider the decomposition of $F$ in $\Bl$ with respect to the torsion triple \eqref{eq19}
    \[
    0 \to F' \to F \to F'' \to 0
    \]
    where $F' \in \langle \Fl [1], W_{0,X}\rangle$ and $F''\in W_{1,X} \cap \Tl$.  Then $\Phi (F')$ is a torsion-free $\mu_{\bo}$-semistable sheaf, while $\Phi (F'')[1]$ is a $\Phi$-WIT$_0$ sheaf in $\Coh^{\leq 1}(X)$.
\end{itemize}
\end{reptheorem}
In the theorem, we write $W_{i,X}$ to denote the category  of coherent sheaves that are $\Phi$-WIT$_i$, i.e.\ that are sent by $\Phi$ to coherent sheaves sitting at degree $i$ in $D^b(X)$.  The categories $\Tl, \Fl$ represent the appropriate limits of $\Tw, \Fw$ as $s \to \infty$ along the positive branch of the hyperbola $ts = \alpha$.  The notation $\ch_{10}$ is the component of the Chern character corresponding to fiber degree with respect to the fibration $\pi$, while $\Coh^{\leq 1}(X)$ is the category of coherent sheaves on $X$ supported in dimension at most 1.

We explain the intuition behind Theorem \ref{theorem2}:
\begin{itemize}
\item Part (A). Any torsion-free reflexive sheaf on a smooth projective threefold $X$ satisfies the vanishing condition in (A) by \cite[Lemma 4.20]{CL}.  On the other hand, any torsion-free sheaf $E$ on $X$ fits in a short exact sequence of coherent sheaves $0 \to E \to E^{\ast \ast} \to T \to 0$ where $E^{\ast \ast}$ is the double dual of $E$, hence reflexive, and $T$ is supported in dimension at most 1.  That is, $E$ fits in the exact triangle
     \[
     T[-1] \to E \to E^{\ast \ast} \to T
     \]
     in $D^b(X)$.  As a result, we can interpret part (A) as:  any slope stable torsion-free sheaf $E$ is sent by $\Phi$ to a limit tilt stable object in $\Bl$, if we modify $E$ in codimension 2.
\item Part (B).  The condition $\ch_{10}(F) \neq 0$ ensures that the transform of $F$ is (a coherent sheaf) of nonzero rank.  Since the component $F''$ of $F$ is sent by $\Phi$ into $\Coh^{\leq 1}(X)$, we can interpret part (B) as: any limit tilt semistable object $F$ in $\Bl$ with nonzero $\ch_{10}$ is sent by $\Phi$ to a slope semistable torsion-free sheaf, if we modify $\Phi F$ in codimension 2.
\end{itemize}
In summary,
\[
\xymatrix{
  \text{ slope stable } \ar@{=>}[rrr]^{\text{modify in codim. 2}}_{\text{then apply $\Phi$}} & & & \text{ limit tilt stable } \ar@{=>}[d] \\
  \text{ slope semistable } &&& \ar@{=>}[lll]^{\text{then modify in codim. 2}}_{\text{apply $\Phi$}} \text{ limit tilt semistable}
}
\]

The comparison of slope stability and limit tilt stability via the Fourier-Mukai functor $\Phi$ is reminiscent of the comparison between slope stability and Gieseker stability for coherent sheaves on a threefold:
\[
\xymatrix{
  \text{ slope stable } \ar@{=>}[r] &  \text{ Gieseker stable } \ar@{=>}[d] \\
  \text{ slope semistable } & \ar@{=>}[l] \text{ Gieseker semistable}
}
\]

\subsection{Idea of proofs of main results}

The key idea behind the proof of Theorem \ref{theorem2} is that, as $s \to \infty$ along the positive branch of the hyperbola $ts = \alpha$, the Chern components in the tilt function $\nu_\omega$ (where $\omega = tH + sD$) are dominated by terms that correspond, via the Fourier-Mukai transform $\Phi$, to the Chern components defining the slope function $\mu_{\bo}$.

On the other hand, to establish that limit tilt stability satisfies the Harder-Narasimhan (HN) property, we first construct a multitude of t-structures on $D^b(X)$ by identifying many torsion classes in the noetherian abelian category $\Coh (X)$.  Then, by taking the intersections and extension closures of the various torsion classes and torsion-free classes, we form a 4-step filtration of the heart $\Bl$, i.e.\ we give a torsion quadruple \eqref{eq35} in $\Bl$.  We then show that this 4-step filtration in $\Bl$ refines to an HN filtration with respect to $\nu^l$-semistability.

\subsection{Product elliptic threefold vs  general  Weierstra{\ss} threefold}

Many of the constructions in this article carry over directly to a general Weierstra{\ss} elliptic threefold, or with slight modifications.  For example, the definitions of the torsion classes in \eqref{eq46}, as they are, make sense on an arbitrary elliptic threefold.  A major reason why we chose to consider the product elliptic threefold in this article is, that the formula for the cohomological Fourier-Mukai transform is very simple (see Section \ref{section-matrixchern}) - Chern classes can be represented by 2 by 3 matrices, and the cohomological Fourier-Mukai simply swaps the two rows and then changes the signs in the second row.  On a general Weierstra{\ss} elliptic threefold $\pi : X \to S$, there would be contribution to the cohomological Fourier-Mukai from the canonical class of the base $S$ of the fibration, and so many computations may not be as clean as in the product case.

\subsection{Relations to other problems}

\subsubsection{Counting invariants}

In this article, we establish equivalences between various categories of coherent sheaves under the Fourier-Mukai transform.  When suitable notions of stability are paired with these categories, they can potentially help us understand various counting invariants better.  For instance, in the work of Oberdieck-Shen \cite{OS1}, in which they give a partial proof of the modularity conjecture on PT invariants due to Huang-Katz-Klemm, they study the transform of stable pairs in the sense of Pandharipande-Thomas under an autoequivalence.  This involves understanding the Fourier-Mukai transforms of rank-one torsion-free sheaves (which are, of course, slope-stable), while our Theorem \ref{theorem2}(A) describes the transforms of slope-stable torsion-free sheaves of any rank satisfying the vanishing condition \eqref{eq53}.

\subsubsection{Walls and moduli spaces}

In Section \ref{section-tiltvslimittilt}, we show that if an object $E \in D^b(X)$ lies in the heart $\Bw$ and is tilt stable for $s \gg 0$ along the hyperbola $ts = \alpha$, then $E$ is limit tilt stable.  A natural follow-up question is: given a limit tilt stable object $E \in \Bl$, can we find a fixed $s_0>0$ such that $E \in \Bw$ and $E$ is $\nu_\omega$-stable for all $s>s_0$ with $ts = \alpha$? (Note that we already know $E \in \Bw$ for $s>s_0$ for some $s_0>0$ from Remark \ref{remark14}(vi).)  In other words, we would like to know whether the `mini-walls' for tilt stability along the hyperbola $ts = \alpha$ is bounded from the right-hand side.  This is similar to the question considered in the author's joint work with Qin \cite{LQ}, in which we show the local finiteness and boundedness of mini-walls for Bridgeland stability on rays on surfaces.  (Bridgeland stability on surfaces are analogous to tilt stability on threefolds in some manners.)

If we can answer the question above and also show that $s_0$ depends only on the Chern classes of $E$, then together with Theorem \ref{theorem2}, we will obtain isomorphisms between moduli spaces of slope stable sheaves and moduli spaces of tilt stable objects - via the Fourier-Mukai transforms $\Phi$.  This will be in line with the observations in various other articles, where tilt stability of slope stable sheaves has been shown:   on arbitrary smooth projective threefolds (see \cite[Section 7.2]{BMT1} and \cite[Corollary 3.11, Example 4.4]{BMS}), and on $\mathbb{P}^3$ (see \cite{Sun} and \cite[Lemma 3.5]{Schmidt}).

In Lemma \ref{lemma31}, we at least show that $E \in \Bw$ for $s>s_0$, for some $s_0$ depending only on the Chern classes of $E$.

\subsubsection{Existence of Bridgeland stability}

In Maciocia-Piyaratne's approach to proving the existence of Bridgeland stability on Abelian threefolds $X$ of Picard rank 1 \cite{MacPir1, MacPir2}, a major part of the work lay in computing cohomology with respect to various t-structures.  For instance, given the heart $\Bw$ of a t-structure, it was necessary to compute the cohomology of objects in   $\Phi (\Bw)$ with respect to various t-structures, where  $\Phi : D^b(X) \to D^b(X)$ is a Fourier-Mukai transform on $D^b(X)$.  If one is to carry out a similar construction on elliptic threefolds, understanding the Fourier-Mukai transforms of slope-stable sheaves would constitute the first step.

\subsection{Outline of the paper}

In Section \ref{sec-generation}, we introduce a nested sequence of torsion classes in $\Coh (X)$ that filter $\Coh (X)$ itself, and introduce a notation for describing these categories that is particularly suited to the Fourier-Mukai transform $\Phi$.  In Section \ref{section-limittiltstab}, we construct the limit $\Bl$ of the hearts $\Bw$ as $\omega$ moves along a hyperbola in the first quadrant, and define the notion of limit tilt stability.  We also compute the phases of various objects with respect to the Laurent-polynomial-valued phase function of limit tilt stability (see Table \ref{table1}).  In Section \ref{section-slopevslimittilt}, we prove Theorem \ref{theorem2}.  Then, in Section \ref{section-HNlimittilt}, we establish the Harder-Narasimhan property of limit tilt stability.  And finally, in Section \ref{section-tiltvslimittilt}, we briefly discuss a connection between tilt stability and limit tilt stability.

\subsection{Acknowledgements}
The author would like to thank Ziyu Zhang, Wanmin Liu and Zhenbo Qin for many valuable discussions.  He would also like to thank the Center for Geometry and Physics of the Institute for Basic Science in Pohang, South Korea, for their hospitality during the author's stay in May 2016, where part of this work was completed.

\section{Preliminaries}

Throughout this paper, we will write $X$ to denote the product elliptic threefold $X = C \times S$, where $C$ is an elliptic curve and $S$ is a K3 surface of Picard rank 1.  We will regard $X$ as a trivial elliptic fibration via the second projection $\pi: X \to S$, and  assume  $\mathrm{Pic}(S)=\mathbb{Z}[H_S]$ where $H_S$ is an ample class on $S$ with  $H_S^2 = 2h$.

We will write $\Phi$ to denote the relative Fourier-Mukai transform $D^b(X) \overset{\thicksim}{\to} D^b(X)$ constructed in the author's joint work with Zhang \cite[(4.3)]{LZ2}.  Over any closed point $s \in S$, the relative Fourier-Mukai transform $\Phi_s : D^b(X_s) \overset{\thicksim}{\to} D^b(X_s)$ is the classical Fourier-Mukai transform on $D^b(X_s) \cong D^b(C)$.  In particular, for any closed point $x \in X$, $\Phi$ takes the structure sheaf $\mathcal{O}_x$ of $x$ to the degree zero line bundle on the fiber $X_{\pi (x)}$ parametrised by $x$.

If $\Ac$ is the heart of a t-structure on $D^b(X)$, then $\mathcal H^i_{\Ac}$ will denote the $i$-th cohomology functor with respect to this t-structure.  We will simply write $H^i$ to denote $\mathcal H^i$ when the t-structure is the standard t-structure.

We have the identity $\Phi^2 = \mathrm{id}_X [-1]$.  For any $E \in D^b(X)$, $\Phi^i (E)$ will denote $H^i (\Phi (E))$, the degree-$i$ cohomology of $\Phi (E)$ with respect to the standard t-structure.  For any coherent sheaf $E$ on $X$, we have $\Phi^i (E)=0$ for all $i \neq 0,1$.  A complex $E \in D^b(X)$ will be called $\Phi$-WIT$_i$ if $\Phi^j(E)=0$ for all $j \neq i$, i.e.\ if $\Phi (E) \cong \wh{E} [-i]$ for some coherent sheaf $\wh{E}$.  We will refer to $\wh{E}$ as the transform of $E$.

See \cite[Section 6.1]{LZ2} for properties of relative integral functors in general, and \cite[Section 4.1]{LZ2} for details of the construction of $\Phi$.

\subsection{Matrix notation for Chern characters}\label{section-matrixchern}  For convenience, we will follow the matrix notation in \cite[Section 4.2]{LZ2} for writing Chern characters of objects in $D^b(X)$.  If we write $e_i, f_j$ to denote the effective generators of $A^i(C), A^j(S)$ for $0 \leq i \leq 1, 0 \leq j \leq 2$, then $A^\ast (X)$ has the $d_{ij} := e_i \otimes f_j$ as an integral basis.  It follows that for any object $E \in D^b(X)$, we can write
\[
  \ch (E)  = \sum_{0\leq i \leq 1,\, 0 \leq j \leq 2} a_{ij} d_{ij} \text{\quad for some integers $a_{ij}$},
\]
and we abbreviate this as
\[
  \ch (E) = (a_{ij}) = \begin{pmatrix} a_{00} & a_{01} & a_{02} \\
  a_{10} & a_{11} & a_{12} \end{pmatrix}.
\]
That is,
\begin{align*}
  \ch_0(E) &= a_{00}, \\
  \ch_1(E) &= a_{10}d_{10} + a_{01}d_{01}, \\
  \ch_2(E) &= a_{11}d_{11} + a_{02}d_{02}, \\
  \ch_3(E) &= a_{12}.
\end{align*}
This matrix notation has the advantage that the cohomological Fourier-Mukai transform can be represented by matrix multiplication.  More concretely, for any $E \in D^b(X)$ with $\ch(E)=(a_{ij})$ we have
\[
  \ch (\Phi E) = \begin{pmatrix} a_{10} & a_{11} & a_{12} \\
  -a_{00} & -a_{01} & -a_{02} \end{pmatrix}
  = \begin{pmatrix} 0 & 1 \\
  -1 & 0 \end{pmatrix} \begin{pmatrix} a_{00} & a_{01} & a_{02} \\
  a_{10} & a_{11} & a_{12} \end{pmatrix}.
\]

 The intersection products between the $d_{ij}$ can be summarised in the multiplication table
\begin{center}
\begin{tabular}{| c | c | c | c | c |}
 \hline
 & $d_{01}$ & $d_{10}$ & $d_{02}$ & $d_{11}$ \\
 \hline
 $d_{01}$ & $2h$ & 1 & 0 & $2h$ \\
 \hline
 $d_{10}$ & & 0 & 1 & 0 \\
 \hline
 $d_{02}$ & & & 0 & 0 \\
 \hline
 $d_{11}$ & & & & 0 \\
 \hline
\end{tabular}
\end{center}
The entries refer to the coefficients of the appropriate generator.  That is, if the entry for the product $d_{ij}d_{i'j'}$ is $c$,  it means that $d_{ij}d_{i'j'}=cd_{i''j''}$ where $i''=i+i'$ and $j'' = j+j'$.

We will often write
\[
  H = d_{10}, \text{\qquad} D = d_{01}
\]
and write
\[
  \omega = tH + sD
\]
to denote the polarisation we fix on $X$, where $t, s>0$ are real numbers.  With $\ch (E)$ as above, we have the conversions
\begin{gather*}
  a_{00}=\ch_{00}(E)=\ch_0(E), \\
   2ha_{01}=HD\ch_{01}(E), \,\, a_{02}=H\ch_{02}(E), \\
  2ha_{10}= D^2\ch_{10}(E), \,\, 2ha_{11}=D\ch_{11}(E),\\
  a_{12}=\ch_{12}(E).
\end{gather*}
We will sometimes abuse notation and write $\ch_{ij} \geq 0$ (resp.\ $\ch_{ij} \leq 0$) to mean $a_{ij} \geq 0$ (resp.\ $a_{ij} \leq 0$).

The following are intersection products that we will need:

\begin{align*}
  \omega^2 &= 2tsd_{11}+2hs^2 d_{02} \\
  \omega^3 &= 6hts^2,
\end{align*}
and
\begin{align*}
  \omega^2 \ch_1(E) &= (2tsd_{11}+2hs^2d_{02})(a_{10}d_{10}+a_{01}d_{01}) \\
  &= 4htsa_{01} + 2hs^2a_{10} \\
  \omega \ch_2(E) &= (td_{10} + sd_{01})(a_{11}d_{11}+a_{02}d_{02})\\
  &= ta_{02} + 2hsa_{11} \\
  \ch_1(E)^2 &= (a_{10}d_{10}+a_{01}d_{01})^2 = 2a_{10}a_{01}d_{11} + 2ha_{01}^2d_{02} \\
  \omega \ch_1(E)^2 &= (td_{10} + sd_{01})(2a_{10}a_{01}d_{11} + 2ha_{01}^2d_{02})\\
  &= 2hta_{01}^2 + 4hsa_{01}a_{10}.
\end{align*}

\subsection{Slope-like functions}  Any time we have a noetherian abelian category $\Ac$ together with a pair of group homomorphisms $C_0 : K(\Ac) \to \mathbb{Z}$ and $C_1 : K(\Ac) \to \mathbb{R}$ satisfying the positivity properties
\begin{itemize}
\item $C_0(E) \geq 0$ for any $E \in \Ac$;
\item if $E \in \Ac$ satisfies $C_0(E)=0$, then $C_1 (E) \geq 0$,
\end{itemize}
we can define a function $\mu : K (\Ac) \to \mathbb{R} \cup \{\infty\}$ by setting
\[
  \mu (E) = \begin{cases} \frac{C_1(E)}{C_0(E)} &\text{ if $C_0(E) \neq 0$} \\
  +\infty &\text{ if $C_0(E)$=0} \end{cases}.
\]
We will refer to any such function $\mu$ as a slope-like function  (see \cite[Section 3.2]{LZ2}).  It gives a notion of semistability for objects in $\Ac$: an object $E \in \Ac$ is called $\mu$-semistable if $\mu (F) \leq \mu (E)$ for every nonzero proper subobject $F \subset E$ in $\Ac$.  In particular, $\mu$ satisfies the Harder-Narasimhan (HN) property, i.e.\ every object $E \in \Ac$ has a filtration
\[
  E_0 \subset E_1 \subset \cdots \subset E_m = E
\]
in $\Ac$ where each $E_i/E_{i-1}$ is $\mu$-semistable, and $\mu (E_0) > \mu (E_1/E_0) > \cdots > \mu (E_m/E_{m-1})$ \cite[Proposition 3.4]{LZ2}. We write $\mu_{max}(E)$ (resp.\ $\mu_{min}(E)$) to denote $\mu (E_0)$ (resp.\ $\mu (E_m/E_{m-1})$).

For any coherent sheaf $E$ on $X$ with $\ch_0(E)=0$, we have $\ch_{10}(E) \geq 0$ as well as $\ch_{01}(E) \geq 0$ \cite[Lemma 5.2]{LZ2}.  Hence we have the following slope-like functions on $\Coh (X)$
\begin{align*}
  \mu_f (E) &= \begin{cases} \frac{a_{10}}{a_{00}} &\text{ if $a_{00} \neq 0$}\\
  +\infty &\text{ if $a_{00} = 0$} \end{cases}, \\
    \mu^\ast (E) &= \begin{cases} \frac{a_{01}}{a_{00}} &\text{ if $a _{00} \neq 0$} \\ +\infty &\text{ if $a_{00}=0$} \end{cases},
\end{align*}
The classical slope function with respect to a fixed polarisation $\omega$
\[
  \mu_\omega (E) = \begin{cases} \frac{\omega^2 \ch_1(E)}{\ch_0(E)} &\text{ if $\ch_0(E) \neq 0$}\\
  +\infty &\text{ if $\ch_0(E) = 0$} \end{cases}
\]
is also a slope-like function.  When $\ch_0(E)=a_{00} \neq 0$, we have
\begin{equation}\label{eq6}
  \mu_\omega (E) = \frac{4htsa_{01} + 2hs^2a_{10}}{a_{00}} = 4hts\mu^\ast (E) + 2hs^2\mu_f (E).
\end{equation}

\subsection{Torsion $n$-tuples}  Given an abelian category $\Ac$, a torsion pair $(\mathcal T, \mathcal F)$ in $\Ac$ is a pair of full subcategories such that
\begin{itemize}
\item $\Hom_{\Ac}(T,F)=0$ for any $T \in \mathcal T, F \in \mathcal F$;
\item every object $E$ in $\Ac$ fits in a short exact sequence
\[
0 \to T \to E \to F \to 0
\]
where $T \in \mathcal T, F \in \mathcal F$.
\end{itemize}
That is, every object in $\Ac$ has a two-step filtration with respect to a torsion pair.  We refer to $\mathcal T$ (resp.\ $\mathcal F$) as the torsion class (resp.\ torsion-free class) of the torsion pair.

More generally, a torsion $n$-tuple $(\CC_1, \CC_2,\cdots,\CC_n)$ in $\Ac$  \cite[Section 2.2]{Pol2} is a collection of full subcategories of $\Ac$ such that
\begin{itemize}
\item $\Hom_{\Ac} (C_i,C_j)=0$ for any $C_i \in \mathcal C_i, C_j \in \mathcal C_j$ where $i<j$;
\item every object $E$ of $\Ac$ admits a filtration in $\Ac$
\[
  0=E_0 \subseteq E_1 \subseteq E_2 \subseteq \cdots \subseteq E_n = E
\]
where $E_i/E_{i-1} \in \mathcal C_i$ for each $1 \leq i \leq n$.
\end{itemize}
The same notion  also appeared in Toda's work \cite[Definition 3.5]{Toda2}.

Any time we have a torsion pair $(\mathcal T, \mathcal F)$ in an abelian category $\Ac$, the extension closure in $D^b(\Ac)$
\begin{align*}
  \Ac^\# &= \langle \mathcal F [1], \mathcal T \rangle \\
  &= \{ E \in D^b(\Ac) : H^{-1}(E) \in \mathcal F, H^0(E) \in \mathcal T, H^i(E) =0 \, \, \forall\, i \neq -1,0\}
\end{align*}
will be the heart of another t-structure on $D^b(\Ac)$.  We say $\Ac^\#$ is the heart obtained by tilting $\Ac$ at the torsion pair $(\mathcal T,\mathcal F)$.

For any polarisation $\omega$ on $X$,  we have the torsion pair $(\Tw, \Fw)$ in $\Coh (X)$ where
\begin{align*}
  \Tw &= \{ E \in \Coh (X): \text{for every sheaf quotient }E \twoheadrightarrow A, \mu_\omega (A)>0\}, \\
  \Fw &= \{ E \in \Coh (X): \text{ for every subsheaf $F \subset E$}, \mu_\omega (F) \leq 0\}
\end{align*}
and also the torsion pair $(W_{0,X},W_{1,X})$ where
\begin{align*}
  W_{0,X} &= \{ E \in \Coh (X): E \text{ is $\Phi$-WIT$_0$}\}, \\
  W_{1,X} &= \{ E \in \Coh (X): E \text{ is $\Phi$-WIT$_1$}\}.
\end{align*}

Since $\Coh (X)$ is a noetherian abelian category, any subcategory $\mathcal T$ closed under quotient and extension in $\Coh (X)$ is the torsion class of a torsion pair in $\Coh (X)$ \cite[Lemma 1.1.3]{Pol}.  The following are all torsion classes in $\Coh (X)$:
\begin{align*}
 \Coh^{\leq d}(X) &= \{ E \in \Coh (X): \dimension \mathrm{supp}(E) \leq d\}, \text{ for any nonnegative integer $d$}, \\
 \Coh (\pi)_{\leq d} &= \{ E \in \Coh (X): \dimension \pi (\mathrm{supp}(E)) \leq d \}, \text{ for any nonnegative integer $d$},\\
 \{\Coh^{\leq 0}\}^\uparrow &= \{E \in \Coh (X): \dimension \mathrm{supp}(E|_s)=0 \text{ for every closed point $s \in S$}\}, \\
 W'_{i,X} &= \{ E \in \Coh (X) : E|_s \text{ is $\Phi_s$-WIT$_i$ for a general closed point $s \in S$}\}, \text{ for $i=0,1$}.
\end{align*}
We also write $\Coh (\pi)_0$ for $\Coh(\pi)_{\leq 0}$.

The following diagram shows the inclusion relations among these torsion classes (see also \cite[Section 1]{Lo11}):
\[
\xymatrix{
 & \Coh (\pi)_{\leq 1} \ar[r] & \Coh^{\leq 2}(X) \ar[r] & W_{0,X}' \\
 \Coh (\pi)_{\leq 0} \ar[r] & \Coh^{\leq 1}(X) \ar[u] & & \\
 \Coh^{\leq 0}(X) \ar[rr] \ar[u] & & \{\Coh^{\leq 0}\}^\uparrow \ar[uu] \ar[r] & W_{0,X} \ar[uu]
}.
\]
For any subcategory $\CC$ of an abelian category $\Ac$, we set
\[
  \CC^\circ = \{ E \in \Ac : \Hom_{\Ac}(A,E)=0 \text{ for all $A \in \CC$}\}.
\]
Whenever $\Ac$ is a noetherian abelian category and $\mathcal T$ is a torsion class in $\Ac$, the category $\mathcal T^\circ$ will be the corresponding torsion-free class in $\Ac$.

Lastly, we set
\[
  \Coh^d(\pi)_e = \{ E \in \Coh (X) : \dimension \mathrm{supp} (E)=d, \dimension \pi (\mathrm{supp}(E))=e\}
\]
for any nonnegative integers $d, e$.

\subsection{Tilt stability} Tilting $\Coh (X)$ at the torsion pair $(\Tw, \Fw)$ gives us the heart
\[
  \Bw = \langle \Fw [1], \Tw \rangle.
\]
In fact, $\Bw$ is a noetherian abelian category \cite[Lemma 3.2.4]{BMT1}, and we have the positivity properties
\begin{itemize}
\item for any $E \in \Bw$, we have $\omega^2 \ch_1 (E) \geq 0$;
\item for any $E \in \Bw$ satisfiying $\omega^2 \ch_1 (E)=0$, we have $\omega \ch_2(E) - \tfrac{\omega^3}{6} \ch_0(E) \geq 0$
\end{itemize}
from \cite[Lemma 3.2.1]{BMT1}.  Therefore, we have the slope-like function on $\Bw$
\begin{equation}
  \nu_\omega (E) = \frac{\omega \ch_2 (E) - \tfrac{\omega^3}{6}\ch_0(E)}{\omega^2\ch_1(E)}.
\end{equation}
An object $E \in \Bw$ is called $\nu_\omega$-semistable if, for every short exact sequence in $\Bw$
\begin{equation}\label{eq51}
0 \to A \to E \to B \to 0
\end{equation}
where $A,B \neq 0$, we have
\begin{equation}\label{eq50}
 \nu_\omega (A) \leq \nu_\omega (B).
\end{equation}
Additionally,  an object $E \in \Bw$ is called $\nu_\omega$-stable if, for every short exact sequence \eqref{eq51} in $\Bw$, we have strict inequality $<$ in \eqref{eq50}.

In the literature, $\nu_\omega$-stability is also called tilt stability.  Tilt stability is an important notion in the study of Bridgeland stability conditions, as it is an intermediate  step in the construction of a Bridgeland stability condition on an arbitrary smooth projective threefold \cite{BMT1, BMS}.

Another equivalent formulation of $\nu_\omega$-stability for objects in $\Bw$ is via the group homomorphism $\oZw : K(X) \to \mathbb{C}$ defined by
\begin{equation}\label{eq2}
 \bZob (E) = \frac{\omega^2}{2} \ch_1^B(E) + i\left( \omega \ch_2^B(E) - \frac{\omega^3}{6} \ch_0^B(E)\right).
\end{equation}
We say $\oZw$ is a reduced central charge because it takes objects in $\Bw$ to the non-strict right-half complex plane including the origin $i\mathbb{H}_0$, where
\[
  \mathbb{H}_0 = \{ z \in \mathbb{C} : \Im z > 0\} \cup \{ z \in \mathbb{C} : \Im z =0, \Re z \leq 0\}.
\]
From \cite[Lemma 3.2.1]{BMT1} and its proof, we know that the objects in $\Bw$ that are taken by $\oZw$ to 0 are precisely the 0-dimensional sheaves on $X$.

For any object $E \in \Bw$, we can now define its phase $\phi (E) \in (-\tfrac{1}{2},\tfrac{1}{2}]$ by the relation
\[
  \oZw (E) \in \mathbb{R}_{>0} e^{i\pi \phi (E)}
\]
if $\oZw (E) \neq 0$; if $\oZw (E)=0$, we set $\phi (E)=\tfrac{1}{2}$.  Using the phase function $\phi$, we can characterise $\nu_\omega$-semistability as follows: an object $E \in\Bw$ is $\nu_\omega$-semistable (resp.\ $\nu_\omega$-stable) if, for every short exact sequence \eqref{eq51} in $\Bw$, we have  $\phi (A) \leq \phi (B)$ (resp.\  $\phi (A) < \phi (B)$).

With $\omega = tH+sD$ and $\ch (E) = (a_{ij})$, we can write
\begin{equation}\label{eq3}
 \oZw (E) = ( 2htsa_{01} + hs^2a_{10}) + i (ta_{02} + 2hsa_{11} -hts^2a_{00}).
\end{equation}
If $ts=\alpha$ for some fixed nonzero real number $\alpha$, then \eqref{eq3} reduces to
\begin{equation}\label{eq4}
  \oZw (E) =  ( 2h\alpha a_{01} + hs^2a_{10}) + i (ta_{02} + 2hsa_{11} -h\alpha sa_{00}).
\end{equation}
We regard \eqref{eq4} as a Laurent polynomial in $s$ (since $t=\alpha s^{-1}$).

\section{Fourier-Mukai transform as a permutation of categories}\label{sec-generation}

In this section, we introduce some of the  torsion pairs that we will use to decompose $\Coh (X)$, and hence the derived category $D^b(X)$.  These torsion pairs will allow us to visualise the Fourier-Mukai transform $\Phi$ as a permutation (up to shift) on suitable categories.

We will use diagrams of the form $\scalea{\gyoung(;*;*;*,;*;*;*)}$, where each $*$ is a sign ($+,-$ or $0$) or empty, to denote  various subcategories of $\Coh (X)$.   Roughly speaking, we think of $\scalea{\gyoung(;;;,;;;)}$ as a 2 by 3 matrix whose entries correspond to the different components of the Chern character in our matrix notation in Section \ref{section-matrixchern}.

For example, if the $\ch_{01}$-position in $\scalea{\gyoung(;;;,;;;)}$ has the sign `$+$',   it means that every sheaf in this category has $\ch_{01}>0$.  When a particular position is left blank, it means the corresponding $\ch_{ij}$ is zero.  We will concatenate multiple such diagrams to denote their extension closure in $\Coh (X)$.

First we set
\[
 \scalea{\gyoung(;;;,;;;+)}=C^0 := \Coh^{\leq 0}(X).
\]

Then, for sheaves in the abelian category $\Coh (\pi)_0$, which we call fiber sheaves, we have the slope-like function $\mu = \ch_3/H\ch_{02}$.  With respect to this slope-like function, we set
\begin{align*}
      \scalea{\gyoung(;;;+,;;;+)}=C^1_{0,+} &:= \{ E \in \Coh^1(\pi)_0 :  \text{ all $\mu$-HN factors of $E$ have $\infty>\mu>0$}\} \\
  \scalea{\gyoung(;;;+,;;;0)}=C^1_{0,0} &:=\{ E \in \Coh^1(\pi)_0 : \text{ all $\mu$-HN factors of $E$ have $\mu=0$}\} \\
  \scalea{\gyoung(;;;+,;;;-)}=C^1_{0,-} &:= \{ E \in \Coh^1(\pi)_0 :  \text{ all $\mu$-HN factors of $E$ have $\mu<0$}\}.
\end{align*}
Note that $C^0, C^1_{0,+}$ are contained in $W_{0,X}$ while  $C^1_{0,0}, C^1_{0,-}$ are contained in $W_{1,X}$ \cite[Corollary 3.29]{FMNT}.  The four categories we have defined so far generate the category of all fiber sheaves, i.e.\
\[
  \Coh (\pi)_0 = \Coh^1(\pi)_0 =
   \xymatrix @-2.3pc{
   \scalea{\gyoung(;;;,;;;+)} &  \scalea{\gyoung(;;;+,;;;+)} \\
   & \scalea{\gyoung(;;;+,;;;0)}\\
   & \scalea{\gyoung(;;;+,;;;-)}
   }.
\]
Notice that $\Phi$ induces the equivalences
\[
   \scalea{\gyoung(;;;,;;;+)} \overset{\thicksim}{\longrightarrow} \scalea{\gyoung(;;;+,;;;0)} \text{\quad and \quad} \scalea{\gyoung(;;;+,;;;+)} \overset{\thicksim}{\longrightarrow} \scalea{\gyoung(;;;+,;;;-)}
\]
by \cite[Corollary 3.29]{FMNT}.

To generate the entire category $\Coh^{\leq 1}(X)$, we also need
\[
  \scalea{\gyoung(;;;*,;;+;*)} = C^1_1 := \Coh^1(\pi)_1 \cap \{\Coh^{\leq 0}\}^\uparrow.
\]
It follows from  \cite[Lemma 3.15]{Lo11} that
\[
\Coh^{\leq 1}(X) = \xymatrix @-2.3pc{
   \scalea{\gyoung(;;;,;;;+)} &  \scalea{\gyoung(;;;+,;;;+)}  & \scalea{\gyoung(;;;*,;;+;*)}\\
   & \scalea{\gyoung(;;;+,;;;0)} & \\
   & \scalea{\gyoung(;;;+,;;;-)} &
   }.
\]

Next, to generate sheaves in $\Coh (\pi)_{\leq 1}$, we consider
\begin{align*}
  \scalea{\gyoung(;;+;*,;;+;*)}= C^2_{1,+} &:= \Coh^2(\pi)_1 \cap W_{0,X}  \\
  \scalea{\gyoung(;;+;*,;;0;*)}= C^2_{1,0} &:= \left\{ E \in \Phi ( \{\Coh^{\leq 0}\}^\uparrow \cap \Coh^{\leq 1}(X)) : \ch_{01}(E) \neq 0\right\} \\
  \scalea{\gyoung(;;+;*,;;-;*)} =C^2_{1,-} &:= \{ E \in \Coh^2(\pi)_1 \cap W_{1,X} : \ch_{11}(E) < 0\}.
\end{align*}

\begin{remark}\label{remark10}
For any object $E \in  \scalea{\gyoung(;;+;*,;;+;*)}$, we know $\ch(E)$ is of the form \mbox{\tiny $\begin{pmatrix}
  0 & \ast & \ast \\
  0 & \ast & \ast
\end{pmatrix}$}
with $\ch_{01}>0$ and $\ch_{11} \geq 0$  \cite[Remark 5.17]{LZ2}.  However, if $E$ is nonzero, then we must have $\ch_{11}(E)>0$.  For, if $\ch_{11}(E)=0$, then $\wh{E}$ would be a $\Phi$-WIT$_1$ sheaf with $\ch (\wh{E}) = \mbox{\tiny $\begin{pmatrix}
                                                                0 & 0 & \ast \\
                                                                0 & \ast & \ast
                                                              \end{pmatrix}$}$,
i.e.\ $\wh{E} \in \Coh^{\leq 1}$ \cite[Lemma 5.10]{LZ2}.  Then $\wh{E}$ must be a fiber sheaf, implying $E$ itself is a fiber sheaf, contradicting $\dimension E = 2$.  For this reason, we have the `+' sign in the $\ch_{11}$-position of $\scalea{\gyoung(;;+;*,;;+;*)}$.
\end{remark}
We have
\[
\Coh (\pi)_{\leq 1} = \left\langle \Coh^{\leq 1}(X), \vcenter{\vbox{
\xymatrix @-2.3pc{
 \scalea{\gyoung(;;+;*,;;+;*)} \\
 \scalea{\gyoung(;;+;*,;;0;*)} \\
 \scalea{\gyoung(;;+;*,;;-;*)}
}
}} \right\rangle
\]
by Lemma \ref{lemma24}.  Note the following equivalences of categories under $\Phi$:
\[
\scalea{\gyoung(;;;*,;;+;*)} \overset{\thicksim}{\longrightarrow} \scalea{\gyoung(;;+;*,;;0;*)} \text{\quad and \quad} \scalea{\gyoung(;;+;*,;;+;*)} \overset{\thicksim}{\longrightarrow} \scalea{\gyoung(;;+;*,;;-;*)}.
\]

To generate  $\Coh^{\leq 2}(X)$, i.e.\ the category of torsion sheaves on $X$, we also need
\[
\scalea{\gyoung(;;*;*,;+;*;*)} = C^2_{2,+} := \{ E \in \Coh^2(\pi)_2 \cap W_{0,X} : \ch_{10}(E)>0 \}.
\]
We have
\[
  \Coh^{\leq 2}(X) = \left\langle \Coh (\pi)_{\leq 1}(X), \scalea{
  \gyoung(;;*;*,;+;*;*)} \right\rangle
\]
by Lemma \ref{lemma25}.

To generate the entire category of coherent sheaves $\Coh (X)$, we need three more subcategories:
\begin{align*}
  \scalea{\gyoung(;+;*;*,;+;*;*)} = C^3_{2,_+} &= \Coh^3(\pi)_2 \cap W_{0,X}  \\
  \scalea{\gyoung(;+;*;*,;0;*;*)} = C^3_{2,0} &= \{ E \in \Coh^3(\pi)_2 \cap W_{1,X} : \ch_{10}(E)=0\} \\
  \scalea{\gyoung(;+;*;*,;-;*;*)} = C^3_{2,-} &= \{ E \in \Coh^3 (\pi)_2 \cap W_{1,X} : \ch_{10}(E)<0\}.
\end{align*}
Note that, for the same reason as in Remark \ref{remark10}, any nonzero object $E \in \Coh^3 (\pi)_2 \cap W_{0,X}$ must have $\ch_{10}(E)>0$.  This is why we have a  `+' sign in the $\ch_{10}$-position of $\scalea{\gyoung(;+;*;*,;+;*;*)}$.  Lemma \ref{lemma26} now gives
\[
  \Coh (X) = \left\langle \Coh^{\leq 2}(X), \vcenter{\vbox{
  \xymatrix @-2.3pc{
   \scalea{\gyoung(;+;*;*,;+;*;*)} \\
   \scalea{\gyoung(;+;*;*,;0;*;*)} \\
   \scalea{\gyoung(;+;*;*,;-;*;*)}
   }
   }} \right\rangle.
\]

Note that $\Phi$ induces the equivalences of categories
\[
  \scalea{\gyoung(;;*;*,;+;*;*)} \overset{\thicksim}{\longrightarrow} \scalea{\gyoung(;+;*;*,;0;*;*)} \text{\quad and \quad} \scalea{\gyoung(;+;*;*,;+;*;*)}  \overset{\thicksim}{\longrightarrow} \scalea{\gyoung(;+;*;*,;-;*;*)}.
\]

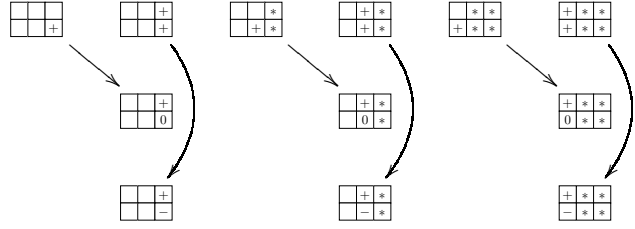
\begin{figure}
  \centerline{
  \xymatrix @-0.7pc{
\scalea{\gyoung(;;;,;;;+)} \ar[dr] & \scalea{\gyoung(;;;+,;;;+)} \ar@/^1.5pc/[dd] & \scalea{\gyoung(;;;*,;;+;*)} \ar[dr] & \scalea{\gyoung(;;+;*,;;+;*)} \ar@/^1.5pc/[dd] & \scalea{\gyoung(;;*;*,;+;*;*)} \ar[dr] & \scalea{\gyoung(;+;*;*,;+;*;*)} \ar@/^1.5pc/[dd] \\
 & \scalea{\gyoung(;;;+,;;;0)} & & \scalea{\gyoung(;;+;*,;;0;*)} & & \scalea{\gyoung(;+;*;*,;0;*;*)} \\
 & \scalea{\gyoung(;;;+,;;;-)} & & \scalea{\gyoung(;;+;*,;;-;*)} & & \scalea{\gyoung(;+;*;*,;-;*;*)}
 }
 }
  \caption{The geometry of the Fourier-Mukai transform $\Phi$ }\label{tetris1}
\end{figure}

Figure \ref{tetris1} summarises the various equivalences of categories induced by the Fourier-Mukai transform $\Phi$.  The configuration of Figure \ref{tetris1}  also contains  information about torsion classes in $\Coh (X)$:
\begin{enumerate}
\item The categories in the top row (those whose lower left-most entry is a `+') are all contained in $W_{0,X}$, while those in the second the third rows (those whose lower left-most  entry is 0 or `$-$') are contained in $W_{1,X}$.
\item The categories in columns 1 through $i$ generate a torsion class in $\Coh (X)$.  As $i$ goes from 1 to 6, they generate the nested sequence
    \begin{equation}\label{eq45}
      \Coh^{\leq 0}(X) \subset \Coh (\pi)_0 \subset \Coh^{\leq 1}(X) \subset \Coh (\pi)_{\leq 1} \subset \Coh^{\leq 2}(X) \subset \Coh (X).
    \end{equation}
\item For $i \in \{1, 3, 5\}$ and any $1 \leq j \leq 3$, the categories in columns 1 through $i$ together with those in rows 1 through $j$ of column $(i+1)$ also generate a torsion class in $\Coh (X)$ (see Lemma \ref{lemma27}).

    More explicitly, we have the following nested sequence of torsion classes refining \eqref{eq45}:
\begin{align}
&\scalebox{0.5}{
\xymatrix{
\ysize\gyoung(;;;,;;;+)
}
} \label{eq46}
\\
\subset
&\scalebox{0.5}{
\xymatrix @-2.5pc{
\ysize\gyoung(;;;,;;;+) & \ysize\gyoung(;;;+,;;;+)
}
}
\subset
\scalebox{0.5}{
\xymatrix @-2.5pc{
\ysize\gyoung(;;;,;;;+) & \ysize\gyoung(;;;+,;;;+) \\
& \ysize\gyoung(;;;+,;;;0)
}
}
\subset
\scalebox{0.5}{
\xymatrix @-2.5pc{
\ysize\gyoung(;;;,;;;+) & \ysize\gyoung(;;;+,;;;+) \\
& \ysize\gyoung(;;;+,;;;0) \\
& \ysize\gyoung(;;;+,;;;-)
}
}
\subset
\scalebox{0.5}{
\xymatrix @-2.5pc{
\ysize\gyoung(;;;,;;;+) & \ysize\gyoung(;;;+,;;;+) & \ysize\gyoung(;;;*,;;+;*) \\
& \ysize\gyoung(;;;+,;;;0) & \\
& \ysize\gyoung(;;;+,;;;-) &
}
} \notag
\\
\subset
&\scalebox{0.5}{
\xymatrix @-2.5pc{
\ysize\gyoung(;;;,;;;+) & \ysize\gyoung(;;;+,;;;+) & \ysize\gyoung(;;;*,;;+;*) & \ysize\gyoung(;;+;*,;;+;*) \\
& \ysize\gyoung(;;;+,;;;0) & &  \\
& \ysize\gyoung(;;;+,;;;-) & &
}
}
\subset
\scalebox{0.5}{
\xymatrix @-2.5pc{
\ysize\gyoung(;;;,;;;+) & \ysize\gyoung(;;;+,;;;+) & \ysize\gyoung(;;;*,;;+;*) & \ysize\gyoung(;;+;*,;;+;*) \\
& \ysize\gyoung(;;;+,;;;0) & & \ysize\gyoung(;;+;*,;;0;*) \\
& \ysize\gyoung(;;;+,;;;-) & &
}
}
\subset
\scalebox{0.5}{
\xymatrix @-2.5pc{
\ysize\gyoung(;;;,;;;+) & \ysize\gyoung(;;;+,;;;+) & \ysize\gyoung(;;;*,;;+;*) & \ysize\gyoung(;;+;*,;;+;*) \\
& \ysize\gyoung(;;;+,;;;0) & & \ysize\gyoung(;;+;*,;;0;*) \\
& \ysize\gyoung(;;;+,;;;-) & & \ysize\gyoung(;;+;*,;;-;*)
}
}
\subset
\scalebox{0.5}{
\xymatrix @-2.5pc{
\ysize\gyoung(;;;,;;;+) & \ysize\gyoung(;;;+,;;;+) & \ysize\gyoung(;;;*,;;+;*) & \ysize\gyoung(;;+;*,;;+;*) & \ysize\gyoung(;;*;*,;+;*;*) \\
& \ysize\gyoung(;;;+,;;;0) & & \ysize\gyoung(;;+;*,;;0;*) &\\
& \ysize\gyoung(;;;+,;;;-) & & \ysize\gyoung(;;+;*,;;-;*) &
}
} \notag
\\
\subset
&\scalebox{0.5}{
\xymatrix @-2.5pc{
\ysize\gyoung(;;;,;;;+) & \ysize\gyoung(;;;+,;;;+) & \ysize\gyoung(;;;*,;;+;*) & \ysize\gyoung(;;+;*,;;+;*) & \ysize\gyoung(;;*;*,;+;*;*) & \ysize\gyoung(;+;*;*,;+;*;*)\\
& \ysize\gyoung(;;;+,;;;0) & & \ysize\gyoung(;;+;*,;;0;*) & & \\
& \ysize\gyoung(;;;+,;;;-) & & \ysize\gyoung(;;+;*,;;-;*) & &
}
}
\subset
\scalebox{0.5}{
\xymatrix @-2.5pc{
\ysize\gyoung(;;;,;;;+) & \ysize\gyoung(;;;+,;;;+) & \ysize\gyoung(;;;*,;;+;*) & \ysize\gyoung(;;+;*,;;+;*) & \ysize\gyoung(;;*;*,;+;*;*) & \ysize\gyoung(;+;*;*,;+;*;*)\\
& \ysize\gyoung(;;;+,;;;0) & & \ysize\gyoung(;;+;*,;;0;*) & & \ysize\gyoung(;+;*;*,;0;*;*)\\
& \ysize\gyoung(;;;+,;;;-) & & \ysize\gyoung(;;+;*,;;-;*) & &
}
}
\subset
\scalebox{0.5}{
\xymatrix @-2.5pc{
\ysize\gyoung(;;;,;;;+) & \ysize\gyoung(;;;+,;;;+) & \ysize\gyoung(;;;*,;;+;*) & \ysize\gyoung(;;+;*,;;+;*) & \ysize\gyoung(;;*;*,;+;*;*) & \ysize\gyoung(;+;*;*,;+;*;*)\\
& \ysize\gyoung(;;;+,;;;0) & & \ysize\gyoung(;;+;*,;;0;*) & & \ysize\gyoung(;+;*;*,;0;*;*)\\
& \ysize\gyoung(;;;+,;;;-) & & \ysize\gyoung(;;+;*,;;-;*) & & \ysize\gyoung(;+;*;*,;-;*;*)
}
} \notag
\end{align}
\end{enumerate}

\section{Limit tilt stability}\label{section-limittiltstab}

\subsection{Heuristics - refining slope stability}\label{section-heuristics}

For a fixed polarisation $\omega$ on $X$, the notion of $\mu_\omega$-stability on coherent sheaves can be phrased in terms of the weak central charge $Z_{\mu_\omega} : K(X) \to \mathbb{C}$ defined by
\begin{equation}\label{eq1}
  Z_{\mu_\omega} (E) = -\omega^2 \ch_1(E) + i \ch_0(E).
\end{equation}
For any coherent sheaf $E$ on $X$, the value of $Z_{\mu_\omega}(E)$ lies in $\mathbb{H}_0$, and so we can define the phase $\psi (E) \in (0,1]$ of $E$ by the relations
\[
  \begin{cases}
  Z_{\mu_\omega}(E) \in \mathbb{R}_{>0}e^{i\pi \psi (E)} &\text{ if $Z_{\mu_\omega} (E)\neq 0$} \\
  \psi (E)=1 &\text{ if $Z_{\mu_\omega} (E)= 0$}.
  \end{cases}
\]
A coherent sheaf $E$ on $X$ is  called $\mu_\omega$-semistable (resp.\ $\mu_\omega$-stable) if, for every short exact sequence in $\Coh (X)$
\[
0 \to A \to E \to B \to 0
\]
where $A,B \neq 0$, we have $\psi (A) \leq \psi (B)$ (resp.\ $\psi (A) < \psi (B)$).

Let us set
\[
\wh{\Coh (X)} = \Phi (\Coh (X))[1] = \langle W_{1,X}[1], W_{0,X}\rangle.
\]
For any fixed coherent sheaf $E$ on $X$, let us write $F = (\Phi (E))[1]$,  $\ch (E)=(a_{ij})$ and $\ch (F) = (b_{ij})$; then $F$ lies in the heart $\wh{\Coh (X)}$ and we  have the relation
\[
  \ch (F) = \begin{pmatrix} b_{00} & b_{01} & b_{02} \\
  b_{10} & b_{11} & b_{12} \end{pmatrix} =
  \begin{pmatrix} -a_{10} & -a_{11} & -a_{12} \\
  a_{00} & a_{01} & a_{02}
  \end{pmatrix}.
\]
With respect to the polarisation
\begin{equation}\label{eq52}
  \bar{\omega} = \tfrac{\lambda}{\alpha} H + \lambda D
\end{equation}
where $\lambda, \alpha >0$ are real numbers, we have
\begin{equation}\label{eq12}
 \bar{\omega}^2 \ch_1 (E) = 2h (2 \tfrac{\lambda^2}{\alpha} b_{11} - \lambda^2 b_{00}) = 2h \tfrac{\lambda^2}{\alpha}(2b_{11}-\alpha b_{00}).
\end{equation}
This means that the group homomorphism $Z_\alpha : K(X) \to \mathbb{C}$ given by
\begin{equation}\label{eq5}
  Z_\alpha (F) = b_{10} + i  (2b_{11}-\alpha b_{00})
\end{equation}
always takes objects $F$ in $\wh{\Coh (X)}$ into the half-plane $-i\mathbb{H}_0$ - see Figure \ref{fig:pic2}.

As a result, for every object $F \in \wh{\Coh (X)}$ we can define its phase $\psi_\alpha (F) \in (-\tfrac{1}{2},\tfrac{1}{2}]$ by the relations
\[
  \begin{cases}
  Z_{\alpha}(F) \in \mathbb{R}_{>0}e^{i\pi \psi_\alpha (E)} &\text{ if $Z_{\alpha} (F)\neq 0$} \\
  \psi_\alpha (F)=\tfrac{1}{2} &\text{ if $Z_{\alpha} (F)= 0$}.
  \end{cases}
\]
This gives a notion of $Z_\alpha$-stability where, for an object $F \in \wh{\Coh (X)}$, we say it is $Z_\alpha$-semistable (resp.\ $Z_\alpha$-stable) if for every short exact sequence in $\wh{\Coh (X)}$
\[
0 \to A \to F \to B \to 0
\]
where $A,B \neq 0$, we have $\psi_\alpha (A) \leq \psi_\alpha (B)$ (resp.\ $\psi_\alpha (A) < \psi_\alpha (B)$).  Overall, if  $E$ is a coherent sheaf on $X$ and we set $F=\Phi E[1]$, then
\begin{equation}
  E  \text{ is $\mu_{\bo}$-semistable}  \text{ if and only if } F  \text{ is $Z_\alpha$-semistable}.
\end{equation}
That is, we can regard the pair $(\wh{\Coh (X)}, Z_\alpha)$ as the image of $(\Coh (X), Z_{\mu_{\bo}})$ under the action of $\Phi$.

Notice that the terms in $Z_\alpha$ correspond to the two highest-degree terms of $\oZw (F)$ (up to scalar multiples) when we consider $\oZw (F)$ as a Laurent polynomial in $s$.  This motivates the principle that $\oZw$ should yield a notion of stability that refines $\mu_{\bo}$-stability for coherent sheaves.  This principle will be made precise in the rest of this paper.

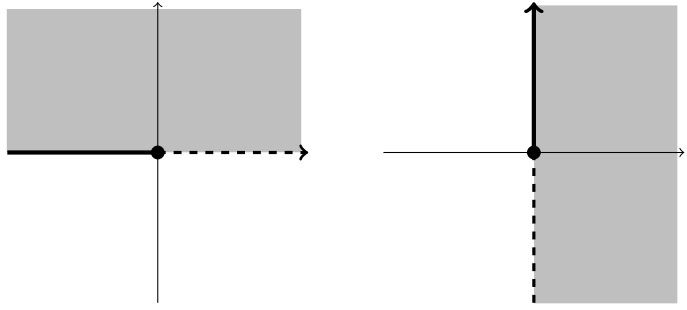
\begin{figure}
\centering
\begin{tikzpicture}
\filldraw[color=gray!50] (-2,0.01)rectangle(1.9,1.9);
\draw[black, ultra thick] (-2,0) -- (0,0);
\draw[black, very thick, dashed, ->] (0,0) -- (2,0);
\draw[black, thin, ->] (0,-2) -- (0,2);
\filldraw[black, very thick] (0,0) circle (2pt);
%
\filldraw[color=gray!50] (5.01,-2)rectangle(6.9,1.95);
\draw[black, thin, ->] (3,0) -- (7,0);
\draw[black, ultra thick, ->] (5,0) -- (5,2);
\draw[black, very thick, dashed] (5,-2) -- (5,0);
\filldraw[black, very thick] (5,0) circle (2pt);
\end{tikzpicture}
\caption{Half-planes $\mathbb{H}_0$ (left) and $-i\mathbb{H}_0$ (right)}
\label{fig:pic2}
\end{figure}

\subsection{A limit of $\Bw$}\label{Bwl-def}

Since the polarisation $\omega = tH+sD$ depends on the coefficients $t,s$, the heart $\Bw$ also depend on $t,s$.  In this section, we show that when $\alpha >0$ is fixed and $t, s$ move  along the positive branch of the hyperbola $ts=\alpha$ on the $(s,t)$-plane in the direction $s \to \infty$, there is a sensible definition of the limit of $\Bw$, which we denote by $\Bl$.


\begin{lemma}\label{lemma1}
Given a fixed real number $t_0>0$ and a coherent sheaf $F$ on $X$, the following are equivalent:
\begin{itemize}
\item[(a)] There exists some $s_0>0$ such that $F \in \Fw$ for all $(s,t) \in (s_0,\infty) \times (0,t_0)$.
\item[(b)] There exists some $s_0>0$ such that, for any nonzero subsheaf $A \subseteq F$, we have $\mu_\omega (A) \leq 0$ for all $(s,t) \in (s_0,\infty) \times (0,t_0)$.
\item[(c)] For any nonzero subsheaf $A \subseteq F$, either (i) $\mu_f (A)<0$, or (ii) $\mu_f(A)=0$ and $\mu^\ast (A) \leq 0$.
\end{itemize}
\end{lemma}

Let $c$ denote the least possible positive fiber degree, which happens to be 1 in our case of the product elliptic threefold $X = C \times S$.

\begin{proof}
(a) $\Leftrightarrow$ (b): this is clear.

(b) $\Rightarrow$ (c): given (b), that (i) or (ii) must hold is clear from \eqref{eq6}.

(c) $\Rightarrow$ (b): suppose (c) holds.  Take any nonzero subsheaf $A \subseteq F$.   Since $\mu_f (A) \leq 0$ by hypothesis, we know $F$ must be torsion-free.  If $\mu_f (A)=0$, then $\mu^\ast (A) \leq 0$ by assumption, and we have $\mu_\omega (A) \leq 0$ for any ample class $\omega$.

On the other hand, if $\mu_f (A) <0$, then from $0 < \rk (A) \leq \rk (F)$ we have
\[
  \mu_f (A) \leq \frac{-c}{\rk (F)} < 0.
\]
Also, since $\mu^\ast$ is a slope-like function on $\Coh (X)$, we have $\mu^\ast_{\max}(F) \geq \mu^\ast (B)$ for any subsheaf $B \subseteq F$.  In particular
\[
  \mu^\ast (A)  \leq \mu^\ast_{\max} (F) < \infty
\]
where the second inequality holds because $F$ is torsion-free. Writing $\ch(A)=(a_{ij})$,  \eqref{eq6} now gives
\[
  \mu_\omega (A) = 4hts\mu^\ast (A) + 2hs^2 \mu_f (A).
\]
Therefore, if $t \in (0,t_0)$ then $4ht\mu^\ast (A)$ is bounded above by $4ht_0 \mu^\ast_{\max}(F)$ which depends only on $t_0$ and $F$, while $2h\mu_f(A)$ is strictly negative with a nonzero upper bound depending only on $\rk (F)$.  From  \eqref{eq6} , we now see that there exists an $s_0>0$, depending only on $F$ and $t_0$, such that  for any nonzero subsheaf $A \subseteq F$ and any $(s,t) \in (s_0,\infty) \times (0,t_0)$ we have $\mu_\omega (A) < 0$.  Thus (b) holds.
\end{proof}

\begin{remark}\label{remark9}
In the implication (c) $\Rightarrow$ (a) in the proof of Lemma \ref{lemma1}, we can pick $s_0$ to be
\[
  s_0 = \frac{4ht_0\mu^\ast_{\max}(F)}{2h\tfrac{c}{\rk(F)}} = \frac{2t_0\rk(F)\mu^\ast_{\max}(F)}{c}.
\]
\end{remark}

Lemma \ref{lemma1} has the following counterpart:

\begin{lemma}\label{lemma2}
Given a fixed real number $t_0 >0$ and a coherent sheaf $F$ on $X$, the following are equivalent:
\begin{itemize}
\item[(a)] There exists some $s_0>0$ such that $F \in \Tw$ for all $(s,t) \in (s_0, \infty) \times (0,t_0)$.
\item[(b)] There exists some $s_0>0$ such that, for any nonzero sheaf quotient $F \twoheadrightarrow A$, we have $\mu_\omega (A) > 0$ for all $(s,t) \in (s_0, \infty) \times (0,t_0)$.
\item[(c)] For any nonzero sheaf quotient $F \twoheadrightarrow A$, either (i) $\mu_f (A)>0$, or (ii) $\mu_f(A)=0$ and $\mu^\ast (A) > 0$.
\end{itemize}
\end{lemma}

Although the proof of Lemma \ref{lemma2} is similar to that of Lemma \ref{lemma1}, we still include it here for the sake of clarity and completeness.

\begin{proof}
(a) $\Leftrightarrow$ (b): this is clear.

(b) $\Rightarrow$ (c): this is also clear from \eqref{eq6}.

(c) $\Rightarrow$ (b): suppose (c) holds. Take any nonzero sheaf quotient $F \twoheadrightarrow A$.  By hypothesis, we have $\mu_f (A) \geq 0$.  If $A$ is a torsion sheaf, then $A \in \Tw$ for any ample class $\omega$.  From now on, let us assume that $\rk (A) >0$.  If $\mu_f (A)=0$, then $\mu^\ast (A) >0$ by assumption, and so $\mu_\omega (A) = 4hts\mu^\ast (A) >0$ for all $(s,t) \in (0,\infty) \times (0,\infty)$.  So let us also assume $\mu_f (A) > 0$ from here onwards.

Now we have $\mu^\ast (A) \geq \mu^\ast_{\min}(F)$ where $\mu^\ast_{\min}(F)<\infty$.  Thus $4ht\mu^\ast (A)$ is bounded below by $-4ht_0 |\mu^\ast_{\min}(F)|$ which depends only on $t_0$ and $F$, while
\[
  \mu_f (A) \geq \frac{c}{\rk (A)} \geq \frac{c}{\rk (F)} > 0.
\]
Therefore, there exists some $s_0>0$ depending only on $F$ and $t_0$ such that, for all $(s,t) \in (s_0,\infty) \times (0,t_0)$ we have $\mu_\omega (A)>0$.
\end{proof}

Note that  condition (c) in Lemma \ref{lemma1} and  condition (c) in Lemma \ref{lemma2} are both independent of $t_0$ as well as $\omega$.  This prompts us to define the following categories:
\begin{itemize}
\item $\Tc^l$, the extension closure of all coherent sheaves satisfying condition (c) in Lemma \ref{lemma2}.
\item $\Fc^l$, the extension closure of all coherent sheaves satisfying condition (c) in Lemma \ref{lemma1}.
\end{itemize}
We also define the extension closure
\begin{equation}\label{eq56}
  \Bl = \langle \Fc^l [1], \Tc^l \rangle.
\end{equation}

\begin{remark}\label{remark14}
Immediate properties of $\Tl, \Fl$ and $\Bl$ include:
 \begin{itemize}
 \item[(i)] The category of torsion sheaves on $X$ is contained in $\Tl$.  This follows from \eqref{eq61} and the fact that  every torsion sheaf is contained in $\Tw$ for any polarisation $\omega$.
 \item[(ii)] Every nonzero object in $\Fl$ is a torsion-free sheaf.  This follows from \eqref{eq61} and the fact that every nonzero object in $\Fw$ is a torsion-free sheaf for any polarisation $\omega$.
 \item[(iii)] $W_{0,X} \subseteq \Tl$.  To see this, take any $E \in W_{0,X}$.  Since  $W_{0,X} \subseteq W_{0,X}'$ \cite[Lemma 3.6]{Lo11},  we know  each $\mu_f$-HN factor of $E$ has $\mu_f>0$ from \cite[Proposition 3.3]{LZ2}.  As a result, any quotient sheaf $E'$ of $E$ satisfies $\mu_f (E') >0$.  By  Lemma \ref{lemma2}(c), this means $E \in \Tl$.
 \item[(iv)] $\ch_{10}(E) \geq 0$ for any $E \in \Bl$.  To see this, observe that for any $E' \in \Fl$ we have $\ch_{10}(E') \leq 0$ by Lemma \ref{lemma1}(c), while for any $E'' \in \Tl$ we have $\ch_{10}(E'') \geq 0$ by Lemma \ref{lemma2}(c).  Subsequently,  for any $E \in \Bl$ we have $\ch_{10}(E) \geq 0$.
 \item[(v)] $\Fl \subset W_{1,X}$.  Since $(W_{0,X}, W_{1,X})$ and $(\Tl, \Fl)$ are both torsion pairs in $\Coh (X)$ (see Lemma \ref{lemma3} below), this follows immediately from (iii).
 \item[(vi)] It is straightforward to check that  arguments similar to those in the proofs of Lemmas \ref{lemma1} and \ref{lemma2} give  the following equivalent descriptions of $\Tl, \Fl$ for any fixed $\alpha >0$:
\begin{align}
  \Tc^l &= \{ F \in \Coh (X) : F \in \Tw \text{ for $s \gg 0, ts = \alpha$}\};\notag\\
  \Fc^l &= \{ F \in \Coh (X) : F \in \Fw \text{ for $s \gg 0, ts = \alpha$}\}. \label{eq61}
\end{align}
More concretely, by Lemmas \ref{lemma1} and \ref{lemma2}, an object $E \in D^b(X)$ lies in $\Bl$ (resp.\ $\Tc^l, \Fc^l$) if and only if $E$ lies in $\Bw$ (resp.\ $\Tw, \Fw$) for all $s>s_0, ts= \alpha$, for some $s_0>0$.  That is, for any fixed $\alpha >0$, we can interpret  $\Bl$ (resp.\ $\Tc^l, \Fc^l$) as the `limit' of the category $\Bw$ (resp.\ $\Tw, \Fw$) as $s \to +\infty$ along the positive branch of the hyperbola  $ts=\alpha$.
\end{itemize}
\end{remark}

In fact, the effective bound $s_0$ in Remark \ref{remark14}(vi) only depends on the Chern classes of $E$.  This observation is needed if one tries to show, using Theorem \ref{theorem2}, that $\Phi$ induces isomorphisms between moduli spaces of slope stable sheaves and moduli spaces of tilt stable objects:

\begin{lemma}\label{lemma31}
Let $\alpha >0$ be fixed, and write $\omega = tH + sD$.  Suppose $B$ is a scheme of finite type, and $\mathcal{E} \in D^b(X \times B)$ is a $B$-flat family of limit tilt stable objects in $\Bl$.  Then there exists $M>0$ such that, for all $s>M, ts = \alpha$ and $b \in B$, we have $\mathcal{E}_b \in \Bw$.
\end{lemma}

\begin{proof}
Suppose $ts = \alpha$.  It is easy to check the following (see the proofs of Lemmas \ref{lemma1} and \ref{lemma2}):
\begin{itemize}
\item If $F \in \Fl$, then $F \in \Fw$ for any $s>0$ satisfying
\[
  s^2 \geq \frac{4h\alpha |\mu^\ast_{max}(F)|}{2h} \cdot \frac{\rank (F)}{c}.
\]
\item If $F \in \Tl$, then $F \in \Tw$ for any $s>0$ satisfying
\[
  s^2 > \frac{4h\alpha |\mu^\ast_{min}(F)|}{2h} \cdot \frac{\rank (F)}{c}.
\]
\end{itemize}
Applying a flattening stratification to the standard cohomology of $\mathcal{E}$ (which are coherent sheaves on $X \times B$) with respect to the morphism $X \times B \to B$, and keeping in mind the two inequalities above, the lemma would follow if we can show the following: given a $B$-flat family $\mathcal{F}$ of coherent sheaves on $X$ where $B$ is of finite type, $\mu^\ast_{max}(\mathcal{F}_b)$ and $\mu^\ast_{min}(\mathcal{F}_b)$ are bounded for $b \in B$.

For any $m>0$, both $2H+mD$ and $H+mD$ are  ample divisors on $X$.  By the existence of relative HN filtration for $\mathcal{F}$ and a flattening stratification, the slopes
\[
  \mu_{2H+mD, max}(\mathcal{F}_b), \,\, \mu_{2H+mD, min}(\mathcal{F}_b), \,\,\mu_{H+mD, max}(\mathcal{F}_b), \,\, \mu_{H+mD, min}(\mathcal{F}_b)
\]
are all bounded for $b \in B$.  Since  we can rewrite $\mu^\ast (F)$ as
\[
  \mu^\ast (F) = \frac{H\cdot c_1(F)}{\rank (F)} = \frac{(2H+mD)\cdot c_1(F)}{\rank (F)} - \frac{(H+mD)\cdot c_1(F)}{\rank (F)}= \mu_{2H+mD}(F) - \mu_{H+mD}(F),
\]
it follows that
\[
\mu^\ast_{max} (\mathcal{F}_b), \,\, \mu^\ast_{min} (\mathcal{F}_b)
\]
are both  bounded for $b \in B$ for any $m>0$.  The lemma thus follows.
\end{proof}

\subsection{Definition of limit tilt stability on  $\Bl$}

 Since $\Bw$ is the heart of a t-structure on $D^b(X)$ for each $s,t>0$, we would hope that the category $\Bl$ is also the heart of a t-structure on $D^b(X)$.  This is indeed the case:

 \begin{lemma}\label{lemma3}
 $(\Tl,\Fl)$ is a torsion pair in $\Coh (X)$, and the category $\Bl$ is the heart of a t-structure on $D^b(X)$.
 \end{lemma}

 Let us write $(\Tl)^\circ := \{ F \in \Coh (X) : \Hom_{\Coh (X)}(E,F)=0 \text{ for all }E \in \Tl\}$.

 \begin{proof}
 It is clear from Lemma \ref{lemma2}(a) that $\Tl$ is closed under extensions and quotients in $\Coh (X)$.  Since $\Coh (X)$ is noetherian, we know  $(\Tl, (\Tl)^\circ)$ is a torsion pair in $\Coh (X)$ by \cite[Lemma 1.1.3]{Pol}.  We now show that $(\Tl, \Fl)$ satisfies the two axioms of a torsion pair; this will imply $\Fl = (\Tl)^\circ$, and the lemma will follow.

 Let us fix a real number $\alpha >0$ and set $ts  =\alpha$. Take any $E \in \Coh (X)$.  We have the $(\Tl, (\Tl)^\circ)$-decomposition of $E$
 \begin{equation}\label{eq9}
   0 \to E' \to E \to E'' \to 0.
 \end{equation}
 By Lemma \ref{lemma2}(a), we know  there exists $s_0 >0$ such that $E' \in \Tw$ for all $s > s_0$.  Also, since $\Tl$ contains all torsion sheaves on $X$, $E''$ must be torsion-free if nonzero.

 Suppose  $E'' \notin \Fl$.  Then $E''$ is nonzero,  and there exists a divergent, strictly  increasing sequence $(s_i)_{i \geq 1}$ of real numbers  such that $s_i > s_0$ for all $i \geq 1$, and  $E'' \notin \Fw$ for $s=s_i$ for all $i\geq 1$.  In particular, for each $i \geq 1$, the sheaf $E''$ has a nonzero subsheaf $A_i$ lying in $\Tw$ for $s=s_i$.

 For each $i$, that $A_i \in \Tw$ for $s=s_i$  implies
 \begin{equation}\label{eq7}
 0< \mu_\omega(A_i) = 4h\alpha\mu^\ast (A_i) + 2hs^2 \mu_f (A_i) \text{\quad for $s=s_i$}.
 \end{equation}
On the other hand, since $W_{0,X} \subseteq \Tl$, we have $(\Tl)^\circ \subseteq (W_{0,X})^\circ = W_{1,X}$.  Thus $E'' \in W_{1,X}$, and so $A_i \in W_{1,X}$ for all $i$, implying
\begin{equation}\label{eq8}
  \mu_f (A_i) \leq 0  \text{\quad for all $i$}.
\end{equation}

Suppose $(s_i)$ has a subsequence $(s_{i_j})_{j \geq 1}$ such that, for each $j$, we have $\mu_f (A_{i_j})< 0$.  Then the inequality \eqref{eq7} implies
\[
  \mu^\ast (A_{i_j}) > \frac{2hs_{i_j}^2 |\mu_f (A_{i_j})|}{4h\alpha} = \frac{s_{i_j}^2}{2\alpha} |\mu_f (A_{i_j})| \text{\quad for each $j$},
\]
where
\[
  |\mu_f (A_{i_j})| \geq \frac{c}{\rk (E'')} >0 \text{\quad for each $j$}.
\]
Hence $\mu^\ast (A_{i_j}) \to \infty$ as $j \to \infty$.  However, each $A_{i_j}$ is a  subsheaf of the nonzero torsion-free sheaf $E''$, which implies  $\mu^\ast (A_{i_j}) \leq \mu^\ast_{\max} (E'') < \infty$ for all $j$ \cite[Section 5.2]{LZ2}, which is a contradiction. Therefore, by omitting a finite number of terms if necessary, we can assume that $\mu_f (A_i)=0$ for all $i$.

Fix any $i$, and consider the $\mu_f$-HN filtration of $A_i$.  Since $A_i \in W_{1,X}$, we have $A_i \in W'_{1,X}$ by \cite[Lemma 3.16(ii)]{Lo11}.  This implies that the fiber $A_i |_s$ is $\Phi_s$-WIT$_1$ for a general closed point $s \in S$, and so all the $\mu_f$-HN factors of $A_i$ have $\mu_f   \leq 0$ \cite[Corollary 3.29]{FMNT}. Together with $\mu_f (A_i) =0$, this implies that there is only one factor in the $\mu_f$-HN filtration of $A_i$, i.e.\ $A_i$ is $\mu_f$-semistable.  As a result, for any nonzero quotient sheaf $Q$ of $A_i$,  there are two cases:
\begin{itemize}
\item[(i)]  $\mu_f (Q)>0$;
\item[(ii)] $\mu_f (Q)=0$.  In this case, we have  $Q \in \Tw$ since $A_i \in \Tw$.  Then $\mu_\omega (Q)=4h\alpha\mu^\ast (Q)>0$, implying $\mu^\ast (Q)>0$.
\end{itemize}
This shows that $A_i \in \Tl$, contradicting the maximality of $E'$.  Thus $E''$ must lie in $\Fl$.  We have now shown that every $E \in \Coh (X)$  fits in a short exact sequence of the form \eqref{eq9} where $E' \in \Tl$ and $E'' \in \Fl$.

To finish the proof, note that for any $E \in \Tl$ and any $F \in \Fl$, we can find some $s_0>0$ such that $E \in \Tw$ and $F \in \Fw$ for any $s>s_0$ by Lemmas \ref{lemma1} and \ref{lemma2}.  Then $\Hom (E,F)=0$.  This shows that $(\Tl, \Fl)$ is a torsion pair in $\Coh (X)$, i.e.\ $\Bl$ is the heart of a t-structure on $D(X)$.
 \end{proof}

\begin{lemma}\label{lemma4}
Fix any $\alpha >0$.  For any $E \in \Bl$, we have
\[
  \oZw (E)  \in -i\mathbb{H}_0 \text{ for $s \gg 0, ts =\alpha$}.
\]
\end{lemma}

\begin{proof}
Take any $E \in \Bl$.  From Lemmas \ref{lemma1} and \ref{lemma2}, we know  there exists some $s_0>0$ such that $E \in \Bw$ for any $s>s_0$.   Hence for any $s>s_0$, we have  $\oZw (E) \in -i\mathbb{H}_0$.
\end{proof}

As a result of Lemma \ref{lemma4},  $\oZw$ is a polynomial stability function on $\Bl$ in the sense of Bayer \cite{BayerPBSC}, provided we set $ts=\alpha$ for some real number $\alpha>0$ and treat $\oZw$ as a Laurent polynomial in $s$.

For any object $E \in \Bl$, we can now define the polynomial phase function $\phi (E)$ by requiring $-\tfrac{1}{2} < \phi (E) \leq \tfrac{1}{2}$ for $s \gg 0$ and
\[
  \begin{cases} \oZw (E) \in \mathbb{R}_{>0} e^{i\pi \phi (E)} &\text{ if $\oZw (E) \neq 0$ for $s \gg 0$}, \\
  \phi (E) = \tfrac{1}{2} &\text{ if $\oZw (E)=0$}.
  \end{cases}
\]
We say an object $E \in \Bl$ is $\nu^l$-semistable or limit tilt semistable (resp.\ $\nu^l$-stable or limit tilt stable) if, for every short exact sequence in $\Bl$
\[
0 \to A \to E \to B \to 0
\]
where $A,B \neq 0$, we have
\[
  \phi (A) \leq (\text{resp.} <)\, \phi (B) \text{ for $s \gg 0$}.
\]

For objects $A, B \in \Bl$, we will often write $\phi (A) \preceq \phi (B)$ (resp.\ $\phi (A) \prec \phi (B)$) to mean $\phi (A) \leq \phi (B)$ for $s \gg 0$ (resp.\ $\phi (A) < \phi (B)$ for $s \gg 0$).

\subsection{Decomposing $\Bl$}\label{sec-generation2}

In this section, we will refine the torsion pair in  \eqref{eq56}.  Let us set
\begin{align*}
\Tc^{l,+} &:= \langle F \in \Tl : F \text{ is $\mu_f$-semistable}, \infty > \mu_f (F) > 0\rangle;\\
\Tc^{l,0} &:= \{  F \in \Tl : F \text{ is $\mu_f$-semistable}, \mu_f(F)=0\};\\
\Fc^{l,0} &:= \{ F \in \Fl : F \text{ is $\mu_f$-semistable}, \mu_f(F)=0\}; \\
\Fc^{l,-} &:= \langle F \in \Fl : F \text{ is $\mu_f$-semistable}, \mu_f(F)<0\rangle.
\end{align*}

\begin{remark}\label{remark11}
The following properties of $\Tc^{l,+}, \Tc^{l,0}, \Fc^{l,0}, \Fc^{l,-}$ are immediate:
\begin{itemize}
\item[(i)] Any $F \in \Tl$ has a filtration $F'' \subseteq F' \subseteq F$ where $F'' \in\Coh^{\leq 2}(X)$, $F'/F'' \in \Tc^{l,+}$  and $F/F' \in \Tc^{l,0}$, while any $F \in \Fl$ has a filtration $F' \subseteq F$ where $F' \in \Fc^{l,0}$ and $F/F' \in \Fc^{l,-}$.  We obtain these filtrations by considering the $\mu_f$-HN filtration of $F$.
\item[(ii)] Any $F \in \Tc^{l,0}$ has $\mu^\ast (F)>0$ while any $F \in \Fc^{l,0}$ has $\mu^\ast (F)\leq 0$.
\item[(iii)] Any $F$ in $\Tc^{l,0}$ or $\Fc^{l,0}$ is $\Phi$-WIT$_1$.  To see this, note that any such $F$ is $\mu_f$-semistable with  $\mu_f<\infty$, and so is torsion-free.  Also, $F \in W_{1,X}'$ by \cite[Corollary 3.29]{FMNT}.  These two conditions together imply $F \in W_{1,X}$ by \cite[Lemma 3.18]{Lo11}.
\end{itemize}
\end{remark}

\begin{remark}\label{remark15}
 Since $W_{0,X}$ is contained in $\Tl$ by Remark \ref{remark14}(iii), the heart $\wh{\Coh(X)}$ is the tilt of $\Bl$ at the torsion pair $(\mathcal T, \mathcal F)$ where
\begin{align*}
    \mathcal T &= \langle \Fl [1], W_{0,X}\rangle, \\
    \mathcal F &=  W_{1,X} \cap \Tl
\end{align*}
by \cite[Proposition 2.3.2]{BMT1}.  From this torsion pair $(\Tc, \Fc)$ in $\Bl$, we see \begin{equation}\label{eq57}
\Phi \Bl \subset D^{[0,1]}_{\Coh (X)}
\end{equation}
and that  an object in $\Bl$ is taken by $\Phi$ to a coherent sheaf sitting at degree zero precisely when it lies in $\langle \Fl [1], W_{0,X}\rangle$.    Having this torsion pair in $\Bl$ implies we  have the torsion triple in $\Bl$
\begin{equation}\label{eq19}
  (\Fl [1], W_{0,X}, W_{1,X} \cap \Tl ).
\end{equation}
\end{remark}

\begin{lemma}\label{lemma7}
We have
\[
  W_{1,X} \cap \Tl = \langle W_{1,X} \cap \Coh^{\leq 2}(X), \Tc^{l,0} \rangle,
\]
and every object in this category can be written as the extension of a sheaf in $\Tc^{l,0}$ by a sheaf in $W_{1,X} \cap \Coh^{\leq 2}(X)$.  For every $F \in W_{1,X} \cap \Tl$, the transform $\wh{F}$ is a torsion sheaf.
\end{lemma}

\begin{proof}
Take any $F \in W_{1,X} \cap \Tl$.  Let $T$ be the maximal torsion subsheaf of $F$.  Then $T \in W_{1,X} \cap \Coh^{\leq 2}(X)$ while $F/T$ is torsion-free and lies in $\Tl$.  By the definition of $\Tl$, all the $\mu_f$-HN factors of $F/T$ have $\ch_{10} \geq 0$.  On the other hand, that $F \in W_{1,X}$ implies $\ch_{10}(F) \leq 0$ by \cite[Remark 5.17]{LZ2}.  Since $\ch_{10}(T)=0$, all the $\mu_f$-HN factors of $F/T$ must have $\ch_{10}=0$, i.e.\ $F/T$ is $\mu_f$-semistable with $\ch_{10}=0$, i.e.\ $F/T \in \Tc^{l,0}$.  This proves the first half of the lemma.

Note that $\ch_{10}(F)=0$ from the previous paragraph.  Thus $\ch_0 (\wh{F})=0$, i.e.\ $\wh{F}$ is a torsion sheaf.
\end{proof}

\begin{lemma}\label{lemma8}
We have the torsion quintuple in $\Bl$
\begin{equation}\label{eq47}
  \left(\Fc^{l,0}[1],\,\, \Fc^{l,-}[1],\,\, \Coh^{\leq 2}(X),\,\, \Tc^{l,+},\,\, \Tc^{l,0} \right).
\end{equation}
\end{lemma}

\begin{proof}
Take any $E \in \Bl$.  We have the short exact sequence in $\Bl$
\[
0 \to H^{-1}(E)[1] \to E \to H^0(E) \to 0
\]
where $H^{-1}(E) \in \Fl$ and $H^0(E) \in \Tl$. The lemma then follows from Remark \ref{remark11}(i).
\end{proof}

\subsection{The transform $\Phi\Bw$ vs $\Bl$}

The torsion triple \eqref{eq19} in $\Bl$ implies the following torsion triple in $\Phi \Bl$:
\begin{equation}\label{eq17}
  ( \Phi \Fl [1], W_{1,X}, \Phi (W_{1,X}\cap \Tl) ).
\end{equation}
Since $\Fl \subseteq W_{1,X}$, the heart $\Phi \Bl[1]$ can be obtained by tilting $\Coh (X)$ at the torsion pair in which
\begin{itemize}
\item the torsion class is $\mathcal T_B :=  \Phi (W_{1,X} \cap \Tl) [1]$, and
\item the torsion-free class is $\mathcal T_B^\circ = \langle \Phi \Fl [1], W_{1,X}\rangle$
\end{itemize}
 by \cite[Proposition 2.3.2]{BMT1}.  Note that $\mathcal T_B$ is contained in the category of torsion  sheaves on $X$ by Lemma \ref{lemma7}.  Consequently, the category $\mathcal T_B^\circ$  contains all the torsion-free  sheaves on $X$.

\begin{lemma}\label{lemma9}
For any torion-free coherent sheaf $F$ on $X$, we have $\Phi F [1] \in \Bl$.
\end{lemma}

\begin{proof}
Since $F \in \mathcal T_B^\circ \subseteq \Phi \Bl$ from \eqref{eq17} and the discussion above, we have $\Phi F [1] \in \Phi (\Phi \Bl)[1] =\Bl$.
\end{proof}

\begin{proposition}\label{pro-theorem1}
  $\Phi (\Bw) \subseteq \langle \Bl, \Bl [-1]\rangle$ for any polarisation $\omega$ on $X$.
\end{proposition}

\begin{proof}
Given any torsion coherent sheaf $T$, we have
\[
  \Phi T \in \langle \Coh (X), \Coh (\pi)_{\leq 1}[-1]\rangle
  \]
by \cite[Lemma 2.6]{Lo7}.  Since the category of torsion sheaves is contained in $\Bl$, we have $\Coh (\pi)_{\leq 1}[-1] \subset \Bl [-1]$.  The proposition then follows.
\end{proof}
In other words, $\Phi \Bw$ is a tilt of $\Bl$, a limit of $\Bw$.

\subsection{Phases of various objects}\label{phase-computation}

In this section, we compute the phases of particular   objects in  $\Bl$.  We will repeatedly make use of \cite[Proposition 5.13]{LZ2} without explicitly mentioning it.

\textbf{Assumption.} For the rest of this article, we will fix an $\alpha \in \mathbb{R}_{>0}$, and assume   $\omega=tH+sD$ with $t, s \in \mathbb{R}_{>0}$  satisfying $ts=\alpha$, and $\bo = \tfrac{\lambda}{\alpha} H + \lambda D$  where $\lambda \in \mathbb{R}_{>0}$.

Suppose $F$ is a nonzero object in $\Bl$ and write $\ch(F)=(b_{ij})$.  Consider the following special cases of $F$:
\begin{itemize}
\item[(i)] $F \in \Coh^{\leq 1}(X)$.
  \begin{itemize}
  \item[(a)] $F \in \Coh^{\leq 0}(X)$: then $\oZw (F)$ is identically zero  and  $\phi (F)=\tfrac{1}{2}$ by definition.
  \item[(b)] $\dimension F=1$:  then either $b_{02}$ or $b_{11}$ is strictly positive, and so $\oZw (F)\in i\mathbb{R}_{> 0}$ for any $s >0$, i.e.\ $\phi (F) = \tfrac{1}{2}$ for all $s>0$.
  \end{itemize}
\item[(ii)] $F \in \scalea{\gyoung(;;*;*,;+;*;*)}$: since $b_{00}=0, b_{01} \geq 0$ and $b_{10}>0$, we have $\Re \oZw (F)>0$ for all $s>0$.  As $s \to \infty$, $\oZw (F)$ is dominated by the term $hs^2b_{10}$ in $\Re \oZw (F)$, and    so $\phi (F) \to 0$.
\item[(iii)] $F \in \Coh^2(\pi)_1$: we have $\ch (F) = \mbox{\tiny $\begin{pmatrix} 0 & \ast & \ast \\ 0 & \ast & \ast \end{pmatrix}$}$ with $\ch_{01}(F)>0$, and $\oZw (F)=2h\alpha b_{01} + i (tb_{02}+ 2hsb_{11})$.    Consider the following special cases:
  \begin{itemize}
  \item[(A)] $F \in W_{0,X}$: then $F \in \scalea{\gyoung(;;+;*,;;+;*)}$, and  $b_{11} > 0$ by Remark \ref{remark10}.  Then  $\phi (F) \to \tfrac{1}{2}$ as $s \to \infty$.
  \item[(B)] $F \in W_{1,X}$: then $b_{11} \leq  0$ by \cite[Remark 5.17]{LZ2}.  Consider the two further special cases:
        \begin{itemize}
        \item[(a)] $F \in \scalea{\gyoung(;;+;*,;;-;*)}$: then $b_{11}<0$ and  $\phi (F) \to -\tfrac{1}{2}$ as $s \to \infty$.
        \item[(b)] $F \in \scalea{\gyoung(;;+;*,;;0;*)}$: then $b_{11}=0, b_{01}>0$ and $\oZw (F) = 2h\alpha b_{01} + itb_{02}$, and so $\phi (F) \to 0$ as $s \to \infty$.  Depending on whether $b_{02}$ is $>0, =0$ or $<0$, we have $\phi (F) \to 0^+$, $\phi (F)$ is identically zero, or $\phi (F) \to 0^-$, respectively.
        \end{itemize}
  \end{itemize}
\item[(iv)] $F \in \Tw^{l,+}$: then $b_{10}>0$ and $\oZw (F)$ is dominated by the term $hs^2b_{10}$ as $s \to \infty$.  Therefore,  $\phi (F) \to 0$ as $s \to \infty$.
\item[(v)] $F \in \Tc^{l,0}$: then
\[
  \oZw (F) = 2h\alpha b_{01} + i (tb_{02} + 2hsb_{11} - h\alpha s b_{00})
\]
where $2h\alpha b_{01}>0$ since $\mu^\ast (F)>0$.  Also, $F \in W_{1,X}$ by Remark \ref{remark11}, and $\wh{F}$ is a 2-dimensional sheaf.   Writing $\ch (\wh{F})=(\hat{b}_{ij})$, we obtain
    \[
      0 < \bo^2\ch_1(\wh{F}) = 4h\tfrac{\lambda^2}{\alpha}\hat{b}_{01} + 2h\lambda^2\hat{b}_{10} = -2h\tfrac{\lambda^2}{\alpha}(2b_{11}-\alpha b_{00}).
    \]
    This shows $(2b_{11}-\alpha b_{00}) <0$, and so $\phi (F) \to -\tfrac{1}{2}$ as $s \to \infty$.
\item[(vi)] $F \in \Fc^{l,0}$: then $F \in W_{1,X}$ also by Remark \ref{remark11}, and the argument in (v) shows that the leading term in $\oZw (F[1])$ is $-ihs(2b_{11}-\alpha b_{00})$, and so $\phi (F[1]) \to \tfrac{1}{2}$ as $s \to \infty$.
\item[(vii)] $F \in \Fw^{l,-}$: then $b_{10}<0$, and  the leading term of $\oZw (F[1])$ is $hs^2(-b_{10})$, implying $\phi (F[1]) \to 0$ as $s \to \infty$.
\end{itemize}
We  summarise these  phase computations  in Table \ref{table1}.

\afterpage{
\begin{table}[p]
\caption{Phases of various types of sheaves under $\oZw$ as $s \to \infty$}
\centering
\begin{tabular}{|m{1.7cm}|m{6cm}|m{2.5cm}|}
\hline
Case & Object $F$ & $\oZw (F), s \to \infty$  \\
\hline
(i)(a) & $\Coh^{\leq 0}(X)$ & \setlength{\unitlength}{1.5mm}
\begin{picture}(10,10)
\put(0,5){\vector(1,0){10}}
\put(5,0){\vector(0,1){10}}
\put(5,5){\circle*{1}}
\end{picture} \\
\hline
(i)(b) & $F \in \Coh^{\leq 1}(X), \dimension F=1$ & \setlength{\unitlength}{1.5mm}
\begin{picture}(10,10)
\put(0,5){\vector(1,0){10}}
\put(5,0){\vector(0,1){10}}
\linethickness{1mm}
\put(5,5.3){\line(0,1){4.7}}
\put(5,5){\circle{1}}
\end{picture} \\
\hline
(ii) & $\scalea{\gyoung(;;*;*,;+;*;*)}$ & \setlength{\unitlength}{1.5mm}
\begin{picture}(10,10)
\put(0,5){\vector(1,0){10}}
\put(5,0){\vector(0,1){10}}
\linethickness{1mm}
\put(8.5,5){\line(1,0){1.5}}
\end{picture} \\
\hline
(iii)(A) & $\scalea{\gyoung(;;+;*,;;+;*)}$ & \setlength{\unitlength}{1.5mm}
\begin{picture}(10,10)
\put(0,5){\vector(1,0){10}}
\put(5,0){\vector(0,1){10}}
\linethickness{1mm}
\put(5.2,8.5){\line(0,1){1.5}}
\end{picture} \\
\hline
(iii)(B)(a) & $\scalea{\gyoung(;;+;*,;;-;*)}$ & \setlength{\unitlength}{1.5mm}
\begin{picture}(10,10)
\put(0,5){\vector(1,0){10}}
\put(5,0){\vector(0,1){10}}
\linethickness{1mm}
\put(5.2,1.5){\line(0,-1){1.5}}
\end{picture} \\
\hline
(iii)(B)(b) & $\scalea{\gyoung(;;+;*,;;0;*)}$ &
\setlength{\unitlength}{1.5mm}
\begin{picture}(10,10)
\put(0,5){\vector(1,0){10}}
\put(5,0){\vector(0,1){10}}
\linethickness{1mm}
\put(5.3,5){\line(1,0){4.7}}
\put(5,5){\circle{1}}
\end{picture} \\
\hline
(iv) & $\Tc^{l,+}$&
\setlength{\unitlength}{1.5mm}
\begin{picture}(10,10)
\put(0,5){\vector(1,0){10}}
\put(5,0){\vector(0,1){10}}
\linethickness{1mm}
\put(8.5,5){\line(1,0){1.5}}
\end{picture} \\
\hline
(v) & $\Tc^{l,0}$ &
\setlength{\unitlength}{1.5mm}
\begin{picture}(10,10)
\put(0,5){\vector(1,0){10}}
\put(5,0){\vector(0,1){10}}
\linethickness{1mm}
\put(5.2,1.5){\line(0,-1){1.5}}
\end{picture} \\
\hline
(vi) & $\Fc^{l,0}[1]$ &
\setlength{\unitlength}{1.5mm}
 \begin{picture}(10,10)
\put(0,5){\vector(1,0){10}}
\put(5,0){\vector(0,1){10}}
\linethickness{1mm}
\put(5.2,8.5){\line(0,1){1.5}}
\end{picture} \\
\hline
(vii) & $\Fc^{l,-}[1]$ &
\setlength{\unitlength}{1.5mm}
\begin{picture}(10,10)
\put(0,5){\vector(1,0){10}}
\put(5,0){\vector(0,1){10}}
\linethickness{1mm}
\put(8.5,5){\line(1,0){1.5}}
\end{picture} \\
\hline
\end{tabular}
\label{table1}
\end{table}
\FloatBarrier
}

\section{Slope stability vs limit tilt stability}\label{section-slopevslimittilt}

In this section, we prove the following comparison theorem between $\mu_{\bo}$-stability for coherent sheaves and $\nu^l$-stability for objects in $\Bl$:

\begin{theorem}[$\mu_{\bo}$-stability vs $\nu^l$-stability]\label{theorem2}
Fix any real numbers $\lambda, \alpha >0$.  Let $\bo = \tfrac{\lambda}{\alpha} H + \lambda D$ and $\omega = tH + sD$ where $ts=\alpha$ and $t, s \in \mathbb{R}_{>0}$.
\begin{itemize}
\item[(A)] Suppose $E$ is a $\mu_{\bar{\omega}}$-stable torsion-free coherent sheaf on $X$ satisfying
    \begin{equation}\label{eq53}
      \Hom (W_{0,X}\cap \Coh^{\leq 1}(X),E[1])=0.
    \end{equation}
    Then $\Phi E [1]$ is a $\nu^l$-stable object in $\Bl$.
\item[(B)] Suppose $F \in \Bl$ is a $\nu^l$-semistable object with $\ch_{10}(F) \neq 0$.  Consider the decomposition of $F$ in $\Bl$ with respect to the torsion triple \eqref{eq19}
    \[
    0 \to F' \to F \to F'' \to 0
    \]
    where $F' \in \langle \Fl [1], W_{0,X}\rangle$ and $F''\in W_{1,X} \cap \Tl$.  Then $\Phi (F')$ is a torsion-free $\mu_{\bo}$-semistable sheaf, while $\Phi (F'')[1]$ is a $\Phi$-WIT$_0$ sheaf in $\Coh^{\leq 1}(X)$.
\end{itemize}
\end{theorem}

\begin{remark}\label{remark12}
Torsion-free reflexive sheaves $E$ on a smooth projective threefold $X$ satisfy the vanishing
\[
  \Hom (\Coh^{\leq 1}(X), E[1]) = 0
\]
by \cite[Lemma 4.20]{CL}.  Hence for any $\mu_{\bar{\omega}}$-stable reflexive sheaves $E$ on $X$, we know $\Phi E [1]$ is a $\nu^l$-stable object in $\Bl$ by Theorem \ref{theorem2}(A).
\end{remark}

\begin{remark}\label{remark5}
 For  a torsion-free sheaf $E$ on $X$, we know $\Phi E [1] \in \Bl$ by Lemma \ref{lemma9}.  Also, $\Phi (\Phi E [1]) = E \in \Coh (X)$, and so $\Phi E [1] \in \langle \Fl [1], W_{0,X} \rangle$ by Remark \ref{remark15}.  Therefore, part (A) of Theorem \ref{theorem2} says that any $\mu_{\bo}$-stable torsion-free sheaf is taken by $\Phi$ into $ \langle \Fl [1], W_{0,X} \rangle$.  
 On the other hand, part (B) of Theorem \ref{theorem2} implies the following: any $\nu^l$-semistable object $F$ in the category $\langle \Fl [1], W_{0,X} \rangle$ with $\ch_{10}(F) \neq 0$  is taken by $\Phi$ to a $\mu_{\bo}$-semistable torsion-free sheaf.
\end{remark}

Recall, that for an arbitrary $\mu_\bo$-stable torsion-free sheaf $E$ on $X$, we can take the double dual of $E$ and form the exact triangle in $D^b(X)$
\[
  T[-1] \to E \to E^{\ast \ast} \to T
\]
where $T \in \Coh^{\leq 1}(X)$ while $E^{\ast \ast}$ is reflexive and  $\mu_\bo$-stable.  With this and Remarks \ref{remark12} and \ref{remark5}, we can now interpret Theorem \ref{theorem2} as follows: for any polarisation $\bo$ on $X$, the Fouier-Mukai transform takes  any $\mu_\bo$-stable torsion-free sheaf on $X$ to a limit stable object `up to a modification in codimension 2'.  Conversely, the Fourier-Mukai transform $\Phi$ takes a  limit tilt semistable object on $X$ to a $\mu_\bo$-semistable sheaf  `up to a modification in codimension 2'.

\subsection{Proof of Theorem \ref{theorem2}(A)}

\subsubsection{Objects in $W_{1,X} \cap \Tl$}\label{para-W1Tlanalysis}

We start with a brief analysis of the phases of certain objects in $W_{1,X} \cap \Tl$ that will be useful in the proof of Theorem \ref{theorem2}(A): take any $M \in W_{1,X} \cap \Tl$.  From Lemma \ref{lemma7}, we know $\ch_{10}(M)=0$.  Writing $\ch (M)=(m_{ij})$, we have
\[
  \oZw (M) = 2h\alpha m_{01} + i (tm_{02} + 2hsm_{11}-h\alpha s m_{00}).
\]
Lemma \ref{lemma7} also tells us there is a short exact sequence of sheaves
\[
0 \to M' \to M \to M'' \to 0
\]
where $M'' \in \Tc^{l,0}$ and $M' \in W_{1,X} \cap \Coh^{\leq 2}(X)$.  Let $\ch(M')=(m_{ij}')$ and $\ch(M'')=(m_{ij}'')$.
\begin{itemize}
\item From  case (v) of Section \ref{phase-computation}, we know  $\phi (M'') \to -\tfrac{1}{2}$ if $M'' \neq 0$.
\item If $\dimension M' = 2$: then $m'_{00}=0, m'_{01}>0$ and $m'_{11} \leq 0$.
  \begin{itemize}
  \item If $m'_{11}<0$, then $\phi (M') \to -\tfrac{1}{2}$.
  \item If $m'_{11}=0$, then $\phi (M') \to 0$.
  \end{itemize}
\item If $\dimension M' =1$: then $M'$ is a fiber sheaf, and $\phi (M') = \tfrac{1}{2}$.
\end{itemize}

\begin{proof}[Proof of Theorem \ref{theorem2}(A)]
Let $E$ be any $\mu_{\bar{\omega}}$-stable torsion-free coherent sheaf on $X$, and write $F = \Phi E[1]$.  That $\ch_0 (E) \neq 0$ implies $\ch_{10}(F) \neq 0$, and so $\phi (F) \to 0$. We know  $F \in \Bl$ from Lemma \ref{lemma9}.  Take any short exact sequence in $\Bl$
\[
 0 \to G \to F \to F/G \to 0
\]
where $G \neq 0$.  This yields the long exact sequence of coherent sheaves
\[
0 \to \Phi^0 G \to E \overset{\alpha}{\to} \Phi^0 (F/G) \to \Phi^1G \to 0 \to \Phi^1 (F/G) \to 0,
\]
and so $\Phi^1 (F/G)=0$.  From the torsion triple \eqref{eq19}, we also have the exact triangle
 \begin{equation}\label{eq18}
      \Phi (\Phi^0 G)[1] \to G \to \Phi (\Phi^1 G) \to \Phi (\Phi^0 G)[2]
 \end{equation}
 where $\Phi (\Phi^0 G)[1] \in \langle \Fl [1], W_{0,X}\rangle$ and $\Phi (\Phi^1 G) \in W_{1,X} \cap \Tl$.

Suppose $\rank (\image \alpha)=0$.  Then $\rank (\Phi^0 G) = \rank (E) >0$, and so $\ch_{10} (\Phi (\Phi^0 G)[1]) >0$.  We divide into two cases:
\begin{itemize}
\item[(a)] $\ch_1 (\image \alpha) \neq 0$.  In this case, we have $\mu_{\bar{\omega}}(\Phi^0 G) < \mu_{\bo}(E)$, which implies $\phi (\Phi(\Phi^0 G)[1]) \prec \phi (F)$.
    \begin{itemize}
    \item If $\dimension \Phi (\Phi^1 G)=3$, then from Section \ref{para-W1Tlanalysis} we have $\phi (\Phi (\Phi^1 G)) \to -\tfrac{1}{2}$, and so  $\phi (G) \prec \phi (F)$.
    \item If $\dimension \Phi (\Phi^1 G)\leq 2$, then the $\Tc^{l,0}$-component of  $\Phi (\Phi^1 G)$ vanishes, and $\Phi (\Phi^1 G) \in W_{1,X} \cap \Coh^{\leq 2}(X)$.
       \begin{itemize}
       \item If  $\dimension \Phi (\Phi^1 G) =2$ and   $\ch_{11} (\Phi (\Phi^1 G)) <0$:  then $\phi (\Phi (\Phi^1 G)) \to -\tfrac{1}{2}$ from Section  \ref{para-W1Tlanalysis} and we have $\phi (G) \prec \phi (F)$.
       \item If  $\dimension \Phi (\Phi^1 G)=2$ and $\ch_{11} (\Phi (\Phi^1 G))=0$, or if $\dimension \Phi (\Phi^1 G)=1$, the Laurent polynomial $\oZw ( \Phi (\Phi^1 G))$ in $s$ would have only constant and $s^{-1}$ terms, and so $\mu_{\bar{\omega}}(\Phi^0 G) < \mu_{\bo}(E)$ still implies $\phi (G) \prec \phi (F)$.
       \end{itemize}
    \end{itemize}
\item[(b)] $\ch_1 (\image \alpha) =0$.  In this case, $\image (\alpha) \in \Coh^{\leq 1}(X)$, meaning $\Phi^0 (G)$ and $E$ have the same $\ch_0, \ch_1$, and so $\mu_{\bo}(\Phi^0 G) = \mu_{\bo}(E)$.   If we are in either of the following situations
    \begin{itemize}
    \item $\dimension \Phi (\Phi^1 G)=3$;
    \item $\dimension \Phi (\Phi^1 G)=2$ and $\ch_{11}(\Phi (\Phi^1 G))<0$,
    \end{itemize}
     then $\oZw (G), \oZw (F)$ would have the same $s^2$ coefficient, but the $s$ coefficient of $\oZw (G)$ would be strictly smaller than that of $\oZw (F)$, giving us  $\phi (G) \prec \phi (F)$.  On the other hand, if we are in either of the following situations
     \begin{itemize}
     \item $\dimension \Phi (\Phi^1 G)=2$ and $\ch_{11} (\Phi (\Phi^1 G))=0$;
     \item $\dimension \Phi (\Phi^1 G)=1$,
     \end{itemize}
     then  $\ch (\Phi^1 G) = \mbox{\tiny $\begin{pmatrix} 0 & 0 & \ast \\ 0 & \ast & \ast \end{pmatrix}$}$, which implies $\Phi^1 G \in \Coh^{\leq 1}(X)$ by \cite[Lemma 5.10]{LZ2}.  Then $\Phi^0 (F/G) \in \Coh^{\leq 1}(X)$ also, and $\ch(F/G)=\mbox{\tiny $\begin{pmatrix} 0 & \ast & \ast \\ 0 & 0 & \ast \end{pmatrix}$}$.  Recall that we have the torsion triple $ \left( \Fl [1], W_{0,X}, W_{1,X} \cap \Tl \right)$ in $\Bl$  from \eqref{eq19}.  Since $\Phi^1 (F/G)=0$, the $W_{1,X}\cap \Tl$-component of $F/G$ with respect to this torsion triple must be zero.  On the other hand, $\ch_{10}(F/G)=0$, so both the $(\Fl[1])$- and $W_{0,X}$-components of $F/G$ satisfy $\ch_{10}=0$, meaning the $(\Fl[1])$-component of $F/G$ is a torsion-free sheaf shifted by 1, while the $W_{0,X}$-component of $F/G$ is a torsion sheaf in $\Coh (\pi)_{\leq 1}$.  Then, since $F/G$ has rank zero, its $(\Fl[1])$-component must be zero.  Thus $F/G$ is a $\Phi$-WIT$_0$ torsion sheaf.  Then $\Phi^0 (F/G)$ is a $\Phi$-WIT$_1$ torsion sheaf supported in dimension at most 1, and  \cite[Lemma 3.15]{Lo11} implies $\Phi^0(F/G)$ is a fiber sheaf, i.e.\ $F/G$ itself is a $\Phi$-WIT$_0$ fiber sheaf.  Thus $\phi (F/G)=\tfrac{1}{2}$ for all $s>0$.  On the other hand, $\ch_{10}(G)=\ch_{10}(F)>0$, and so $\phi (G) \to 0$ as $s \gg 0$.  Thus $\phi (G) \prec \phi (F/G)$.
\end{itemize}

Now suppose $\rank (\image \alpha)>0$.  If $\Phi^0 G \neq 0$, then $0 < \rank (\Phi^0 G) < \rank E$, and we have $\mu_{\bo}(\Phi^0 G) < \mu_{\bo}(E)$.  The same argument as in part (a) above then shows $\phi (G) \prec \phi (F)$.  For the rest of the proof, suppose that $\Phi^0 G =0$, so we have the short exact sequence of sheaves
\begin{equation}\label{eq55}
0 \to E \to \Phi^0 (F/G) \to \Phi^1 G \to 0.
\end{equation}
From the torsion triple \eqref{eq19}, we must have $G \in W_{1,X} \cap \Tl$; in particular, $G \in \Coh (X)$. Lemma \ref{lemma7} then gives a 2-step filtration $G_0 \subseteq G_1 = G$ of $G$ in $\Coh (X)$ where $G_0 \in W_{1,X} \cap \Coh^{\leq 2}(X)$ and $G_1/G_0 \in \Tc^{l,0}$.  In particular, $\ch_{10}(G)=0$.  Writing $\ch(G)=(g_{ij})$, we have
\[
  \oZw (G) = 2h\alpha g_{01} + i (tg_{02} + 2hsg_{11}-h\alpha sg_{00}).
\]
We divide into three cases:
\begin{itemize}
\item $G_1/G_0 \neq 0$: from case (v) of Section \ref{phase-computation} and knowing $\ch_{11}(G_0) \leq 0$, we know $2hg_{11}-h\alpha g_{00} <0$, forcing $\phi (G) \to -\tfrac{1}{2}$ and so $\phi (G) \prec \phi (F)$.
\item $G_1/G_0=0$ and $\ch_{11}(G_0)<0$: then $G=G_0$,  $\ch_{01}(G) >0$ (if $\ch_{01}(G)=0$, then $\ch_{11}(G_0)$ must be nonnegative by \cite[Lemma 5.9]{LZ2}).  Hence  $G \in \scalea{\gyoung(;;+;*,;;-;*)}$ and $\phi (G) \to -\tfrac{1}{2}$ from Table \ref{table1}.  Then $\phi (G) \prec \phi (F)$.
\item $G_1/G_0=0$ and $\ch_{11}(G_0) =0$: then $G=G_0$ and $\ch (\wh{G}) = \mbox{\tiny $\begin{pmatrix} 0 & 0 & \ast \\ 0 & \ast & \ast \end{pmatrix}$}$, and so $\wh{G} \in \Coh^{\leq 1}(X)$ by \cite[Lemma 5.10]{LZ2}.  That is, $\wh{G}=\Phi^1 G \in W_{0,X} \cap \Coh^{\leq 1}(X)$.  The assumption \eqref{eq53} then implies $0=\Hom (\wh{G},E[1])=\Hom (G, \Phi E [1])=\Hom (G,F)$, contradicting $G \neq 0$.  Hence this case cannot occur.
\end{itemize}
This completes the proof that $F$ is $\nu^l$-stable as an object of $\Bl$.
\end{proof}

\subsection{Proof of Theorem \ref{theorem2}(B)}

We now proceed to prove Theorem \ref{theorem2}(B) via a series of lemmas.  The first thing that we need is:

\begin{lemma}\label{lemma30}
$\Coh^{\leq 0}(X)$, $\Coh (\pi)_0$ and $\Coh^{\leq 1}(X)$ are Serre subcategories of $\Bl$.
\end{lemma}

\begin{proof}
Consider the case of $\Coh^{\leq 1}(X)$ first.  Take any $\Bl$-short exact sequence
\[
0 \to A \to E \to G \to 0
\]
where $E \in \Coh^{\leq 1}(X)$  and the associated long exact sequence of sheaves
\[
0 \to H^{-1}(G) \to H^0(A) \overset{\alpha}{\to} E \to H^0(G) \to 0.
\]
Since $E$ satisfies $\ch_{10}=0$, Remark \ref{remark14}(iv) implies that $A, G$ both satisfy $\ch_{10}=0$ and $H^i(A), H^i(G)$ satisfy $\ch_{10}=0$ for all $i$.  In particular, we have $\ch_{10}(H^{-1}(G)) = 0 =\ch_{10}(H^0(A))$.  The vanishing $\ch_{10}(H^{-1}(G))=0$ means $H^{-1}(G) \in \Fc^{l,0}$, and so $\mu^\ast (H^{-1}(G)) \leq 0$ if $H^{-1}(G)$ is nonzero.  Since $\ch_{01} (\image \alpha)=0$, we have $\ch_{01}(H^0(A)) \leq 0$.  Now, $\ch_{10}(H^0(A))=0$ also implies $H^0(A)$ is the extension of a sheaf in $\Tc^{l,0}$ by a sheaf in $\Coh (\pi)_{\leq 1}$.  Since $\ch_{01}(H^0(A)) \leq 0$, the $\Tc^{l,0}$ component of $H^0(A)$ must be zero, i.e.\ $H^0(A) \in \Coh (\pi)_{\leq 1}$.  This forces $H^{-1}(G)$ to be zero since every nonzero sheaf in $\Fc^l$ is torsion-free.  Hence $A=H^0(A)$ is a subsheaf of $E$ and lies in $\Coh^{\leq 1}(X)$.  Also, $G=H^0(G)$ lies in $\Coh^{\leq 1}(X)$ since it is a sheaf quotient of $E$.  That is, $\Coh^{\leq 1}(X)$ is a Serre subcategory of $\Bl$.

The argument above still applies when we replace  $\Coh^{\leq 1}(X)$ with either $\Coh^{\leq 0}(X)$ or $\Coh (\pi)_0$.
\end{proof}

\begin{lemma}\label{lemma5}
The category
\begin{equation}\label{eq16}
 \langle \Coh^{\leq 1}(X), \Fc^{l,0}[1] \rangle
\end{equation}
 is a torsion class in the heart $\Bl$.
\end{lemma}

\begin{proof}
Let us write $\mathcal C$ to denote the category \eqref{eq16} in this proof.  We follow the argument in the proof of \cite[Lemma 2.16]{Toda1}, and show that $(\CC, \CC^\circ)$ satisfies the two axioms of a torsion pair.  Since $\Hom (\CC , \CC^\circ)=0$ follows from the definition of $\CC^\circ$,  we only have to check that every object $E \in \Bl$ fits in a short exact sequence in $\Bl$ of the form
\[
 0 \to A \to E \to B \to 0
\]
where $A \in \CC$ and $B \in \CC^\circ$.

Take any object $E \in \Bl$.  Let $F \subseteq H^{-1}(E)$ be the $\mu_f$-HN factor of $H^{-1}(E)$ with $\mu_f=0$, i.e.\ $F$ is the left-most $\mu_f$-HN factor of $H^{-1}(E)$.  Then we have a short exact sequence in $\Bl$
\[
0 \to F[1] \to E \to E' \to 0
\]
and $\Hom (A[1], E')=0$ for any $\mu_f$-semistable sheaf $A$ with $\mu_f(A)=0$.  Now, if $E' \notin \CC^\circ$, then there exists some $U \in \Coh^{\leq 1}(X)$ such that $\Hom (U,E') \neq 0$.  Since $\Coh^{\leq 1}(X)$ is closed under quotients in $\Bl$ by Lemma \ref{lemma30}, we can further assume $U \subseteq E'$ in $\Bl$.

We claim that we can choose a maximal such $U$: suppose there is an ascending chain in $\Bl$
\[
  U_1 \subseteq U_2 \subseteq \cdots U_m \subseteq \cdots \subseteq E'
\]
where each $U_i \in \Coh^{\leq 1}(X)$.  For each $i \geq 1$, consider the short exact sequence in $\Bl$
\[
  0 \to U_i \to E' \overset{\alpha_i}{\to} G_i \to 0.
\]
These in turn give rise to the short exact sequences in $\Bl$
\begin{equation}\label{eq14}
  0 \to U_{i+1}/U_{i} \to G_i \to G_{i+1} \to 0
\end{equation}
from which we obtain the inclusions in $\Coh (X)$
\begin{equation}\label{eq13}
  H^{-1}(E') \subseteq H^{-1}(G_1) \subseteq H^{-1}(G_2) \subseteq \cdots \subseteq H^{-1}(G_m) \subseteq \cdots \subseteq H^{-1}(E')^{\ast \ast}.
\end{equation}
The existence of the last inclusion \eqref{eq13} can be seen as follows: for each $i$, the cokernel of the injection $H^{-1}(\alpha_i)$ in $\Coh (X)$ lies in $\Coh^{\leq 1}(X)$, and so taking double dual gives $H^{-1}(G_i)^{\ast \ast} \cong H^{-1} (E')^{\ast \ast}$ for each $i$.

Since $\Coh (X)$ is a Noetherian abelian category, the ascending chain of $H^{-1}(G_i)$ in \eqref{eq13} must stabilise.  Suppose $m$ is such that $H^{-1}(G_i) = H^{-1}(G_{i+1})$ for $i \geq m$.  Then we have the exact sequence in $\Coh (X)$
\[
 0 \to \cokernel (H^{-1}(\alpha_i)) \to U_i \to H^0(E')
\]
which gives $U_i/\cokernel (H^{-1}(\alpha_i)) \hookrightarrow H^0(E')$.  Since $\cokernel (H^{-1}(\alpha_i))$ is constant for $i \geq m$, we see that the $U_i$ in \eqref{eq14} must stabilise.

Let $U \in \Coh^{\leq 1}(X)$ be maximal such that $U\subseteq E'$ in $\Bl$, and let $\beta$ denote the  composition of surjections in $\Bl$
\begin{equation}\label{eq15}
  E \twoheadrightarrow E' \twoheadrightarrow E'/U.
\end{equation}
Then we have a short exact sequence in $\Bl$
\[
  0\to F' \to E \overset{\beta}{\to} E'/U \to 0
\]
for some $F'$.  Applying the octahedral axiom to the surjections in \eqref{eq15}, we see right away that $H^{-1}(F') \cong F$ and $H^0(F') \cong U$, i.e.\ $F' \in \mathcal{C}$.
On the other hand, we have $\Hom (\Coh^{\leq 1}(X),E'/U)=0$ by the maximality of $U$.  Also, from the $\Bl$-short exact sequence
\[
0 \to U \to E' \to E'/U \to 0
\]
we have the exact sequence of sheaves
\[
0 \to H^{-1}(E') \to H^{-1}(E'/U) \to U.
\]
Since $U \in \Coh^{\leq 1}(X)$ and $H^{-1}(E') \in \Fc^{l,-}$ by construction, we also have $H^{-1}(E'/U) \in \Fc^{l,-}$.  Hence $\Hom (\Fc^{l,0}[1],E'/U)=0$.  Overall, we obtain $\Hom (\mathcal{C},E'/U)=0$, i.e.\ $E'/U \in \mathcal{C}^\circ$.  Thus every object $E$ in $\Bl$ can be written as the extension of an object $E'/U$ in $\mathcal{C}^\circ$ by an object $F'$ in $\mathcal{C}$, proving the lemma.
\end{proof}

Let us define
\begin{align}
  \mathcal A_{\bullet} &= \Coh^{\leq 0}(X), \\
  \mathcal A_{\bullet,1/2} &=  \langle \Coh^{\leq 1}(X), W_{0,X} \cap \Coh (\pi)_{\leq 1}, \Fc^{l,0} [1] \rangle \notag\\
  &= \langle
\xymatrix @-2.3pc{
\scalea{\gyoung(;;;,;;;+)} & \scalea{\gyoung(;;;+,;;;+)} & \scalea{\gyoung(;;;*,;;+;*)} & \scalea{\gyoung(;;+;*,;;+;*)} \\
& \scalea{\gyoung(;;;+,;;;0)} & &  \\
& \scalea{\gyoung(;;;+,;;;-)} & &
}, \Fc^{l,0}[1]\rangle. \label{eq23}
\end{align}

\begin{remark}\label{remark16}
$\mathcal A_{\bullet,1/2}$ is the extension closure of all the objects in Table \ref{table1} where $\phi \to \tfrac{1}{2}$ as $s \to \infty$, while $\Ac_\bullet$ is precisely the category of all objects in $\Bl$ whose $\oZw$ is identically zero.  
\end{remark}

\begin{lemma}\label{lemma12}
The subcategory $\mathcal A_{\bullet,1/2}$ is closed under quotient in $\Bl$, and every object in this category satisfies $\phi \to \tfrac{1}{2}$ as $s \to \infty$.
\end{lemma}

\begin{proof}
Given Lemma \ref{lemma5}, in order to prove the first assertion in the lemma, we only need to show that any quotient of $W_{0,X} \cap \Coh (\pi)_{\leq 1}$ in $\Bl$ is again an object in the category \eqref{eq23}.

Take any $A \in W_{0,X} \cap \Coh (\pi)_{\leq 1}$, and consider any $\Bl$-short exact sequence of the form
\[
 0 \to A' \to A \to A'' \to 0.
\]
From the long exact sequence
\[
 0 \to H^{-1}(A'') \to H^0(A') \to H^0(A) \to H^0(A'') \to 0
\]
and the fact that $\Coh (\pi)_{\leq 1}$ is a Serre subcategory of $\Coh (X)$, we see that $H^{-1}(A'')$ and $H^0(A')$ have the same $\ch_0$ and $\ch_{10}$.  By Remark \ref{remark14}(iv), it follows  that either $H^{-1}(A'')=0$ or $\mu_f (H^{-1}(A'')) = \mu_f (H^0(A'))=0$.  In the latter case, $H^{-1}(A'')$ is a $\mu_f$-semistable sheaf with $\mu_f=0$.  Also, $H^0(A'')$ lies in $W_{0,X} \cap \Coh (\pi)_{\leq 1}$.  Hence $A''$ lies $\Bl$.

The second assertion of the lemma follows from Remark \ref{remark16}.
\end{proof}

\begin{lemma}\label{lemma15}
Suppose $F \in \Bl$ is a $\nu^l$-semistable object with $\ch_{10} \neq 0$.  Then
\begin{equation}\label{eq58}
  \Hom (\Coh^{\leq 2}(X), \Phi F)=0.
\end{equation}
If, in addition, $F = \Phi E [1]$ for some coherent sheaf $E$, then $E$ is a torsion-free sheaf.
\end{lemma}

\begin{proof}
Suppose $F \in \Bl$ is $\nu^l$-semistable and $\ch_{10}(F) \neq 0$.  The vanishing \eqref{eq58} is equivalent to
\begin{equation}\label{eq21}
  0 = \Hom ( (\Phi \Coh^{\leq 2}(X))[1], F ).
\end{equation}

Note that
\begin{align*}
  \Phi (\Coh^{\leq 2}(X))[1] &\subset \langle \Phi (W_{0,X} \cap \Coh^{\leq 2}(X)), \Phi (W_{1,X} \cap \Coh (\pi)_{\leq 1}) \rangle [1] \\
  &\subset \langle \{ E \in W_{1,X} : \ch_{10}(E)=0\}, (W_{0,X} \cap \Coh (\pi)_{\leq 1}) [-1]\rangle [1],
\end{align*}
and so
\[
  \Phi (\Coh^{\leq 2}(X))[1] \subset \langle \Bl [1], \Bl \rangle.
\]
Since $F \in \Bl$, we have $\Hom (\Bl [1], F)=0$.  Therefore, in order to show the vanishing \eqref{eq21}, we just need to check the following:
\begin{itemize}
\item[(i)] For any $G \in W_{1,X}$ satisfying $\ch_{10}(G)=0$, we have $\Hom_{\Bl} ( \mathcal{H}_{\Bl}^0(G[1]), F)=0$;
\item[(ii)] $\Hom_{\Bl} ( W_{0,X} \cap \Coh (\pi)_{\leq 1}, F)=0$.
\end{itemize}

 To show (i), take any $G \in W_{1,X}$ where $\ch_{10}(G)=0$.  Consider the $(\Tl, \Fl)$-decomposition of $G$ in $\Coh (X)$:
\[
0 \to G' \to G \to G'' \to 0.
\]
Then $\mathcal{H}_{\Bl}^0 (G[1])=G''[1]$.  Since $G'$ is a subsheaf of $G$, we have $G' \in W_{1,X}$, too.  Also, $G' \in \Tl$ by construction.  Thus $G' \in W_{1,X} \cap \Tl$. By Lemma \ref{lemma7}, we have $\ch_{10}(G')=0$. This forces $\ch_{10}(G'')=0$, i.e.\ $G'' \in \Fc^{l,0}$.  Then by Lemma \ref{lemma12}, for any morphism $\alpha : G''[1] \to F$ in $\Bl$ we have $\phi (\image \alpha) \to \tfrac{1}{2}$ as $s \to \infty$.  By the $\nu^l$-semistability of $F$, $\image \alpha$ must be zero.  Thus (i) holds.

 To show (ii), take any nonzero morphism $\beta: G \to F$ in $\Bl$ where $G \in W_{0,X} \cap \Coh (\pi)_{\leq 1}$.   From Lemma \ref{lemma12} we have  $\phi (\image \beta) \to \tfrac{1}{2}$ as $s \to \infty$, i.e.\ $\image (\beta)$ destabilises $F$, contradicting the $\nu^l$-semistability of $F$.

 Lastly, if $F=\Phi E[1]$, then $\Phi F = E$, and $E$ must be a torsion-free sheaf.
\end{proof}

As part of Theorem \ref{theorem2} (B), we need to understand the following question: given an object $E \in \Bl$ with $\ch_{10}(E) \neq 0$, when do we have $\Phi (E) \in \Coh (X)$?  We have  the torsion triple \eqref{eq19} in $\Bl$
\[
  (\Fl [1], W_{0,X}, W_{1,X} \cap \Tl),
\]
and we noted in Remark \ref{remark15} that for any $E \in \Bl$, we have $\Phi (E) \in \Coh (X)$ if and only if $E \in \langle \Fl [1], W_{0,X}\rangle$, which is equivalent to either of the following two conditions:
\begin{itemize}
\item[(a)] The $\Phi$-WIT$_1$ component of $H^0(E)$ is zero.
\item[(b)] $\Hom (E,W_{1,X}\cap \Tl)=0$.
\end{itemize}

\begin{lemma}\label{lemma23}
Let $E \in \Bl$ be a $\nu^l$-semistable object with $\ch_{10}(E)>0$.  Then the $\Phi$-WIT$_1$ part of $H^0(E)$ has $\ch$ of the form {\tiny $\begin{pmatrix} 0 & \ast & \ast \\ 0 & 0 & \ast \end{pmatrix}$}.
\end{lemma}

In other words, if we let $G$ denote the $\Phi$-WIT$_1$ part of $H^0(E)$, then $\wh{G} \in \Coh^{\leq 1}(X)$, and $G$ is a vertical sheaf in the sense of Diaconescu \cite{Dia15}.

\begin{proof}
Since $\ch_{10}(E)>0$, we have $\phi (E) \to 0$ as $s \to \infty$.  Let $G$ denote the $\Phi$-WIT$_1$ part of $H^0(E)$.  Then the canonical morphism of sheaves $H^0(E) \to G$ coming from the $(W_{0,X},W_{1,X})$-decomposition of $H^0(E)$ is a $\Bl$-surjection - this follows from having the torsion triple \eqref{eq19}, and noting that  $G$ necessarily lies in $W_{1,X} \cap \Tl$.  Pre-composing this  with the canonical map $E \to H^0(E)$, we obtain  a surjection $E \twoheadrightarrow G$ in $\Bl$.

From Lemma \ref{lemma7}, we know $G$ has a 2-term filtration in $\Coh (X)$: $G_0 \subseteq G_1 = G$ where $G_0$ is a $\Phi$-WIT$_1$ torsion sheaf and $G_1/G_0 \in \Fc^{l,0}$.   If $G_1/G_0 \neq 0$, then $G_1/G_0$ would be a $\nu^l$-destabilising quotient of $E$  in $\Bl$ since  $\phi (G_1/G_0) \to -\tfrac{1}{2}$ from Table \ref{table1}.  Hence $G \in W_{1,X} \cap \Coh (\pi)_{\leq 1}$.  Then  $\ch_{11}(G) \leq 0$ from \cite[Remark 5.17]{LZ2}.  However, if $\ch_{11}(G)<0$, then $\phi (G)\to -\tfrac{1}{2}$, again destabilising $E$. Thus $\ch_{11}(G)=0$, and the lemma is proved.
\end{proof}

\begin{proof}[Proof of Theorem \ref{theorem2}(B)]
That $\Phi (F') \in \Coh (X)$ is clear.  To show that $\Phi (F')$ is a torsion-free sheaf, it suffices to show the vanishing $\Hom (\Coh^{\leq 2}(X), \Phi (F'))=0$.  From the proof of Lemma \ref{lemma15} (replacing $F$ in that proof by $F'$ in the statement of Theorem \ref{theorem2}(B)), we see that if this vanishing does not hold, then $F'$ would have a nonzero subobject in $\Bl$ with $\phi \to \tfrac{1}{2}$ as $s \to \infty$, which would then destabilise $F$.  Hence $\Phi (F')$ is a torsion-free sheaf.

Next, we show that $\Phi (F')$ is $\mu_{\bo}$-semistable: take any short exact sequence of sheaves on $X$
\[
0 \to B \to \Phi (F') \to C \to 0
\]
where $B, C$ are both torsion-free.  By Lemma \ref{lemma9}, $\Phi[1]$ takes this short exact sequence to an exact sequence in $\Bl$
\[
0 \to \Phi B[1] \to F' \to \Phi C [1] \to 0.
\]
Thus $\Phi B[1]$ is a subobject of $F$ in $\Bl$.  By the $\nu^l$-semistability of $F$, we have $\phi (\Phi B[1]) \preceq \phi (F)$.  This implies $\mu_{\bo}(B) \leq \mu_{\bo}(\Phi F)$.  On the other hand, $\Phi F''[1]$ is a $\Phi$-WIT$_0$ sheaf in $\Coh^{\leq 1}(X)$ by Lemma \ref{lemma23}.  Hence $\mu_{\bo}(\Phi F) = \mu_{\bo}(\Phi F')$, and we have $\mu_{\bo}(B) \leq \mu_{\bo}(\Phi F')$ overall.  Hence $\Phi F'$ is a $\mu_{\bo}$-semistable sheaf, and the proof is complete.
\end{proof}

\section{The Harder-Narasimhan property of limit tilt stability}\label{section-HNlimittilt}

We establish the Harder-Narasimhan (HN) property of the polynomial stability function $\oZw$  on  $\Bl$ in this section.  We will follow Toda's approach  \cite[Theorem 2.29(i)]{Toda1}, where he showed the HN property for  `limit stability' (also a polynomial stability).   In his approach, Toda begins with a torsion pair, which he then refines to the HN filtration. In our case, we begin with a torsion triple.

\subsection{Torsion classes in $\Bl$}

We begin by identifying two torsion classes in $\Bl$, which will give us a torsion triple in $\Bl$ that acts as an estimate of the HN filtration for $\nu^l$-stability on $\Bl$.

\begin{lemma}\label{lemma16}
We have
\begin{itemize}
\item[(a)] $\Ac_\bullet$ and $\Coh (\pi)_0$ are both  Serre subcategories of $\Bl$ and torsion classes in $\Bl$.
\item[(b)] The category $\mathcal A_{\bullet,1/2}$ (defined in \eqref{eq23}) is a torsion class in $\Bl$.
\end{itemize}
\end{lemma}

\begin{proof}
(a) We already proved in Lemma \ref{lemma30} that $\Ac_\bullet$ and $\Coh (\pi)_0$ are Serre subcategories of $\Bl$.  The proofs for their being torsion classes in $\Bl$ are similar to that for $\Ac_{\bullet,1/2}$ below.

\smallskip

(b) We already know $\Ac_{\bullet,1/2}$ is closed under quotient in $\Bl$ from Lemma \ref{lemma12}.  For any $E \in \Bl$, let
\[
0 \to F[1] \to E \to E' \to 0
\]
be the short exact sequence in $\Bl$ such that $F$ is the $\mu_f$-HN factor of $H^{-1}(E)$ with $\mu_f=0$.  Then $H^{-1}(E') \in \Fc^{l,-}$ and $\Hom (\Fc^{l,0}[1],E')=0$.

If $E' \notin \Ac_{\bullet,1/2}^\circ$, then there must exist an object $U \in \langle \Coh^{\leq 1}(X), W_{0,X} \cap \Coh (\pi)_{\leq 1}\rangle$ and a nonzero morphism $\beta : U \to E'$ in $\Bl$.  Since $\Ac_{\bullet,1/2}$ is closed under $\Bl$-quotients and $\Hom (\Fc^{l,0}[1],E')=0$ from the previous paragraph, we can replace $U$ by $\image \beta$ and assume that $\beta$ is a $\Bl$-injection.

Suppose we have an ascending chain
\[
 U_1 \subseteq U_2 \subseteq \cdots \subseteq U_m \subseteq \cdots \subseteq E'
\]
in $\Bl$ where $U_i \in \langle \Coh^{\leq 1}(X), W_{0,X} \cap \Coh (\pi)_{\leq 1}\rangle$ for all $i$.  We will show that this chain stabilises.  Let $G_i$ denote the cokernel of the inclusion $U_i \subseteq E'$; then we have the following short exact sequences in $\Bl$ for each $i$:
\begin{equation}\label{eq40}
0 \to U_i \to E' \overset{\alpha_i} \to G_i \to 0
\end{equation}
and
\begin{equation}\label{eq41}
0 \to U_{i+1}/U_i \to G_i \to G_{i+1} \to 0.
\end{equation}
We also have
\begin{equation}\label{eq42}
0 \to U_i \to U_{i+1} \to U_{i+1}/U_i \to 0
\end{equation}
for each $i$.   Since  \eqref{eq42} is a short exact sequence of sheaves, it is easy to see  $U_{i+1}/U_i \in \langle \Coh^{\leq 1}(X), W_{0,X} \cap \Coh (\pi)_{\leq 1}\rangle$ for each $i$.  Also, from \eqref{eq40}, \eqref{eq42} and recalling $\Phi \Bl \subset D^{[0,1]}_{\Coh (X)}$, we obtain an ascending chain of  sheaves
\[
\Phi^0U_1 \subseteq \Phi^0 U_2 \subseteq \cdots \subseteq \Phi^0 E',
\]
and so the $\Phi^0 U_i$ must stabilise.  Therefore, omitting a finite number of terms if necessary, we can now assume that all the $\Phi^0 U_i$ are isomorphic.  As a result, we have the long exact sequence of sheaves from \eqref{eq42}:
\begin{equation}\label{eq43}
0 \to \Phi^0 (U_{i+1}/U_i) \to \Phi^1 U_i \to \Phi^1 U_{i+1} \to \Phi^1 (U_{i+1}/U_i) \to 0.
\end{equation}

Notice that every object $E$ in either $\Coh^{\leq 1}(X)$ or $W_{0,X} \cap \Coh (\pi)_{\leq 1}$ satisfies the following property:
\[
  \textnormal{The WIT$_1$ part of $E$ is a fiber sheaf.}
\]
Since this property is preserved under extension of sheaves, every object in the extension closure $\langle \Coh^{\leq 1}(X), W_{0,X} \cap \Coh (\pi)_{\leq 1}\rangle$ also satisfies this property.  As a result, all the terms in \eqref{eq43} are fiber sheaves.  In particular, $\Phi^j (U_{i+1}/U_i)$ is a fiber sheaf for all $j$, and so  $U_{i+1}/U_i$ itself is a fiber sheaf.

From \eqref{eq41}, we have the long exact sequence of coherent sheaves
\begin{equation}\label{eq44}
0 \to H^{-1}(G_i) \to H^{-1}(G_{i+1}) \to U_{i+1}/U_i \to H^0(G_i) \to H^0(G_{i+1}) \to 0
\end{equation}
which yields $H^{-1}(G_i)^{\ast \ast} \cong H^{-1}(G_{i+1})^{\ast \ast}$ since $U_{i+1}/U_i \in \Coh^{\leq 1}(X)$.  Thus  the ascending chain of sheaves  $H^{-1}(G_i)$ stabilises since they are all subsheaves of $H^{-1}(G_i)^{\ast \ast}$, which is independent of $i$.  On the other hand, \eqref{eq41} also gives a chain of surjections of sheaves
\[
  H^0(G_i) \twoheadrightarrow H^0(G_{i+1}) \twoheadrightarrow \cdots.
\]
Thus the $H^0(G_i)$ also stabilise, and from \eqref{eq44} we see that $U_{i+1}/U_i =0$ for $i \gg 0$, i.e.\  the $U_i$ stabilise as well.  Therefore, we can pick a maximal subobject $U$ of $E'$ lying in $\langle \Coh^{\leq 1}(X), W_{0,X} \cap \Coh (\pi)_{\leq 1}\rangle$.

Continuing as in the proof of Lemma \ref{lemma5}, we consider the surjections
\[
  E \twoheadrightarrow E' \twoheadrightarrow E'/U
\]
in $\Bl$.  Applying the octahedral axiom to these surjections gives the diagram
\[
\scalebox{0.8}{\xymatrix{
   & & & F[2] \ar[ddd] \\
  & & & \\
  & E' \ar[dr] \ar[uurr] & & \\
  E \ar[ur] \ar[rr] & & E'/U \ar[r] \ar[dr] & M [1] \ar[d] \\
  & & & U [1]
}
}
\]
where $M = \kernel (E \to E'/U)$ and all straight lines are exact triangles.  From the $\Bl$-short exact sequence
\[
0 \to F[1] \to M \to U \to 0
\]
we see that $H^{-1}(M) \cong F$ and $H^0(M) \cong U$.

Now, $E'/U$ lies in $\Ac_{\bullet,1/2}^\circ$ because:
\begin{itemize}
\item $\Hom (\langle \Coh^{\leq 1}(X), W_{0,X} \cap \Coh (\pi)_{\leq 1}\rangle, E'/U)=0$ by the maximality of $U$.
\item $\Hom (\Fc^{l,0}[1], E'/U)=0$.  To see this, consider the short exact sequence $0 \to U \to E' \to E'/U \to 0$ in $\Bl$ and the exact sequence of sheaves it yields:
    \[
    0 \to H^{-1}(E') \to H^{-1}(E'/U) \to U.
    \]
    Since $U \in \Coh (\pi)_{\leq 1}$ (i.e.\ $U$ is a sheaf that vanishes on the generic fiber of $\pi$) and $H^{-1}(E') \in \Fc^{l,-}$ by construction, we have $H^{-1}(E'/U) \in \Fc^{l,-}$ as well, which gives $\Hom (\Fc^{l,0}[1], E'/U)=0$.
\end{itemize}
We have now proved that every object in $\Bl$ is the extension of an object in $\Ac_{\bullet,1/2}^\circ$ by an object in $\Ac_{\bullet,1/2}$, showing $\Ac_{\bullet,1/2}$ is a torsion class in $\Bl$.
\end{proof}

Let us define $\mathcal A_{1/2,0}$ to be the  extension closure
\begin{align}
  \mathcal A_{1/2,0} &= \langle \mathcal A_{\bullet,1/2}, \scalea{\gyoung(;;+;*,;;0;*)},  \scalea{\gyoung(;;*;*,;+;*;*)},\scalea{\gyoung(;+;*;*,;+;*;*)},\Fw^{l,-}[1]\rangle \notag\\
  &= \langle \Fl[1],
  \xymatrix @-2.3pc{
\scalea{\gyoung(;;;,;;;+)} & \scalea{\gyoung(;;;+,;;;+)} & \scalea{\gyoung(;;;*,;;+;*)} & \scalea{\gyoung(;;+;*,;;+;*)} & \scalea{\gyoung(;;*;*,;+;*;*)} &  \scalea{\gyoung(;+;*;*,;+;*;*)}  \\
& \scalea{\gyoung(;;;+,;;;0)} & & \scalea{\gyoung(;;+;*,;;0;*)} & & \\
& \scalea{\gyoung(;;;+,;;;-)} & & & &
}  \rangle. \label{eq49}
\end{align}
Note that, for every object in $\mathcal A_{1/2,0}$, we have $\phi \to \tfrac{1}{2}$ or $0$ as $s \to \infty$.  Also, $\mathcal A_{1/2,0}$ is very similar to the extension closure of all objects in Table \ref{table1} such that $\phi \to \tfrac{1}{2}$ or $\phi \to 0$ as $s \to \infty$, except that we use the category $\scalea{\gyoung(;+;*;*,;+;*;*)}$ instead of  $\Tw^{l,+}$ in the definition of $\mathcal A_{1/2,0}$.

\begin{lemma}\label{lemma17}
The category $\mathcal A_{1/2,0}$ is a torsion class in $\Bl$.
\end{lemma}

\begin{proof}
For the purpose of this proof, let us write
\[
  \mathcal E :=  \xymatrix @-2.3pc{
\scalea{\gyoung(;;;,;;;+)} & \scalea{\gyoung(;;;+,;;;+)} & \scalea{\gyoung(;;;*,;;+;*)} & \scalea{\gyoung(;;+;*,;;+;*)} & \scalea{\gyoung(;;*;*,;+;*;*)} &  \scalea{\gyoung(;+;*;*,;+;*;*)}  \\
& \scalea{\gyoung(;;;+,;;;0)} & & \scalea{\gyoung(;;+;*,;;0;*)} & & \\
& \scalea{\gyoung(;;;+,;;;-)} & & & &
},
\]
which is a torsion class in $\Coh (X)$ by Lemma \ref{lemma28}. Clearly $\mathcal{E} \subseteq \{H^0(E) : E \in \mathcal{A}_{1/2,0}\}$.  On the other hand, since $\mathcal{A}_{1/2,0}$ is defined to be the extension closure of $\Fl [1]$ and $\mathcal{E}$, and $\mathcal{E}$ is a torsion class in $\Coh (X)$, we have $\{H^0(E) : E \in \mathcal{A}_{1/2,0}\} \subseteq \mathcal{E}$, and so
\[
  \mathcal E = \{ H^0(E) : E \in \mathcal A_{1/2,0}\}.
\]

 For any $E \in \Bl$, let $A$ denote the maximal subsheaf of $H^0(E)$ in $\mathcal E$, and let $E''$ denote the preimage of $A$ under the canonical $\Bl$-quotient $E \twoheadrightarrow H^0(E)$.  Then $H^{-1}(E'')=H^{-1}(E)$ while $H^0(E'')=A$, i.e.\ $E'' \in \mathcal A_{1/2,0}$.  Thus we have the following commutative diagram of short exact sequences in $\Bl$
 \[
 \xymatrix{
   0 \ar[r] & H^{-1}(E)[1] \ar[r] & E \ar[r] & H^0(E) \ar[r] & 0 \\
   0 \ar[r] & H^{-1}(E)[1] \ar@{=}[u] \ar[r] & E'' \ar[r] \ar[u] & A \ar[u] \ar[r] & 0
   }
 \]
 where the middle and the right vertical arrows are injections in $\Bl$ (since the cokernel of the sheaf injection $A \hookrightarrow H^0(E)$ lies in $\Tl$, the morphism $A \to H^0(E)$ is also a $\Bl$-injection).

 On the other hand, $E' := H^0(E)/A$ lies in $\mathcal E^\circ = \{ F \in \Coh (X): \Hom (\mathcal{E},F)=0\}$.  Also, applying the octahedral axiom to the composition of $\Bl$-injections $H^{-1}(E)[1] \to E'' \to E$ gives $E/E'' \cong H^0(E)/A$.

 Overall,  we have shown that every object $E \in \Bl$ can be written as the extension of an object  in $\mathcal{E}^\circ$ by an object in  $\mathcal A_{1/2,0}$.  Since $\Hom (\mathcal A_{1/2,0}, \mathcal E^\circ)=0$, $\mathcal{A}_{1/2,0}$ is indeed a torsion class in $\Bl$, and $\mathcal E^\circ$ is the corresponding torsion-free class in $\Bl$.
\end{proof}

\subsection{Finiteness properties}

Since $\Ac_\bullet, \Ac_{\bullet,1/2}$ and $\Ac_{1/2,0}$ are all torsion classes in $\Bl$ and we have the inclusions $\Ac_\bullet \subset \Ac_{\bullet,1/2} \subset \Ac_{1/2,0}$, we have the  following torsion quadruple in $\Bl$
\begin{equation}\label{eq35}
  \left( \Ac_\bullet,\,\,\, \mathcal A_{\bullet,1/2} \cap \Ac_\bullet^\circ, \,\,\, \mathcal A_{1/2,0} \cap  \mathcal A_{\bullet,1/2}^\circ, \,\,\, \mathcal A_{1/2,0}^\circ \right).
\end{equation}
That is, every object $E \in \Bl$ has a filtration
\[
  E_0 \subseteq E_1 \subseteq E_2 \subseteq E_3 = E
\]
in $\Bl$ where
\begin{itemize}
\item $E_0 \in \Ac_\bullet$,
\item $E_1/E_0 \in  \mathcal A_{\bullet,1/2} \cap \Ac_\bullet^\circ$,
\item $E_2/E_1 \in \mathcal A_{1/2,0} \cap  \mathcal A_{\bullet,1/2}^\circ$, and
\item $E_3/E_2 \in  \mathcal A_{1/2,0}^\circ$.
\end{itemize}
This filtration can be constructed by first taking $E_2$ to be the $\mathcal A_{1/2,0}$ part of $E$,  then taking $E_1$ to be the $\mathcal A_{\bullet,1/2}$ part of $E_2$, and finally taking $E_0$ to be the $\Ac_\bullet$ part of $E_1$.

\begin{proposition}\label{pro1}
The following finiteness properties hold:
\begin{itemize}
\item[(1)] For $\Ac = \mathcal A_{\bullet,1/2}\cap \Ac_\bullet^\circ$:
\begin{itemize}
\item[(a)] There is no infinite sequence of strict monomorphisms in $\Ac$
\begin{equation}\label{eq25}
  \cdots \hookrightarrow E_n \hookrightarrow \cdots \hookrightarrow E_1 \hookrightarrow E_0
\end{equation}
where $\phi (E_{i+1}) \succ \phi (E_i)$ for all $i$.
\item[(b)] There is no infinite sequence of strict epimorphisms in $\Ac$
\begin{equation}\label{eq26}
  E_0 \twoheadrightarrow E_1 \twoheadrightarrow \cdots \twoheadrightarrow E_n \twoheadrightarrow \cdots.
\end{equation}
\end{itemize}
\item[(2)] For $\Ac = \Ac_{1/2,0} \cap \Ac_{\bullet,1/2}^\circ$:
    \begin{itemize}
    \item[(a)] There is no infinite sequence of strict monomorphisms \eqref{eq25} in $\Ac$.
    \item[(b)] There is no infinite sequence of strict epimorphisms \eqref{eq26} in $\Ac$.
    \end{itemize}
\item[(3)] For $\Ac = \Ac_{1/2,0}^\circ$:
  \begin{itemize}
  \item[(a)] There is no infinite sequence of strict monomorphisms \eqref{eq25} in $\Ac$.
  \item[(b)] There is no infinite sequence of strict epimorphisms \eqref{eq26} in $\Ac$.
  \end{itemize}
\end{itemize}
\end{proposition}

\begin{proof}
Let us fix some notations first.  In the proofs of (1)(a), (2)(a) and (3)(a), we will consider the $\Bl$-short exact sequences
\begin{equation}\label{eq27}
  0 \to E_{i+1} \overset{\beta_i}{\to} E_i \to G_i \to 0.
\end{equation}
On the other hand, in the proofs of (1)(b), (2)(b) and (3)(b), we will consider the $\Bl$-short exact sequences
\begin{equation}\label{eq28}
  0 \to K_i \to E_i \to E_{i+1} \to 0.
\end{equation}
Since $\ch_{10}(-) \geq 0$ on $\Bl$ by Remark \ref{remark14}(iv), from \eqref{eq27} (resp.\ \eqref{eq28}) we know that $\ch_{10}(E_i)$ is a decreasing sequence of nonnegative integers when  we are proving part (a) (resp.\ part (b)) of (1), (2) or (3).  Therefore, by omitting a finite number of terms in \eqref{eq27} (resp.\ \eqref{eq28}) if necessary, we can assume that the   $\ch_{10}(E_i)$ are constant.  This also implies that $\ch_{10}(G_i)=0$ and $\ch_{10}(K_i)=0$ for all $i \geq 0$, which in turn implies $\ch_{10}(H^{-1}(G_i)[1]), \ch_{10}(H^0(G_i)), \ch_{10}(H^{-1}(K_i)[1])$ and $\ch_{10}(H^0(K_i))$ are all zero for all $i$.

\smallskip

Part (1): $\Ac = \mathcal A_{\bullet,1/2}\cap \Ac_\bullet^\circ$.  For (a), suppose we have an infinite sequence of strict monomorphisms in $\Ac$ as in \eqref{eq25}.  This gives the inclusions of sheaves $\alpha_i : H^{-1}(E_{i+1}) \hookrightarrow H^{-1}(E_i)$ for all $i$, and so $\rk (H^{-1}(E_i))$ must become constant for $i$ large enough.  In that case, $\cokernel (\alpha_i)$ is  a torsion sheaf and lies in $\Fl$, and so must be zero.  Hence $\alpha_i$ is an isomorphism for $i \gg 0$, i.e.\  the $H^{-1}(E_i)$ stabilise.

Consider the short exact sequences \eqref{eq27} in $\Bl$.  Since $\beta_i$ is a strict monomorphism in $\Ac$, the cokernel $G_i$ lies in $\Ac$.  Also,  $H^0(E_{i+1})$ is a torsion sheaf by the definition of $\Ac_{\bullet,1/2}$.   The inclusion $H^{-1}(G_i) \hookrightarrow H^0(E_{i+1})$ then forces $H^{-1}(G_i)$ to be zero, so we can assume $G_i = H^0(G_i)$ and we have the short exact sequence of sheaves
\[
 0 \to H^0 (E_{i+1}) \to H^0 (E_i) \to H^0 (G_i) \to 0.
\]

Recall that $\ch_{01} (-)\geq 0$ on $\Ac_{\bullet,1/2}$ by construction.  Hence $HD\ch_{01} (H^0(E_i))$ is a decreasing sequence of nonnegative integers, and must stabilise for $i \gg 0$, in which case $\ch_{01}(H^0(G_i))=0$.  From the definition of $\Ac_{\bullet,1/2}$ in \eqref{eq23}, we see that this forces  $H^0(G_i) \in \Coh^{\leq 1}(X)$, implying $\phi (G_i)=\tfrac{1}{2}$; however, this cannot happen under the assumption that $\phi (E_{i+1}) \succ \phi (E_i)$ for all $i$.  Hence such an infinite chain of strict monomorphisms \eqref{eq25} cannot exist.

For (b), suppose we have an infinite sequence of strict epimorphisms in $\Ac$ as in \eqref{eq26}.  The epimorphisms in \eqref{eq26} gives the surjections of sheaves
\[
  H^0(E_0) \twoheadrightarrow H^0(E_1) \twoheadrightarrow \cdots,
\]
which must  stabilise since $\Coh (X)$ is Noetherian. As a result,  we can assume that $H^0(E_i) \twoheadrightarrow H^0(E_{i+1})$ is an isomorphism for all $i \geq 0$ by omitting a finite number of terms.

Consider the short exact sequences \eqref{eq28} in $\Bl$ where $K_i \in \Ac_{\bullet,1/2}$ by assumption.  From  the long exact sequence of cohomology of \eqref{eq28}, we have
 \begin{equation}\label{eq32}
 0 \to H^{-1}(K_i) \to H^{-1}(E_i) \to H^{-1}(E_{i+1}) \to H^0(K_i) \to 0.
 \end{equation}
  Since $H^0(K_i)$ is a torsion sheaf, we have $\rk (H^{-1}(E_i)) \geq \rk(H^{-1}(E_{i+1}))$, and so $\rk (H^{-1}(E_i))$ must stabilise eventually, implying $\rk (H^{-1}(K_i))=0$ for $i \gg 0$, in which case $H^{-1}(K_i)=0$.  Therefore, for large enough $i$ we have the short exact sequence
  \begin{equation}\label{eq29}
  0 \to H^{-1}(E_i) \to H^{-1}(E_{i+1}) \to H^0(K_i) \to 0.
  \end{equation}

  As we observed in part (a), we have $\ch_{01}(-) \geq 0$ for all objects in $\Ac_{\bullet,1/2}$, and so from \eqref{eq28} we see $\ch_{01} (E_i)$ is decreasing and must stabilise.  That is, for $i \gg 0$, we have $K_i=H^0(K_i)$ and $\ch_{01}(K_i)=0$; these imply that $K_i=H^0(K_i) \in \Coh^{\leq 1}(X)$.  Hence the $H^{-1}(E_i)$ are all isomorphic in codimension one for $i \gg 0$ and  we have
  \[
    H^{-1}(E_i) \hookrightarrow H^{-1}(E_{i+1}) \hookrightarrow H^{-1}(E_{i+1})^{\ast \ast},
  \]
  where  $H^{-1}(E_{i+1})^{\ast \ast}$ is constant for $i \gg 0$.  And so $H^{-1}(E_i)$ must stabilise, and the $E_i$ themselves stabilise.

\smallskip

 Part (2): $\Ac = \Ac_{1/2,0} \cap \Ac_{\bullet,1/2}^\circ$.  For (a), consider an infinite sequence of strict monomorphisms in $\Ac$ as in \eqref{eq25}, and define the $G_i$ as in \eqref{eq27}.

From \eqref{eq27}, we have the exact sequence of sheaves
\[
0 \to H^{-1}(E_{i+1}) \to H^{-1}(E_i) \overset{\gamma_i}{\to} H^{-1}(G_i).
\]
By the same argument as in (1)(a), we can assume that the $H^{-1}(E_i)$ are constant.

Recall also that we can assume $\ch_{10}(H^{-1}(G_i))=0$ for all $i$, i.e.\ $H^{-1}(G_i) \in \Fc^{l,0}$ for all $i$.  On the other hand, $G_i  \in \Ac_{\bullet,1/2}^\circ$ by assumption.  This forces $H^{-1}(G_i)=0$, and so  $G_i = H^0(G_i)$ for all $i \geq 0$.

From \eqref{eq27}, we now have the exact sequence of sheaves
\begin{equation}\label{eq31}
0 \to H^0(E_{i+1}) \to H^0(E_i) \to H^0(G_i) \to 0
\end{equation}
for all $i$.  Since $\ch_{10}(G_i)=0$ and $G_i \in \Ac_{1/2,0}$, from the definition of $\Ac_{1/2,0}$ in \eqref{eq49} we see that $G_i$ cannot have any subfactor in $\scalea{\gyoung(;;*;*,;+;*;*)}$ or $\scalea{\gyoung(;+;*;*,;+;*;*)}$.  Hence $G_i \in \Coh(\pi)_{\leq 1}$.  Also, since $G_i \in \Ac_{\bullet,1/2}^\circ$, it follows that $G_i$ must be a pure 2-dimensional sheaf and must be $\Phi$-WIT$_1$.  Applying $\Phi$ to \eqref{eq31}, we obtain the long exact sequence of sheaves
\begin{equation*}
0 \to \Phi^0 (H^0(E_{i+1})) \to \Phi^0 (H^0(E_i)) \to 0
\to \Phi^1 (H^0(E_{i+1})) \to \Phi^1 (H^0(E_i)) \to \Phi^1 (H^0(G_i)) \to 0
\end{equation*}
where $\Phi^1 (H^0(E_{i+1})),  \Phi^1 (H^0(E_i)), \Phi^1 (H^0(G_i))$ all lie in $\Coh^{\leq 1}(X)$ by Lemma \ref{lemma20} below.  Since we have  $\ch_{11}(-) \geq 0$ on $\Coh^{\leq 1}(X)\cap W_{0,X}$, we know from the last four terms of this long exact sequence that  $D\ch_{11}(\Phi^1 (H^0(E_i)))$ is a decreasing sequence of nonnegative integers.  Hence the sequence $D\ch_{11}(\Phi^1 (H^0(E_i)))$ eventually becomes constant, forcing $\ch_{11}(\Phi^1 (H^0(G_i)))=0$.  This means that $G_i = H^0(G_i)$ is a $\Phi$-WIT$_1$ fiber sheaf.  Since we observed earlier that $G_i$ is a pure 2-dimensional sheaf, $G_i$ must be zero, and so the $E_i$ stabilise.

For (b), suppose we have an infinite sequence of strict epimorphisms in $\Ac$ as in \eqref{eq26}.  As in the proof of part (1)(b), we can assume that the $H^0(E_i)$ stabilise.  Let $K_i$ be as in \eqref{eq28}.  We want to show that the $H^{-1}(E_i)$ eventually stabilise as well.

Consider the long exact sequence of sheaves \eqref{eq32}.  Since we can assume $\ch_{10}(H^{-1}(K_i)[1])=0$, we have $H^{-1}(K_i) \in \Fc^{l,0}$, i.e.\ $H^{-1}(K_i)[1] \in \Ac_{\bullet,1/2}$.  However, $E_i \in \Ac_{\bullet,1/2}^\circ$ and we have the inclusions in $\Bl$
\[
  H^{-1}(K_i)[1] \hookrightarrow K_i \hookrightarrow E_i.
\]
Therefore, $H^{-1}(K_i)=0$  and $K_i=H^0(K_i) \in \Ac_{1/2,0}$ for all $i$.  Since $\ch_{10}(K_i)=0$, from the definition of $\Ac_{1/2,0}$ in \eqref{eq49}  we see that $K_i$ must be a torsion sheaf (since any subfactor of $K_i$ from $\scalea{\gyoung(;;*;*,;+;*;*)}$ or $\scalea{\gyoung(;+;*;*,;+;*;*)}$ would have a strictly positive contribution towards $\ch_{10}$).  In particular, $K_i \in \Coh (\pi)_{\leq 1}$.  However, $K_i \in \Ac_{\bullet,1/2}^\circ$, and so $K_i$ is also  pure 2-dimensional and is $\Phi$-WIT$_1$.

The long exact sequence \eqref{eq32} now reduces to
\[
0 \to H^{-1}(E_i) \to H^{-1}(E_{i+1}) \to H^0(K_i) \to 0
\]
where all the terms are $\Phi$-WIT$_1$, and so applying $\Phi[1]$ gives the short exact sequence of sheaves
\begin{equation}\label{eq33}
0 \to \wh{H^{-1}(E_i)} \to \wh{H^{-1}(E_{i+1})} \to \wh{H^0(K_i)} \to 0,
\end{equation}
where $\wh{H^0(K_i)} \in\Coh^{\leq 1}(X)$ by Lemma \ref{lemma20}.

 We claim that $\wh{H^{-1}(E_i)}$ is torsion-free for all $i$.  To see this, fix any $i$ and suppose $\wh{H^{-1}(E_i)}$ has a subsheaf $T$ lying in $\Coh^{\leq 2}(X)$.  Let $T_i$ denote the $\Phi$-WIT$_i$ part of $T$.  Then we have an injection of sheaves $T_0 \hookrightarrow \wh{H^{-1}(E_i)}$ the cokernel of which is $\Phi$-WIT$_0$, and so this injection is taken by $\Phi$ to the injection of sheaves $\wh{T_0} \hookrightarrow H^{-1}(E_i)$.  Note that $f\ch_1(\wh{T_0})=-\ch_0(T_0)=0$.  On the other hand, since $\wh{T_0}$ is a subsheaf of $H^{-1}(E_i)$, which lies in $\Fc^{l,-}$ (because $E_i \in \Ac_{\bullet, 1/2}^\circ$), it follows that $\wh{T_0} \in \Fc^{l,-}$ and  $f\ch_1(\wh{T_0})< 0$ if $\wh{T_0}$ is nonzero.  Thus $T_0$ must vanish, i.e.\ $T=T_1$ lies in $\Coh (\pi)_{\leq 1} \cap W_{1,\Phi}$.  Then $\wh{T} = \Phi T [1] \in \Coh(\pi)_{\leq 1} \cap W_{0,\Phi}$, which is contained in $\Ac_{\bullet, 1/2}$. Now
\begin{align*}
\Hom (T,\wh{H^{-1}(E_i)}) &\cong \Hom (\Phi T, \Phi \wh{H^{-1}(E_i)}) \\
&\cong \Hom (\wh{T}[-1],H^{-1}(E_i)) \\
&\cong \Hom (\wh{T},H^{-1}(E_i)[1]) \\
&=0
\end{align*}
where the last equality holds because  $E_i \in \Ac_{\bullet, 1/2}^\circ$.  Thus $T$ itself must be zero, i.e.\ $\wh{H^{-1}(E_i)}$ is torsion-free.

The short exact sequences \eqref{eq33} now give us inclusions
\[
  \wh{H^{-1}(E_i)} \hookrightarrow \wh{H^{-1}(E_{i+1})} \hookrightarrow (\wh{H^{-1}(E_i)})^{\ast \ast},
\]
where the $(\wh{H^{-1}(E_i)})^{\ast \ast}$ are constant because $\wh{H^0(K_i)} \in \Coh^{\leq 1}(X)$.  Thus the $\wh{H^{-1}(E_i)}$ stabilise, i.e.\ the $H^{-1}(E_i)$ stabilise for $i \gg 0$.

\smallskip

Part (3): $\Ac =  \Ac_{1/2,0}^\circ$.  For (a), consider an infinite sequence of strict monomorphisms as in \eqref{eq25}, and let $G_i$ be as in \eqref{eq27}.  By Remark \ref{remark8}, we have $E_i, G_i \in W_{1,X} \cap \Tl$ for all $i$.  Thus \eqref{eq25} is a sequence of strict monomorphisms in $W_{1,X} \cap \Tl$ where $\rk (E_i)$ is a decreasing sequence.  We can therefore assume that $\rk (E_i)$ is constant and $\rk (G_i)=0$ for all $i$; this means $G_i \in W_{1,X} \cap \Coh (\pi)_{\leq 1}$ by Lemma \ref{lemma7}.

Since every object in $W_{1,X} \cap \Tl$ is a $\Phi$-WIT$_1$ sheaf with $\ch_{10}=0$ (from the decomposition in Lemma \ref{lemma7}),  the short exact sequence \eqref{eq27} is taken by $\Phi [1]$ to the short exact sequence of torsion sheaves
\[
  0 \to \wh{E_{i+1}} \to \wh{E_i} \to \wh{G_i} \to 0.
\]
Since $\ch_{01} \geq 0$ for any torsion sheaf, the  $HD \ch_{01}(\wh{E_i})$ form a decreasing sequence of  nonnegative integers.  Hence  $\ch_{01}(\wh{E_i})$ becomes constant eventually, and we can assume $\ch_{01}(\wh{G_i})=0$.  Since we also know $G_i \in \Coh (\pi)_{\leq 1}$ from the previous paragraph, we have  $\wh{G_i} \in W_{0,X} \cap \Coh^{\leq 1}(X)$, i.e.\ $G_i \in \Phi (W_{0,X} \cap \Coh^{\leq 1}(X))$.  However, $\Phi (W_{0,X} \cap \Coh^{\leq 1}(X))$  is  contained in $\langle \Coh^{\leq 1}(X), \scalea{\gyoung(;;+;*,;;0;*)} \rangle$, which in turn is contained in $\Ac_{1/2,0}$.  Since $G_i \in \Ac_{1/2,0}^\circ$ by assumption, we must have $G_i=0$, i.e.\ the $E_i$ stabilise.

For (b), suppose we have an infinite sequence of strict epimorphisms in $\Ac$ as in \eqref{eq26}, and let $K_i$ be as in \eqref{eq28}.  Again by Remark \ref{remark8}, all the $E_i, K_i$ lie in $W_{1,X} \cap \Tl$, and so \eqref{eq28} is a short exact sequence of coherent sheaves.  Thus the $E_i$ must stabilise since $\Coh (X)$ is Noetherian.
\end{proof}

\begin{lemma}\label{lemma20}
For any $A \in \Ac_{1/2,0}$, let $A_1$ denote the $\Phi$-WIT$_1$ part of $H^0(A)$.  Then $\wh{ A_1} \in \Coh^{\leq 1}(X)$.
\end{lemma}

\begin{proof}
Recall that $\Fl \subset W_{1,X}$ and $\Phi \Bl \subset D^{[0,1]}_{\Coh (X)}$.  Given any $E \in \Bl$, applying $\Phi$ to the $\Bl$-short exact sequence $0 \to H^{-1}(E)[1] \to E \to H^0(E) \to 0$ and taking $\Coh (X)$-cohomology gives
\[
  \Phi^1 (E) \cong \Phi^1 (H^0(E))
\]
where $\wh{\Phi^1 (H^0(E))}$ is precisely the $\Phi$-WIT$_1$ part of $H^0(E)$.  As a result, the following property for objects $A$ in $\Bl$
\begin{equation}\label{eq30}
 \text{The $\Phi$-WIT$_1$ part $A_1$ of $H^0(A)$ is such that $\wh{A_1} \in \Coh^{\leq 1}(X)$}
\end{equation}
is preserved under extension.  Consequently, the lemma would follow if it holds for all  objects $A$ that lies in each of the categories used to generate $\Ac_{1/2,0}$ in its definition \eqref{eq49}, which is clear.
\end{proof}

\subsection{Construction of the HN filtration}

Let us now set
\begin{align*}
  \Ac_{1/2} &= \Ac_{\bullet,1/2} \cap \Ac_\bullet^\circ, \\
  \Ac_0 &= \Ac_{1/2,0} \cap \Ac_{\bullet,1/2}^\circ,\\
  \Ac_{-1/2} &= \Ac_{1/2,0}^\circ.
\end{align*}
The torsion quadruple \eqref{eq35} in $\Bl$ can now be denoted as
\begin{equation}\label{eq59}
  (\Ac_\bullet,\, \Ac_{1/2}, \, \Ac_0, \, \Ac_{-1/2}).
\end{equation}

\begin{lemma}\label{lemma22}
For $i=\tfrac{1}{2}, 0, -\tfrac{1}{2}$, and any $E \in \Ac_i$, we have $\phi (E) \to i$ as $s \to \infty$.
\end{lemma}

\begin{proof}
The case  $E \in \Ac_{1/2}$: from the definition of $\Ac_{\bullet,1/2}$ in \eqref{eq23}, it is clear that  $\phi (E) \to \tfrac{1}{2}$ when $s \to \infty$ for any $E \in \Ac_{1/2}$.

The case   $E \in \Ac_0$:  if $\ch_{10}(E)>0$, then we certainly have $\phi (E) \to 0$ as $s \to \infty$, so let us suppose $\ch_{10}(E)=0$.  Then $H^{-1}(E) \in \Fc^{l,0}$.  However, we also have $E \in  \Ac_{\bullet,1/2}^\circ$, and so $H^{-1}(E)$ must be zero and  $E=H^0(E)$.  Then,  from the definition of $\Ac_{1/2,0}$, we know that each $\mu_f$-HN factor of $E$ must be either a torsion sheaf or a $\mu_f$-semistable sheaf with $\mu_f >0$.  That $\ch_{10} (E)=0$ then forces $E$ to be a torsion sheaf in $\Coh (\pi)_{\leq 1}$.  However, $E \in \Ac_{\bullet,1/2}^\circ$, and so $E$ must be a pure 2-dimensional sheaf that is  $\Phi$-WIT$_1$.  By Lemma \ref{lemma20}, we know $\wh{E} \in \Coh^{\leq 1}(X)$, implying $\ch_{11}(E)=0$.  Hence $\phi (E) \to 0$ as $s \to \infty$.

The case $E \in \Ac_{-1/2}$: we have  $H^{-1}(E)=0$ since $\Fl [1] \subset \Ac_{1/2,0}$.  Also, $H^0(E) \in W_{1,X} \cap \Tl$ by Remark \ref{remark8}, and from  Lemma \ref{lemma7} we know $H^0(E)$ has a filtration in $\Coh (X)$
\begin{equation}\label{eq37}
      M_0 \subseteq M_1 = H^0(E)=E
\end{equation}
where $M_1/M_0 \in \Tc^{l,0}$ and $M_0 \in W_{1,X} \cap \Coh^{\leq 2}(X)$ is the maximal torsion subsheaf of $H^0(E)$.

Since $\Coh^{\leq 1}(X) \subset \Ac_{1/2,0}$, we can assume that $M_0$ is pure 2-dimensional.  Since $H^0(E)$ is $\Phi$-WIT$_1$, the same holds for $M_0$.  We claim that $\ch_{11}(M_0)<0$, i.e.\ $M_0 \in \scalea{\gyoung(;;+;*,;;-;*)}$.  For, if $\ch_{11}(M_0)=0$, then the argument in the last part of the proof of Lemma \ref{lemma24} shows that $M_0$ is the extension of a fiber sheaf by a sheaf in $\langle \scalea{\gyoung(;;;+,;;;0))}, \scalea{\gyoung(;;+;*,;;0;*)}\rangle$, i.e.\ $M_0 \in \Ac_{1/2,0}$.  However, the inclusion in \eqref{eq37} implies $M_0 \in \Ac_{1/2,0}^\circ$, and $M_0$ would have to be zero.  Hence $M_0 \in \scalea{\gyoung(;;+;*,;;-;*)}$.

From Table \ref{table1}, we know $\phi  \to -\tfrac{1}{2}$ as $s \to \infty$ for both $M_0$ and $M_1/M_0$.  Hence the same holds for $E$.
\end{proof}

\begin{lemma}\label{lemma21}
 An object $E \in \Ac_\bullet^\circ$ is $\nu^l$-semistable in $\Bl$ iff, for some $i=\tfrac{1}{2}, 0, -\tfrac{1}{2}$, we have:
 \begin{itemize}
 \item $E \in \Ac_i$;
 \item for any strict monomorphism $0 \neq E' \hookrightarrow E$ in $\Ac_i$, we have $\phi (E') \preceq \phi (E)$.
 \end{itemize}
\end{lemma}

\begin{proof}
From the torsion quadruple \eqref{eq59}, we know $\Ac_\bullet^\circ$ is the extension closure of $\Ac_{1/2}, \Ac_0$ and $\Ac_{-1/2}$; by  Lemma \ref{lemma22}, any $\nu^l$-semistable object in $\Ac_\bullet^\circ$ must lie in $\Ac_{1/2}, \Ac_0$ or $\Ac_{-1/2}$. For the rest of the proof, we divide into three cases:

Case 1:  $E$ is an object of $\Ac_{1/2}$.  The proof of this case is similar to that of Case 2 below, and so we omit it here.

Case 2:  $E$ is an object of $\Ac_0$.  Take any short exact sequence
\begin{equation}\label{eq36}
0 \to F \to E \to G \to 0.
\end{equation}
in $\Bl$ where $F, G \neq 0$.  First, we consider the $\left( \Ac_{1/2,0}, \Ac_{1/2,0}^\circ\right)$-decomposition of $F$ in $\Bl$:
\[
 0 \to F_1 \to F \to F_2 \to 0.
\]
Then $\phi (F_1) \succeq \phi (F) \succeq \phi (F_2)$ by Lemma \ref{lemma22}.  Since $F_1$ is a subobject of $E$ in $\Bl$, we also have $F_1 \in \Ac_{1/2}^\circ$.  Now, consider the short exact sequence from the injection $F_1 \hookrightarrow E$ in $\Bl$:
\[
 0 \to F_1 \to E \to N \to 0.
\]
We have $F_1 \in \Ac_{1/2,0} \cap \Ac_{1/2}^\circ$ and $N \in \Ac_{1/2,0}$.  Consider also the $(\Ac_{1/2}, \Ac_{1/2}^\circ)$-decomposition of $N$ in $\Bl$
\[
0 \to N_0 \to N \to N_1 \to 0,
\]
Applying the octahedral axiom to the composition of surjections in $\Bl$
\[
  E \twoheadrightarrow N \twoheadrightarrow N_1,
\]
we obtain the diagram
\[
\scalebox{0.8}{
\xymatrix{
  & & & F_1[1] \ar[ddd] \\
  & & & \\
  & N \ar[dr] \ar[uurr] & & \\
  E \ar[ur] \ar[rr] & & N_1 \ar[r] \ar[dr] & M [1] \ar[d] \\
  & & & N_0 [1]
}
}
\]
where $M$ is the kernel of $E \twoheadrightarrow N_1$ in $\Bl$, and all straight lines are exact triangles.  Since we have
\begin{itemize}
\item $F_1 \in  \Ac_{1/2,0} \cap \Ac_{1/2}^\circ$ and
\item $N_0 \in \Ac_{1/2} \subset \Ac_{1/2,0}$,
\end{itemize}
we see that $M \in \Ac_{1/2,0}$.  However, $M$ is a subobject of $E$ in $\Bl$, so $M \in \Ac_{1/2}^\circ$ also.  That is, $M \in \Ac_0$.  In addition, $N_1$ also lies in $\Ac_0$.  Hence the injection $M \hookrightarrow E$ in $\Bl$ is a strict monomorphism in $\Ac_0$.

Now, if $\phi (M) \preceq \phi (E)$ holds, then it follows that
\[
  \phi (M) \preceq \phi (E) \preceq \phi (N_1).
\]
On the other hand, since $N_0 \in \Ac_{1/2}$ and $N_1 \in \Ac_0$, we have $\phi (N_0) \succeq \phi (N) \succeq \phi (N_1)$.  Hence $\phi (E) \preceq \phi (N)$.  From the short exact sequence $0 \to F_1 \to E \to N\to 0$, we have $\phi (F_1) \preceq \phi (E)$.  It follows that $\phi (F) \preceq \phi (F_1) \preceq  \phi (E)$.  Since $\oZw (F) \neq 0$ (This is because $F$, being a $\Bl$-subobject of $E \in \Ac_0$, lies in $\Ac_{1/2}^\circ$.), from the short exact sequence \eqref{eq36}  we have $\phi (F) \preceq \phi (E) \preceq \phi (G)$.

That is, to check that $E$ is $\nu^l$-semistable, it suffices to check strict monomorphisms $M \hookrightarrow E$ in $\Ac_0$.

Case 3: $E$ is an object of $\Ac_{-1/2}$.  Take any short exact sequence in $\Bl$ as in \eqref{eq36}.  Consider the $(\Ac_{1/2,0}, \Ac_{1/2,0}^\circ)$-decomposition of $G$
\[
0 \to G_1 \to G \to G_2 \to 0.
\]
By Lemma \ref{lemma22}, we have $\phi (G_1) \succeq \phi (G) \succeq \phi (G_2)$.

Let $K$ be the kernel of the composition $E \twoheadrightarrow G \twoheadrightarrow G_2$, so that we have the $\Bl$-short exact sequence
\[
0 \to K \to E \to G_2 \to 0
\]
Note that $K$ and $G_2$ both lie in $\Ac_{-1/2}$, and so the injection $K \hookrightarrow E$ is a strict injection in $\Ac_{-1/2}$.

If $\phi (K) \preceq \phi (E)$ holds, we would have $\phi (K) \preceq \phi (E) \preceq\phi (G_2) \preceq \phi (G)$, and so $\phi (E) \preceq \phi (G)$.  If $\oZw (G) = 0$, then $\phi (G) = \tfrac{1}{2}$ and we have $\phi (F) \preceq \phi (G)$.  If, on the other hand, $\oZw (G) \neq 0$, then since we also have $\oZw (F) \neq 0$ (which follows from  $F \in \Ac_{-1/2}$), the inequality  $\phi (E) \preceq \phi (G)$ implies $\phi (F) \preceq \phi (G)$.

Overall, in order to check that $E$ is $\nu^l$-semistable, it suffices to check strict monomorphisms in $\Ac_{-1/2}$.
\end{proof}

\begin{theorem}\label{theorem3}
The Harder-Narasimhan property holds for $(\oZw, \Bl)$.  That is, every object $E \in \Bl$ admits a filtration in $\Bl$
\[
  E_0 \subseteq E_1 \subseteq \cdots \subseteq E_n = E
\]
where each subfactor $E_{i+1}/E_i$ is $\nu^l$-semistable, and $\phi (E_i/E_{i-1}) \succ \phi (E_{i+1}/E_i)$ for each $i$.
\end{theorem}

\begin{proof}
Take any object $E \in \Bl$.  From the torsion quadruple \eqref{eq59}, we have a filtration of $E$ in $\Bl$
\[
  E_\bullet \subseteq E_1 \subseteq E_2 \subseteq E_3 = E
\]
where the subfactors $E_\bullet, E_1/E_\bullet, E_2/E_1, E_3/E_2$ lie in $\Ac_\bullet, \Ac_{1/2}, \Ac_0, \Ac_{-1/2}$, respectively.  Noting that $\Ac_{1/2}, \Ac_0, \Ac_{-1/2}$ are all closed under extension, and that every nonzero object in each of these three categories has a nonzero value under $\oZw$, the same argument as in \cite[Theorem 2.29]{Toda1} applies.  (More specifically, for $i=\tfrac{1}{2},0,-\tfrac{1}{2}$, the proof of \cite[Proposition 2.4]{StabTC} applies to $\Ac_i$,    if we work exclusively with strict monomorphisms and strict epimorphisms in $\Ac_i$ and use Lemma \ref{lemma21} and Proposition \ref{pro1}.)  This gives us the HN filtration for each of $E_1/E_\bullet, E_2/E_1, E_3/E_2$, and concatenating these filtrations gives us a filtration for $E/E_\bullet$
\[
  E_0' \subseteq E_1' \subseteq E_2' \subseteq \cdots \subseteq E_m' = E/E_\bullet
\]
where $\phi (E_0') \succ \phi (E_1'/E_0') \succ \cdots \succ \phi (E_m'/E_{m-1}')$ and each $E_{i}'/E_{i-1}'$ (for $i \geq 1$) is $\nu^l$-semistable.

Lastly, setting $E_i$ to be the preimage of $E_i'$ in $E$ for each $i$ and noting that $\phi (E_i/E_{i-1})=\phi (E_i'/E_{i-1}')$ for all $i$ (since $\oZw (E_\bullet)=0$), we obtain the following filtration of $E$
\begin{equation}\label{eq:theorem3-eq1}
  E_0 \subseteq E_1 \subseteq E_2 \subseteq \cdots \subseteq E_m = E
\end{equation}
where $\phi (E_0) \succ \phi (E_1/E_0) \succ \cdots \succ \phi (E_m/E_{m-1})$ and each $E_i/E_{i-1} \cong E_i'/E_{i-1}'$ (for $i \geq 1$) is $\nu^l$-semistable.
\begin{itemize}
\item If $\phi (E_0') = \phi (E_0/E_\bullet) = \tfrac{1}{2}$, then $\phi (E_0)=\tfrac{1}{2}$ and $E_0$ is $\nu^l$-semistable by Remark \ref{remark17} below.  In this case, \eqref{eq:theorem3-eq1} is the HN filtration of $E$ with respect to $\nu^l$-semistability.
\item If $\tfrac{1}{2} \succ \phi (E_0')$, then we can consider the filtration of $E$ in $\Bl$
\begin{equation}\label{eq:theorem3-eq2}
 E_\bullet \subseteq E_0 \subseteq E_1 \subseteq E_2 \subseteq \cdots \subseteq E_m = E.
\end{equation}
We know $E_\bullet$ is $\nu^l$-semistable by Remark \ref{remark17}.  In this case, \eqref{eq:theorem3-eq2} is the HN filtration of $E$ with respect to $\nu^l$-semistability.
\end{itemize}
\end{proof}

\begin{remark}\label{remark17}
By Lemma \ref{lemma4}, the category
\[
  \{ E \in \Bl : \Re \oZw (E) = 0 \} = \{ E \in\Bl : \phi (E) = \tfrac{1}{2}\}
\]
is a Serre subcategory of $\Bl$.  As a result, every object in this category is $\nu^l$-semistable.
\end{remark}

\section{Tilt stability and limit tilt stability}\label{section-tiltvslimittilt}

The following lemma shows that, if $E \in D^b(X)$ is an object that is tilt (semi)stable for $s$ large enough along the hyperbola $ts = \alpha$, then $E$ is limit tilt stable.

\begin{lemma}\label{lemma-tiltimiplieslimtilt}
Fix any $\alpha >0$.  Suppose $E \in D^b(X)$ satisfies the following: there exists $s_0>0$ such that, for any $s>s_0, ts=\alpha$ and $\omega = tH+sD$, the object $E$ lies in $\Bw$ and is $\nu_\omega$-(semi)stable.  Then $E \in \Bl$ and $E$ is $\nu^l$-(semi)stable.
\end{lemma}

\begin{proof}
Suppose $E$ is as given, and suppose $\omega = tH+sD$ where $ts=\alpha$ with $t,s>0$.  By Remark \ref{remark14}(vi), we have $E \in \Bl$.  Now, take any $\Bl$-short exact sequence
\begin{equation}\label{eq62}
  0 \to M \to E \to N \to 0.
\end{equation}
Again by Remark \ref{remark14}(vi), there exists some $s_1>0$ such that, for any $s>s_1$, we have $M, E, N \in \Bw$, in which case \eqref{eq62} is a $\Bw$-short exact sequence.  Let us replace $s_0$ by $\max{\{ s_0, s_1\}}$.

Suppose $E$ is $\nu_\omega$-(semi)stable for all $s>s_0$.  Then we have $\nu_\omega (M) (\leq)< \nu_\omega (N)$ for all $s>s_0$, which implies $\phi (M) (\preceq) \prec \phi (N)$, i.e.\  $E$ is $\nu^l$-(semi)stable as wanted.
\end{proof}

\begin{remark}\label{remark18}
The structure sheaf $\OO_X$ of $X$ satisfies the hypotheses in Lemma \ref{lemma-tiltimiplieslimtilt}.  Since $\ch_1(\OO_X)=0$ and $\ch_2(\OO_X)=0$, we have $\overline{\Delta}_\omega (\OO_X)=0$ for any ample class $\omega$, where
\[
    \overline{\Delta}_\omega (E) = (\omega^2 \ch_1 (E))^2 - 2\omega^3 \ch_0(E)\cdot \omega \ch_2 (E)
\]
is the discriminant in the sense of \cite[Proposition 7.3.2]{BMT1}.  Then $\OO_X$ is $\nu_\omega$-stable for any ample class $\omega$ by \cite[Proposition 7.4.1]{BMT1}, and hence is $\nu^l$-stable.
\end{remark}

\appendix

\section{Generating torsion classes in $\Coh (X)$}

In this section, we collect the proofs of some results used in Section \ref{sec-generation}.

\begin{lemma}\label{lemma24}
We have
\[
\Coh (\pi)_{\leq 1} = \left\langle \Coh^{\leq 1}(X), \vcenter{\vbox{
\xymatrix @-2.3pc{
 \scalea{\gyoung(;;+;*,;;+;*)} \\
 \scalea{\gyoung(;;+;*,;;0;*)} \\
 \scalea{\gyoung(;;+;*,;;-;*)}
}
}} \right\rangle.
\]
\end{lemma}

\begin{proof}
Take any $E \in \Coh (\pi)_{\leq 1}$.  To show that $E$ lies in the above extension closure, it suffices to assume $E$ is pure 2-dimensional.  The $\Phi$-WIT$_0$ part of $E$, if nonzero, must lie in $\scalea{\gyoung(;;+;*,;;+;*)}$ by Remark \ref{remark10}.  Thus we can further assume $E$ is $\Phi$-WIT$_1$ and pure 2-dimensional.  If $\ch_{11}(E)<0$, then $E$ lies in $\scalea{\gyoung(;;+;*,;;-;*)}$.  So let us assume that $\ch_{11}(E)=0$.  Then $\wh{E}$ is a $\Phi$-WIT$_0$ 1-dimensional sheaf  \cite[Lemma 5.10]{LZ2}.  We can now filter  $\wh{E}$ as follows: let $A'$ be the maximal subsheaf of $\wh{E}$ in the torsion class $\{\Coh^{\leq 0}\}^\uparrow$ in $\Coh (X)$; this gives  a short exact sequence of sheaves
\[
0 \to A' \to \wh{E} \to A'' \to 0
\]
where $A''$ is a $\Phi$-WIT$_0$ fiber sheaf \cite[Lemma 3.15]{Lo11} and $A' \in \{\Coh^{\leq 0}\}^\uparrow \cap \Coh^{\leq 1}(X)$.  Let $A'_0$ be the maximal 0-dimensional subsheaf of $A'$.  Then we have the filtration in $\Coh (X)$
\[
 A'_0 \subseteq A' \subseteq \wh{E}
\]
where all subfactors are $\Phi$-WIT$_0$ sheaves.  Note that if $A'/A'_0 \neq 0$ then $\ch_{11}(A'/A'_0) >0$: if $A'/A'_0 \neq 0$ and  $\ch_{11}(A'/A'_0)=0$ then $A'/A'_0$ would be a fiber sheaf \cite[Lemma 5.11]{LZ2}, and so must be a 0-dimensional sheaf, contradicting the maximality of $A'_0$.  Thus if $A'/A'_0 \neq 0$, we have $\ch_{01}(\Phi (A'/A'_0)) \neq 0$.  Now we have
\begin{itemize}
\item $\Phi (A'_0) \in \scalea{\gyoung(;;;+,;;;0)}$,
\item $\Phi (A'/A'_0) \in \scalea{\gyoung(;;+;*,;;0;*)}$,
\item $\Phi (\wh{E}/A') \in \Coh (\pi)_0$,
\end{itemize}
and so $E$ lies in the described extension closure, proving  the lemma.
\end{proof}

\begin{lemma}\label{lemma25}
We have
\[
  \Coh^{\leq 2}(X) = \left\langle \Coh (\pi)_{\leq 1}(X), \scalea{
  \gyoung(;;*;*,;+;*;*)} \right\rangle.
\]
\end{lemma}

\begin{proof}
  Take any $E \in \Coh^{\leq 2}(X)$.  Consider the $(W_{0,X}, W_{1,X})$-decomposition of $E$ in $\Coh (X)$:
  \[
  0 \to E_0 \to E \to E_1 \to 0.
  \]
  We have $E_1 \in \Coh (\pi)_{\leq 1}$ by \cite[Lemma 2.6]{Lo7}.  Since $E_0$ is a torsion sheaf, we have $\ch_{10}(E_0) \geq 0$.  Now, either  $\ch_{10}(E_0)=0$ and $E_0 \in \Coh (\pi)_{\leq 1}$ by \cite[Proposition 5.13(4)]{LZ2}, or $\ch_{10}(E_0)>0$ and $E_0 \in \scalea{\gyoung(;;*;*,;+;*;*)}$.  The lemma thus follows.
\end{proof}

\begin{lemma}\label{lemma26}
We have
\[
  \Coh (X) = \left\langle \Coh^{\leq 2}(X), \vcenter{\vbox{
  \xymatrix @-2.3pc{
   \scalea{\gyoung(;+;*;*,;+;*;*)} \\
   \scalea{\gyoung(;+;*;*,;0;*;*)} \\
   \scalea{\gyoung(;+;*;*,;-;*;*)}
   }
   }} \right\rangle.
\]
\end{lemma}

\begin{proof}
Take any torsion-free sheaf $E$ on $X$, and consider its $(W_{0,X},W_{1,X})$-decomposition in $\Coh (X)$:
\[
0 \to E_0 \to E \to E_1 \to 0.
\]
If $E_0 \neq 0$ then it lies in $\scalea{\gyoung(;+;*;*,;+;*;*)}$.  On the other hand, $E_1$ either lies in $\Coh^{\leq 2}(X)$, or is supported in dimension 3 and lies in either $\scalea{\gyoung(;+;*;*,;0;*;*)}$ or $\scalea{\gyoung(;+;*;*,;-;*;*)}$.
\end{proof}

\begin{lemma}\label{lemma27}
Every category in the nested sequence \eqref{eq46} is a torsion class in $\Coh (X)$.
\end{lemma}

\begin{proof}
By \cite[Lemma 1.1.3]{Pol}, if $\mathcal C$ is an extension-closed  subcategory of $\Coh (X)$, then it is a torsion class if and only if it is closed under quotient in $\Coh (X)$.We prove the lemma step by step:

(i) To begin with, $\scalea{\gyoung(;;;,;;;+)
}=\Coh^{\leq 0}(X)$ is clearly closed under quotient in $\Coh (X)$, and so is a torsion class in $\Coh (X)$.  To see that $\scalebox{0.5}{
\xymatrix @-2.5pc{
\ysize\gyoung(;;;,;;;+) & \ysize\gyoung(;;;+,;;;+)
}
}$ (resp.\ $\scalebox{0.5}{
\xymatrix @-2.5pc{
\ysize\gyoung(;;;,;;;+) & \ysize\gyoung(;;;+,;;;+) \\
& \ysize\gyoung(;;;+,;;;0)
}
}$) is closed under quotient in $\Coh (X)$, observe that   it can be written as the extension closure of all fiber sheaves where all the $\mu$-HN factors have $\mu >0$ (resp.\ $\mu \geq 0$) if we set $\mu = \ch_3/D\ch_2$.  Next, $\scalebox{0.5}{
\xymatrix @-2.5pc{
\ysize\gyoung(;;;,;;;+) & \ysize\gyoung(;;;+,;;;+) \\
& \ysize\gyoung(;;;+,;;;0) \\
& \ysize\gyoung(;;;+,;;;-)
}
}=\Coh (\pi)_0$ is the category of all fiber sheaves, and so is closed under quotient in $\Coh (X)$.  Also, $\scalebox{0.5}{
\xymatrix @-2.5pc{
\ysize\gyoung(;;;,;;;+) & \ysize\gyoung(;;;+,;;;+) & \ysize\gyoung(;;;*,;;+;*)\\
& \ysize\gyoung(;;;+,;;;0) \\
& \ysize\gyoung(;;;+,;;;-)
}
}=\Coh^{\leq 1}(X)$ from Section \ref{sec-generation}, and so is closed under quotient in $\Coh (X)$.

(ii) That $\scalebox{0.5}{
\xymatrix @-2.5pc{
\ysize\gyoung(;;;,;;;+) & \ysize\gyoung(;;;+,;;;+) & \ysize\gyoung(;;;*,;;+;*) & \ysize\gyoung(;;+;*,;;+;*) \\
& \ysize\gyoung(;;;+,;;;0) & &  \\
& \ysize\gyoung(;;;+,;;;-) & &
}
}=\langle \Coh^{\leq 1}(X), \scalea{
\gyoung(;;+;*,;;+;*)
}\rangle$ is closed under quotient in $\Coh (X)$ is also easy to see: any $\Coh(X)$-quotient of an object in $\scalea{\gyoung(;;+;*,;;+;*)}$ either lies in $\Coh^2(\pi)_1$ (in which case it lies in $\scalea{\gyoung(;;+;*,;;+;*)}$) or lies in $\Coh^{\leq 1}(X)$.

(iii) To see that $\scalebox{0.5}{
\xymatrix @-2.5pc{
\ysize\gyoung(;;;,;;;+) & \ysize\gyoung(;;;+,;;;+) & \ysize\gyoung(;;;*,;;+;*) & \ysize\gyoung(;;+;*,;;+;*) \\
& \ysize\gyoung(;;;+,;;;0) & & \ysize\gyoung(;;+;*,;;0;*) \\
& \ysize\gyoung(;;;+,;;;-) & &
}
}$ is closed under quotient in $\Coh (X)$, take any object $E \in \scalea{
\gyoung(;;+;*,;;0;*)}$ and consider any $\Coh (X)$-quotient $E \twoheadrightarrow A$.  Let $K$ be the kernel of this surjection, and the $\Coh (X)$-short exact sequence yields the exact sequence of sheaves
\[
0 \to \Phi^0A \to \wh{K} \to \wh{E} \to \Phi^1A \to 0.
\]
Since $\wh{E}$ lies in $\scalea{\gyoung(;;;*,;;+;*)}$, its quotient sheaf $\Phi^1 A$  lies in either $\scalea{\gyoung(;;;*,;;+;*)}$ or $\Coh (\pi)_0$.  Hence $\wh{\Phi^1 A}$  lies in either $\scalea{\gyoung(;;+;*,;;0;*)}$ or $\Coh (\pi)_0$.  On the other hand,  $\wh{\Phi^0A}$ is a $\Phi$-WIT$_0$ sheaf lying in $\Coh (\pi)_{\leq 1}$, and so $\wh{\Phi^0A} \in \langle \Coh^{\leq 1}(X), \scalea{\gyoung(;;+;*,;;+;*)}\rangle$.  Since $A$ is the extension of $\wh{\Phi^1A}$ by $\wh{\Phi^0A}$, we have $A \in \scalebox{0.5}{
\xymatrix @-2.5pc{
\ysize\gyoung(;;;,;;;+) & \ysize\gyoung(;;;+,;;;+) & \ysize\gyoung(;;;*,;;+;*) & \ysize\gyoung(;;+;*,;;+;*) \\
& \ysize\gyoung(;;;+,;;;0) & & \ysize\gyoung(;;+;*,;;0;*) \\
& \ysize\gyoung(;;;+,;;;-) & &
}
}$.

(iv)  $\scalebox{0.5}{
\xymatrix @-2.5pc{
\ysize\gyoung(;;;,;;;+) & \ysize\gyoung(;;;+,;;;+) & \ysize\gyoung(;;;*,;;+;*) & \ysize\gyoung(;;+;*,;;+;*) \\
& \ysize\gyoung(;;;+,;;;0) & & \ysize\gyoung(;;+;*,;;0;*) \\
& \ysize\gyoung(;;;+,;;;-) & & \ysize\gyoung(;;+;*,;;-;*)
}
}=\Coh (\pi)_{\leq 1}$ by Lemma \ref{lemma24}, which is  closed under quotient in $\Coh (X)$.

(v) $\scalebox{0.5}{
\xymatrix @-2.5pc{
\ysize\gyoung(;;;,;;;+) & \ysize\gyoung(;;;+,;;;+) & \ysize\gyoung(;;;*,;;+;*) & \ysize\gyoung(;;+;*,;;+;*) &   \ysize\gyoung(;;*;*,;+;*;*)\\
& \ysize\gyoung(;;;+,;;;0) & & \ysize\gyoung(;;+;*,;;0;*) & \\
& \ysize\gyoung(;;;+,;;;-) & & \ysize\gyoung(;;+;*,;;-;*) &
}
}=\Coh^{\leq 2}(X)$ from Section \ref{sec-generation}, and so is closed under quotient in $\Coh (X)$.

(vi) The proof that $\scalebox{0.5}{
\xymatrix @-2.5pc{
\ysize\gyoung(;;;,;;;+) & \ysize\gyoung(;;;+,;;;+) & \ysize\gyoung(;;;*,;;+;*) & \ysize\gyoung(;;+;*,;;+;*) & \ysize\gyoung(;;*;*,;+;*;*) & \ysize\gyoung(;+;*;*,;+;*;*)\\
& \ysize\gyoung(;;;+,;;;0) & & \ysize\gyoung(;;+;*,;;0;*) & & \\
& \ysize\gyoung(;;;+,;;;-) & & \ysize\gyoung(;;+;*,;;-;*) & &
}
}=\langle \Coh^{\leq 2}(X), \scalea{\gyoung(+;*;*,;+;*;*)}\rangle$ is closed under quotient in $\Coh (X)$ is similar to that in step (ii).

(vii) The proof that $\scalebox{0.5}{
\xymatrix @-2.5pc{
\ysize\gyoung(;;;,;;;+) & \ysize\gyoung(;;;+,;;;+) & \ysize\gyoung(;;;*,;;+;*) & \ysize\gyoung(;;+;*,;;+;*) & \ysize\gyoung(;;*;*,;+;*;*) & \ysize\gyoung(;+;*;*,;+;*;*)\\
& \ysize\gyoung(;;;+,;;;0) & & \ysize\gyoung(;;+;*,;;0;*) & & \ysize\gyoung(;+;*;*,;0;*;*)\\
& \ysize\gyoung(;;;+,;;;-) & & \ysize\gyoung(;;+;*,;;-;*) & &
}
}$ is closed under quotient in $\Coh (X)$ is similar to that in step (iii).
\end{proof}

\begin{lemma}\label{lemma28}
Every category in the following nested sequence is a torsion class in $\Coh (X)$:
\begin{equation}\label{eq48}
\scalebox{0.5}{\xymatrix @-2.5pc{
\ysize{\gyoung(;;;,;;;+)} & \ysize{\gyoung(;;;+,;;;+)} & \ysize{\gyoung(;;;*,;;+;*)} & \ysize{\gyoung(;;+;*,;;+;*)}  \\
& \ysize{\gyoung(;;;+,;;;0)} & &  \\
& \ysize{\gyoung(;;;+,;;;-)} & &
}
}
\subset
\scalebox{0.5}{\xymatrix @-2.5pc{
\ysize{\gyoung(;;;,;;;+)} & \ysize{\gyoung(;;;+,;;;+)} & \ysize{\gyoung(;;;*,;;+;*)} & \ysize{\gyoung(;;+;*,;;+;*)} & \ysize{\gyoung(;;*;*,;+;*;*)}  \\
& \ysize{\gyoung(;;;+,;;;0)} & &  &  \\
& \ysize{\gyoung(;;;+,;;;-)} & & &
}
}
\subset
\scalebox{0.5}{\xymatrix @-2.5pc{
\ysize{\gyoung(;;;,;;;+)} & \ysize{\gyoung(;;;+,;;;+)} & \ysize{\gyoung(;;;*,;;+;*)} & \ysize{\gyoung(;;+;*,;;+;*)} & \ysize{\gyoung(;;*;*,;+;*;*)} \\
& \ysize{\gyoung(;;;+,;;;0)} & & \ysize{\gyoung(;;+;*,;;0;*)}  \\
& \ysize{\gyoung(;;;+,;;;-)} & &
}
}
\subset
\scalebox{0.5}{\xymatrix @-2.5pc{
\ysize{\gyoung(;;;,;;;+)} & \ysize{\gyoung(;;;+,;;;+)} & \ysize{\gyoung(;;;*,;;+;*)} & \ysize{\gyoung(;;+;*,;;+;*)} & \ysize{\gyoung(;;*;*,;+;*;*)} &  \ysize{\gyoung(;+;*;*,;+;*;*)}  \\
& \ysize{\gyoung(;;;+,;;;0)} & & \ysize{\gyoung(;;+;*,;;0;*)} & & \\
& \ysize{\gyoung(;;;+,;;;-)} & & & &
}
}.
\end{equation}
\end{lemma}

\begin{proof}
We already know that the first term in \eqref{eq48} is a torsion class in $\Coh (X)$ from  Lemma \ref{lemma27}.  The proof that the second term in \eqref{eq48} is a torsion class in $\Coh (X)$ follows an argument similar to that in step (ii) in the proof of Lemma \ref{lemma27}.  That the third  term in \eqref{eq48} is also a torsion class in $\Coh (X)$ follows from $\scalebox{0.5}{\xymatrix @-2.5pc{
\ysize{\gyoung(;;;,;;;+)} & \ysize{\gyoung(;;;+,;;;+)} & \ysize{\gyoung(;;;*,;;+;*)} & \ysize{\gyoung(;;+;*,;;+;*)}  \\
& \ysize{\gyoung(;;;+,;;;0)} & & \ysize{\gyoung(;;+;*,;;0;*)}  \\
& \ysize{\gyoung(;;;+,;;;-)} & &
}
}$ being a torsion class (Lemma \ref{lemma27}), and that any sheaf quotient of an object in $\scalea{\gyoung(;;*;*,;+;*;*)}$ again lies in $\scalebox{0.5}{\xymatrix @-2.5pc{
\ysize{\gyoung(;;;,;;;+)} & \ysize{\gyoung(;;;+,;;;+)} & \ysize{\gyoung(;;;*,;;+;*)} & \ysize{\gyoung(;;+;*,;;+;*)} & \ysize{\gyoung(;;*;*,;+;*;*)}  \\
& \ysize{\gyoung(;;;+,;;;0)} & &  &  \\
& \ysize{\gyoung(;;;+,;;;-)} & & &
}
}$ (from the previous sentence).

To see that the last term in \eqref{eq48} is a torsion class, again take any surjection of sheaves $A \twoheadrightarrow A'$ with $A \in \scalea{\gyoung(;+;*;*,;+;*;*)}$.  If $\rk (A')>0$, then $A' \in \scalea{\gyoung(;+;*;*,;+;*;*)}$; if $\rk (A')=0$, then $A'$ is still $\Phi$-WIT$_0$, and so lies in the third term in \eqref{eq48}. Thus the last term in \eqref{eq48} is also a torsion class in $\Coh (X)$.
\end{proof}

\begin{remark}\label{remark8}
The last part of the proof of Lemma \ref{lemma28} shows $W_{0,X} \subset \mathcal A_{1/2,0}$.  Since $\Fl [1]$ is entirely contained in $\mathcal A_{1/2,0}$, any object $E$ in the torsion-free class $\mathcal A_{1/2,0}^\circ$ (taken in $\Bl$) satisfies $H^{-1}(E)=0$ and $H^0(E) \in W_{1,X}$.  Hence
\[
  \mathcal A_{1/2,0}^\circ \subset W_{1,X} \cap \Tl
\]
where the category $W_{1,X} \cap \Tl$ is described in Lemma \ref{lemma7}.
\end{remark}

\bibliography{refs}{}
\bibliographystyle{plain}

\end{document}